\documentclass[reqno]{amsart}
\usepackage[utf8]{inputenc}
\usepackage[usenames,dvipsnames,svgnames,table]{xcolor}
\usepackage{amsmath, amssymb, amsthm, verbatim, hyperref, datetime}
\usepackage{mathtools,bm,bbm,framed,graphicx,subcaption,dsfont,booktabs}
\usepackage{enumitem}
\usepackage{stmaryrd}
\usepackage{mathrsfs}
\usepackage{comment}
\usepackage{ytableau}
\usepackage[margin=1in]{geometry}
\numberwithin{equation}{section}

\usepackage{etoolbox}

\usepackage{amsthm}

\usepackage{tikz-cd} 
\usepackage{rotating}
\usepackage{tablefootnote}
\usepackage[table]{xcolor}
\usepackage{ytableau}
\ytableausetup{smalltableaux}
\usepackage{multirow}
\usepackage[utf8]{inputenc}
\usepackage{arydshln}
\usepackage{ wasysym }
\usepackage{shuffle}
\usepackage{ stmaryrd }
\usepackage{tikz} 
\usetikzlibrary{arrows,shapes}
\usepackage{mathtools}
\usepackage{ bbold }
\usepackage{nicematrix}
\usepackage{ctable}
\usepackage{enumitem}
\usepackage{xcolor,shuffle}
\usepackage[utf8]{inputenc}

\newcommand{\rad}{\mathop{\mathrm{rad}}\nolimits}
\newcommand{\im}{\mathop{\mathrm{im}}}

\newcommand{\catzero}{\mathrm{CAT}(0)}

\renewcommand{\hom}{\mathop{\mathrm{Hom}}\nolimits}
\newcommand{\rk}{\mathop{\mathrm{rk}}}
\newcommand{\Char}{\mathop{\mathrm{char}}}

\newcommand{\Z}{{\mathbb{Z}}}

\newcommand{\kk}{\mathbf{k}}

\newcommand{\stab}{\mathrm{Stab}}
\newcommand{\sgn}{\mathrm{sgn}}

\newcommand{\supp}{\Lambda}

\newcommand{\pg}{\mathrm{PG}}
\newcommand{\ag}{\mathrm{AG}}

\newcommand{\GX}{G_X}
\newcommand{\GY}{G_Y}

\newcommand{\End}{\mathop{\mathrm{End}}\nolimits}

\newcommand{\freelrb}{\mathfrak{F}}

\DeclareMathOperator*{\bigboxtimes}{\bigotimes}

\newcommand{ \free}{\freelrb}

\usepackage[capitalize]{cleveref} 

 \newtheorem{thm}{Theorem}[section]
 \newtheorem{lemma}[thm]{Lemma}
 \newtheorem{prop}[thm]{Proposition}
 \newtheorem{cor}[thm]{Corollary}

\newtheorem{convention}[thm]{Convention}

 \usepackage{comment}

 \newtheorem{example}[thm]{Example}

 \newtheorem{definition}[thm]{Definition}

 \newtheorem{remark}[thm]{Remark}

\newtheorem{maintheorem}{Theorem}

\newtheorem{maincorollary}{Corollary}

\newtheorem*{mainproposition*}{Proposition}

\newcommand{\qbinom}[2]{\genfrac{[}{]}{0pt}{}{#1}{#2}}

\usepackage{xargs}
\usepackage{todonotes}

\title{Left regular bands with symmetry}
\author{Patricia Commins and Benjamin Steinberg}
\begin{document}

\maketitle
\begin{abstract} The representation theory of \textit{left regular band semigroup algebras} is well-studied and known to have close connections with combinatorial topology, as established in the work of Margolis--Saliola--Steinberg ('15, '21).
In this paper, we investigate the representation theory of the \textit{invariant subalgebras} of left regular band semigroup algebras carrying the action of a finite group through the lens of group-equivariant combinatorial topology.

We characterize when the invariant subalgebra is semisimple or commutative and examine the equivariant structure of the \textit{Peirce components} of the semigroup algebra. For \textit{CW left regular bands}, we interpret these Peirce components in terms of the equivariant topology of intervals in the support semilattice, yielding the Cartan invariants of the invariant subalgebras of left regular bands associated to $\catzero$-cube complexes. We also give a topological formula for the Peirce components for left regular bands with hereditary algebras. Finally, in specializing to left regular bands associated to \textit{geometric lattices}, we explore generalizations of the Desarm\'{e}ni\'{e}n--Wachs \textit{derangement representation} and their connections to Markov chains.
\end{abstract}

\tableofcontents
\section{Introduction}
\label{intro_chapter}

\subsection{Historical context}
\subsubsection{LRBs and their semigroup algebras}
A finite semigroup $B$ is an \textbf{LRB} if it satisfies (i) $x^2 = x$ and (ii) $xyx = xy$ for all $x, y \in B.$ LRBs can be (roughly) thought of as nonommutative generalizations of semilattices, and they were initially studied from this perspective in the 1940s by Sch\"{u}tzenberger \cite{Schutzernberger}. They have since been studied for their {probabilistic} applications, their frequent appearance throughout {combinatorics}, and their rich yet tractable {representation theory}. 
 
 LRBs are perhaps best known for their  connections to the theory of Markov chains.  In  seminal work in the 1990s \cite{BHR}, Bidigare--Hanlon--Rockmore united several well-studied---yet seemingly unrelated--- Markov chains as random walks on an LRB formed from the faces of a hyperplane arrangement. Using the representation theory of the LRB, they found \textit{uniform} formulas for the eigenvalues of these random walks. Soon after, Brown--Diaconis \cite{browndiac} strengthened their results, showing these operators are diagonalizable. In a vast generalization, Brown \cite{BrownonLRBs} lifted these results to all LRBs, providing many more examples of LRBs and associated Markov chains along the way.

LRBs seem to arise surprisingly often in combinatorics and discrete geometry. For example, in addition to semilattices, there are LRBs associated to graphs, $\catzero$-cube complexes, certain $\catzero$-zonotopal complexes, real (central or affine) and complex hyperplane arrangements, (oriented or classical) matroids, and (oriented or classical) interval greedoids.

Most recently---and most relevantly for this paper---LRBs have been studied for the  representation theory of their {semigroup algebras}. Although their semigroup algebras are seldom semsimple, many researchers, including Aguiar, Mahajan, Margolis, Saliola, and Steinberg \cite{Aguiar-Mahajan, MSS, margolis2015combinatorial, saliolafacealgebra, saliolaquiverdescalgebra, saliolaquiverlrb}, have had great success in finding explicit descriptions of representation theoretic data which are usually too hard to hope to find for a generic finite dimensional algebra. For example, these authors have found explicit constructions of complete families of primitive orthogonal idempotents for the semigroup algebras, formulas for their Cartan invariants, descriptions of their associated Ext-quivers, projective resolutions of their simple modules, the structure of the Ext-groups between their simples, and more. Similar studies for more general classes of semigroups have been performed subsequently by a number of authors, e.g., \cite{mobius2, rrbg, Jtrivialpaper,   ASST, DO, stuartmargolist, Stein22, SteinbergMonoidTopology, steinberg2024topologymonoidrepresentationsii}.

One tool which has made the representation theory for LRBs so tractable is \textit{combinatorial topology}, or more specifically, the topology of posets and regular CW complexes. The use of combinatorial topology to understand LRBs dates back to work of Brown--Diaconis \cite{browndiac}, in which the cellular chain complexes of the \textit{zonotope} associated to a hyperplane arrangement played a critical role in proving the diagonalizibility of the random walks on the associated semigroups \cite{browndiac}. Later, Saliola proved these chain complexes (and certain truncations) give minimal projective resolutions of the simple modules of the face semigroup algebras of hyperplane arrangements \cite{saliolafacealgebra}, which in turn allowed him to compute the Ext-groups between their simple modules. More recently, \cite{MSS, margolis2015combinatorial}, Margolis--Saliola--Steinberg generalized this work to all LRBs. They improved on and vastly generalized the techniques for hyperplane arrangement LRBs to all LRBs with associated regular CW complexes. Extensions to more general LRBs required new techniques, involving the topology of posets associated to LRBs.

\subsubsection{LRBs with symmetry} 
Many of the LRBs in the literature have groups naturally associated to them. For example, reflection groups act on the face semigroups of reflection arrangements, and there is a well-studied action of the symmetric group on the \textit{space of phylogenetic trees},  a famous $\catzero$-cube complex.  

Additionally, there are several combinatorial algebras which appear as \textit{invariant subalgebras} of LRB algebras or closely related semigroup algebras. Most famously, Bidigare  \cite{Bidigare} proved the invariant subalgebra of the face semigroup algebra associated to a reflection arrangement is the well-loved \textit{Solomon's descent algebra}, a remarkable subalgebra of a reflection group algebra. Later, in \cite{hsiao2009semigroup}, Hsiao generalized Bidigare's work to show that the \textit{Mantaci--Reutenauer algebra} (a wreath product generalization of Solomon's descent algebra introduced in \cite{mantaci1995generalization}) can also be viewed as the invariant subalgebra of a semigroup algebra closely related to LRBs. Most recently, Brauner--Commins--Reiner \cite{BraunerComminsReiner} proved that the symmetric group invariant subalgebra of the \textit{free LRB} algebra is a commutative subalgebra of the symmetric group algebra generated by the ``random-to-top'' shuffling element, an algebra also implicity studied by Garsia \cite{Garsia} and Tian \cite{Tian}.  

We claim that studying these combinatorial algebras within the context of their ambient semigroup algebras is a useful perspective that helps one better understand their representation theory and brings to light interesting combinatorial structures. For example, work to understand the representation theory of the descent algebra was already well-underway by the time of Bidigare's discovery \cite{Atkinson, Bergerons-hyper2, Bergerons-general, Bergerons-hyper1, garsia-reutenauer}, but the semigroup perspective of the descent algebra still turned out to be quite useful in shedding new light on its structure. For example, Saliola used this perspective to compute the quiver of the descent algebra in types A and B in \cite{saliolaquiverdescalgebra}, and its Loewy length in type D in \cite{saliolaloewytypeD}. \normalcolor

Beyond the \textit{internal} structure of the invariant subalgebra of an LRB semigroup algebra, one can also study some of its interesting \textit{external} modules arising from classical invariant theory. Specifically, in invariant theory, one often examines the structure of the {ambient algebra} as a \textit{module over its invariant subalgebra}.  Schocker started off such a study for the type $A$ descent algebra within the face semigroup algebra of the braid arrangement by examining the projective indecomposable and simple modules of the face algebra as descent algebra modules in \cite[Prop 9.4, Prop 9.6, Thm 9.7]{Schocker}. 

 More finely, say the ambient algebra $A$ is over a field whose characteristic does not divide the order of the acting group $G$. Then, $A$ decomposes into a direct sum of its \textit{$G$-isotypic subspaces}, $A^\chi$, as $\chi$ varies over the irreducible characters of $G$: $A = \bigoplus_{\chi}A^\chi.$ The isotypic subspace for the trivial character is the invariant subalgebra $A^G$ of $A$. More generally, each isotypic subspace $A^\chi$ is a \textit{module} over $A^G.$ Studying the $A^G$-module structure of each isotypic subspace $A^\chi$ examines $A$ as a module over both $G$ and $A^G$ \textit{simultaneously}.\normalcolor

In \cite{BraunerComminsReiner}, Brauner--Commins--Reiner analyzed this more refined question for the semigroup algebra of the \textit{free LRB}. Interestingly, the answers centered around a well-studied symmetric group representation called the \textit{derangement representation}, first introduced by Desarm\'{e}nien--Wachs \cite{FrenchDesarmenienWachs}.   
More recently, Commins examined this question for the semigroup algebra of the LRB associated to the \textit{braid arrangement} \cite{commins2024invariant}. The story became more elaborate, but the answer still centered around a class of well-studied symmetric group representions: this time, the \textit{higher Lie representations}.

\subsection{Summary}
In this paper, we study the general theory of groups acting on LRBs. In particular, let $B$ be an LRB, let $G$ be a finite group acting by semigroup automorphisms on $B$, set $\kk$ to be a field with characteristic not dividing $|G|$, and assume that the semigroup ring $\kk B$ is \textit{unital} (i.e., has a two-sided multiplicative identity). We will study the \textit{representation theory} of the invariant subalgebras $(\kk B)^G$ and the \textit{invariant theory} of the semigroup algebra $\kk B$. Inspired by the role that poset topology played in understanding the representation theory of LRBs, as well as the huge volume of literature in algebraic combinatorics studying \textit{group-equivariant poset topology} \cite{SundaramJerusalem, SundaramAdvances, SundaramNonModular, stanley_grps_on_postes, WachsPosetTopology, WachsWhitneyHomology, GottliebWachs, Hanlon}, we will carry out our study through the lens of group-equivariant poset topology.

\subsubsection{Goals}
Given the conditions of $B, G$, and $\kk$ above, we aim to answer the following questions:

\begin{enumerate}
    \item[(1)] When is the invariant subalgebra $(\kk B)^G$ especially ``nice'': i.e., semisimple, commutative?
    \item[(2)] When $(\kk B)^G$ is not semisimple, what are its Cartan invariants?
    \item[(3)] What is the structure of $\kk B$ as a module simultaneously over $G$ and $(\kk B)^G$?
\end{enumerate}

\subsubsection{Summary of answers to question (1)}

Our main answer to question (1) is \cref{intro:thmA} below, which has a quite straightforward proof. Our explanation is expressed in terms of the $G$-orbits of a certain poset associated to $B$ called the \textit{support semilattice} $\supp(B)$.

\begin{maintheorem}[\cref{cor:when-is-invariant-semisimple}]\label{intro:thmA}
 Assume that $\kk$ is a field whose characteristic does not divide $|G|$. The following are equivalent:
\begin{itemize}
    \item The invariant subalgebra $(\kk B)^G$ is semisimple.
    \item The number of $G$-orbits $|\supp(B) / G| = |B / G|.$
    \item  The invariant subalgebra $(\kk B)^G$ is commutative.
\end{itemize}    
\end{maintheorem}

\cref{intro:thmA} immediately recovers the known results that Solomon's descent algebra is noncommutative and nonsemisimple in the majority of Coxeter types, while the invariant subalgebra of the free LRB is semisimple and commutative \cite[Theorem 2.9]{BraunerComminsReiner}. 

In addition to \cref{intro:thmA}, we give a necessary and sufficient condition for a $G$-orbit-sum of an element of $B$ to generate $(\kk B)^G$ in \cref{prop:diag}, which generalizes \cite[Theorem 1.3]{BraunerComminsReiner}. 

\subsubsection{Summary of answers to question (2) and (3)}

Questions (2) and (3) are more complex. Answering them reduces to understanding the $G$-representation structure of certain spaces which we call \textit{invariant Peirce components} of $\kk B$. Invariant Peirce components  are of the form $E_{[Y]} \cdot \kk B \cdot E_{[X]}$, where $E_{[X]}$ and $E_{[Y]}$ are special idempotents in the invariant subalgebra $(\kk B)^G$ indexed by $G$-orbits $[X], [Y]$ of the support semilattice $\supp(B)$. The following proposition explains this reduction.

The simple modules $\{M_{[X]}\}$ and the projective indecomposable modules $\{P_{[X]}\}$ of $(\kk B)^G$ are also indexed by $G$-orbits $[X] \in \supp(B) / G$. For a $(\kk B)^G$-module $U$, we write $[U:M_{[X]}]$ to denote the \textit{composition multiplicity of} the simple module $M_{[X]}$ in $U.$ For a $G$-representation $V$ and an irreducible $G$-character $\chi$, write $V^\chi$ for its $\chi$-isotypic component and let $\left\langle \chi, V\right\rangle_G$ be the multiplicity of $\chi$ in $V.$

\begin{mainproposition*}[Combines \cref{prop:decomposing-monoid-algebra} and \cref{conversion-of-algebra-rep-to-g-rep}]
    Assume that $\kk$ is a field whose characteristic does not divide $|G|$. Then, 
    \begin{itemize}
        \item The Cartan invariants of $(\kk B)^G$ are $[P_{[X]}: M_{[Y]}] = \left\langle \mathbb{1}, E_{[Y]} \cdot \kk B \cdot E_{[X]}\right\rangle_G$.
        \item More generally, as a $(\kk B)^G$-module, the $\chi$-isotypic component $(\kk B)^\chi$ decomposes as \[(\kk B)^\chi = \displaystyle \bigoplus_{[X] \in \supp(B) / G} (\kk B)^\chi\cdot E_{[X]},\] and the multiplicity of the composition factor of $M_{[Y]}$ in each submodule $(\kk B)^\chi \cdot E_{[X]}$ is \[[(\kk B)^\chi \cdot E_{[X]}: M_{[Y]}] = \chi(1) \cdot \left\langle \chi, \ E_{[Y]} \cdot \kk B \cdot E_{[X]}\right\rangle_G. \]
    \end{itemize}
\end{mainproposition*}

Given their roles in addressing goals (2) and (3), we spend a significant portion of this paper aiming to understand the $G$-representation structure of the invariant Peirce components.

In \cite{commins2024invariant}, Commins studied the invariant Peirce components for the LRB associated to the braid arrangement. A crucial tool in her work was a known translation of its invariant Peirce components in terms of the poset cohomology of open intervals in the lattice of set partitions. One of the key contributions of this paper is to generalize this translation to all \textit{CW LRBs}, which is a large class of LRBs containing hyperplane face semigroups as well as $\catzero$-cube complex LRBs. Our generalization, written below, involves defining certain one-dimensional characters we call \textit{degree characters }$\deg(X)$, which can be computed via poset cohomology (see \cref{lem:simpler-det-stabilizer}). See \cref{sec:CWLRBs} for undefined notation.  

\begin{maintheorem}[combines \cref{cor:CWLRBs-as-poset-topology} and \cref{prop:invariant-Peirce-decomp}]\label{intro:thmC}
Assume that $\kk$ is a field whose characteristic does not divide $|G|$ and let $B$ be a CW LRB for which $\kk B$ is unital. Then, for $G$-orbits $[X], [Y]$ in $\supp(B) / G$, as $G$-representations,
    \[E_{[Y]} \cdot \kk B \cdot E_{[X]} \cong \bigoplus_{\substack{[X' \leq Y']:\\ X' \in [X],\ Y' \in [Y]}} \left[\deg(X') \otimes \widetilde{H}^{\mathrm{top}}\left((X', Y')\right)  \otimes \deg(Y')\right] \Bigg\uparrow_{G_{X'} \cap G_{Y'}}^G.\]
\end{maintheorem}

By applying \cref{intro:thmC} to $\catzero$-cube LRBs with reasonably well-behaved group actions, we obtain the interesting corollary that their invariant Peirce components are \textit{permutation representations}: 

\begin{maincorollary}[\cref{cor:Peirce-decomp-rep}]\label{cor:intro-cor-c}
    Let $G$ be a finite group acting on a finite, connected $\catzero$-cube complex $\mathcal{C}$ satisfying certain mild {conditions}\footnote{Namely, we require that $G$ acts by isometries, permutes the cubes of $\mathcal{C}$, and that if \textit{all} of $G$ fixes a face $c$ of $\mathcal{C}$ setwise, then all of $G$ must fix $c$ pointwise. We note that this final condition is weaker than the perhaps more standard condition that if $g \in G$ fixes $c$ setwise, then $g$ fixes $c$ pointwise.}. Let $\mathcal{F}(\mathcal{C})$ be the LRB on the faces of $\mathcal{C}$ and assume that $\kk$ is a field whose characteristic does not divide $|G|$. Then, for $G$-orbits $[X]$, $[Y]$ in $\supp(\mathcal{F}(\mathcal{C})) / G$, as $G$-representations
   \begin{align*}
        E_{[Y]} \cdot \kk \mathcal{F}(\mathcal{C}) \cdot E_{[X]} \cong \kk \{X \leq Y: X' \in [X], Y' \in [Y]\}.
    \end{align*}
\end{maincorollary}
\cref{cor:intro-cor-c} allows us to compute the Cartan invariants of the invariant subalgebras for LRBs arising from  $\catzero$-cube complexes \cref{cor:Cartan-invariant}.

Although CW LRBs are a large class, there are still plenty of LRBs which are not CW LRBs. For instance,  the LRBs associated to geometric lattices are never CW LRBs. However, geometric lattice LRBs do turn out to have the special property that their semigroup algebras are \textit{hereditary algebras}. We derive the following formula for the invariant Peirce components for LRBs with hereditary semigroup algebras. 

\begin{maintheorem}[\cref{thm:thmD-restated}] \label{intro:thmD}
Assume the LRB semigroup algebra $\kk B$ is a hereditary algebra and that $\kk$ is a field whose characteristic does not divide $|G|$. Then, for $[X] < [Y]$, as $G$-representations,
    \begin{align*}
    E_{[Y]} \cdot \kk B \cdot E_{[X]}\cong \bigoplus_{\substack{m \geq 1, \\ [X_0 <  \cdots < X_m]:\\ X_0 \in [X], X_m \in [Y]}}  \left({\bigboxtimes_{i = 1}^m} \ \widetilde{H}_0 \left(\widetilde{B_{\geq X_{i-1}}^{<b_{X_i}}} \right)\right) \Bigg \uparrow_{\bigcap_i G_{X_i}}^{G}.
    \end{align*}
    Note that $\widetilde{B_{\geq X_{i-1}}^{< b_{X_i}}}$ is a 
    poset depending on $X_{i -1}$ and $X_i$ that will be defined in \cref{ch:background_lrbs} and \cref{ch:invariant_peirce_general}.
\end{maintheorem}

Although the formula in \cref{intro:thmD} might look intimidating, it simplifies quite nicely for the LRBs associated to geometric lattices. This allows us to generalize the results in \cite{BraunerComminsReiner} and provides interesting connections with various generalizations of \textit{derangements}. 
\subsubsection{Connections with generalized derangement representations and generalized random-to-top shuffles}\label{subsec:intro-derangement-reps}
A permutation in the symmetric group $S_n$ is a \textbf{derangement} if it has no fixed points. The number of derangements in $S_n$, written $d_n$, is usually counted recursively (\cref{eqn:recursive}) or with M\"{o}bius inversion (\cref{eqn:Mobius}):

\begin{align}
    n! &=   \sum_{k = 0}^n {n \choose k}d_{n-k} \label{eqn:recursive},\\
    d_n &= \sum_{k = 0}^n(-1)^{k}{n \choose k}(n-k)!\label{eqn:Mobius}
\end{align}

One manifestly positive formula for counting derangements can be deduced from work of Gessel \cite[Eqn (6)]{MR1109577}, Brown \cite[inductively applying Prop.~10]{BrownonLRBs}, or Stembridge \cite[Theorem 1.1]{STEMBRIDGE1992307}.  Here, $\alpha \vDash n$ denotes an integer \textbf{composition} of $n,$ meaning a sequence $\alpha = (\alpha_1, \alpha_2, \ldots, \alpha_{\ell(\alpha)})$ of positive integers summing to $n$, and ${n \choose \alpha}$ is the \textbf{multinomial coefficient} ${n \choose \alpha_1, \alpha_2, \ldots, \alpha_{\ell(\alpha)}} =  \frac{n!}{\prod_i \alpha_i!}:$
\begin{align}\label{eqn:derangement-positive}
    d_n = \sum_{\alpha \vDash n} {n \choose \alpha} \prod_{i = 1}^{\ell(\alpha)}(\alpha_i - 1).
\end{align}
For any geometric lattice $\mathcal{L}$, Brown \cite[Appendix C]{BrownonLRBs} defined and studied a \textbf{generalized derangement number} $d_{\mathcal{L}}$. Let $S(\mathcal{L})$ be the LRB associated to the lattice $\mathcal{L}$ and note that $\supp(S(\mathcal{L}))$ can be identified with $\mathcal{L}^{\mathrm{opp}}$. Brown's definition of $d_\mathcal{L}$ can be rephrased to be in terms of ``maximal'' invariant Peirce component:
\begin{align*}
    d_\mathcal{L} &:= \dim_{\kk} E_{\hat{0}_{\mathcal{L}}} \cdot \kk S(\mathcal{L}) \cdot E_{\hat{1}_{\mathcal{L}}}.
\end{align*}
He shows these numbers satisfy formulas generalizing \cref{eqn:recursive}, \cref{eqn:Mobius} and \cref{eqn:derangement-positive}. 

On the other hand, in \cite{BraunerComminsReiner}, it is shown 
that for $\mathcal{L}$ the Boolean lattice, the ``maximal'' invariant Peirce component carries the aforementioned \textit{derangement representation} of the symmetric group. This inspires a geometric lattice generalization of the derangement representation. For a geometric lattice $\mathcal{L}$ carrying symmetries by $G$, we define $\mathrm{Der}(\mathcal{L})$ to be the $G$-representation given by the invariant Peirce component:
\[\mathrm{Der}(\mathcal{L}) := E_{\hat{0}_{\mathcal{L}}} \cdot \kk S(\mathcal{L}) \cdot E_{\hat{1}_{\mathcal{L}}}.\] Using \cref{intro:thmD}, we prove categorifications of the generalizations of \eqref{eqn:recursive}, \eqref{eqn:Mobius}, and \eqref{eqn:derangement-positive} for $d_\mathcal{L}$. 
\begin{maintheorem}[\cref{cor:derangement-rep-genera-matroids} and \cref{prop:mobius-rep}]\label{intro:thmE}
 Assume that $\kk$ is a 
 field whose characteristic does not divide $|G|$. Let $\kk \mathcal{F}(\mathcal{L})$ be the $\kk$-span of complete flags of $\mathcal{L}.$ Then,
   \begin{enumerate}
    \item (Categorification of \cref{eqn:recursive}):  As $G$-representations,
    \begin{align}\label{eqn:categorify-}
    \kk \mathcal{F}(\mathcal{L}) \cong \bigoplus_{[X] \in \mathcal{L} / G} \mathrm{Der}([X,  \hat{1}_{\mathcal{L}}])\Big \uparrow_{G_X}^{G}.
   \end{align}
    \item (Categorification of \cref{eqn:Mobius}): As virtual $G$-representations,
    \begin{align*}
        \mathrm{Der}(\mathcal{L}) \cong \bigoplus_{[Y] \in \mathcal{L} / G} (-1)^{\rk(Y) - 2} \left(\widetilde{H}_{\mathrm{rk}(Y) - 2}\left((\hat{0}_{\mathcal{L}}, Y)\right) \otimes \kk \mathcal{F}([Y, \hat{1}_{\mathcal{L}}])\right) \Big \uparrow_{G_Y}^G.
    \end{align*}
         \item (Categorification of \cref{eqn:derangement-positive}): If $\mathcal{L}$ has one element, then $\mathrm{Der}(\mathcal{L})$ carries the trivial $G$-representation. Otherwise, as $G$-representations,
       \begin{align} \label{eqn:categorify-manifestly-positive}
        \mathrm{Der}(\mathcal{L}) \cong \bigoplus_{\substack{m \geq 1, \\G\text{\rm -orbits}\\ [\hat{1}= X_0 >_{\mathcal{L}}  \cdots >_{\mathcal{L}} X_m = \hat{0}]}}  \left({\bigboxtimes_{i = 0}^{m-1}} {V_{X_i, X_{i + 1}}}\right) \Bigg \uparrow_{\bigcap_i G_{X_i}}^{G},
    \end{align}
    where $ {V_{X, Y}}$ is the honest $G_{X}\cap G_{Y}$-representation $
        {V_{X, Y}}:=\mathrm{span}_{\kk} \{Z \in \mathcal{L}: Y \lessdot_{\mathcal{L}} Z \leq_{\mathcal{L}} X\} - \mathbb{1}.$
   \end{enumerate}
 
\end{maintheorem}

It is known that the \textit{kernel} of the \textit{random-to-top} shuffling operator carries the classical derangement representation of the symmetric group and that {each} eigenspace of the random-to-top operator can be expressed as lifts of the classical derangement representation from a smaller symmetric groups. 

Part of Brown's motivation in studying generalized derangement numbers was their interpretation as the dimension of the kernel of a geometric lattice analogue of the random-to-top operator. In \cref{sub:randomtotop}, we prove the kernel of Brown's generalized random-to-top operator carries the generalized derangement representation, and that all of its eigenspaces are lifts of generalized derangement representations of sublattices.

\subsubsection{Outline of paper}

The remainder of the paper is structured as follows.

In \cref{ch:background_lrbs}, we provide background on LRBs. As examples, we explore various LRBs carrying symmetries: those arising from hyperplane arrangements, $\catzero$-cube complexes, and geometric lattices. 
We also describe the different classes of LRBs appearing in this paper: \textit{connected}, \textit{CW}, and \textit{hereditary} LRBs.

In \cref{ch:fin-dim-algebras}, we review  the representation theory of finite dimensional algebras and LRB semigroup algebras. We also discuss idempotents for LRB semigroup algebras which behave well under symmetry.

In \cref{ch:invariant_subalgebras}, we prove Theorem A. We then explain our condition for when an orbit-sum of an element of an LRB generates the invariant subalgebra.

\cref{ch:invariant_peirce_general} explores invariant Peirce components for general LRBs and provides a proof of the Proposition from the introduction. We also review background on poset topology. In the remainder of the paper, we specialize to two large classes of LRBs and use combinatorial topology to give stronger answers. 

First, in \cref{sec:CWLRBs}, we specialize to $CW$ LRBs. We explain the known case of LRBs associated to hyperplane arrangements, prove \cref{intro:thmC}, and apply \cref{intro:thmC} to $\catzero$-cube complexes in \cref{cor:intro-cor-c}.

Then, in \cref{sec:HereditaryLRBs}, we specialize to LRBs with hereditary semigroup algebras and prove \cref{intro:thmD}. In \cref{subsec:thmE}, we apply \cref{intro:thmD} to geometric lattice LRBs and prove \cref{intro:thmE} on generalized derangement representations. We explore connections to generalizations of random-to-top shuffling in \cref{sub:randomtotop}. 

\vspace{.5cm}
\noindent\textbf{Acknowledgements.} Both authors would like to thank Peter Webb for some helpful background on modular representations.  The first author would like to thank Vic Reiner for invaluable guidance throughout this project, as well as Marcelo Aguiar, Esther Banaian, Sarah Brauner, Kyle Celano, Joseph Pappe, Franco Saliola, and Sheila Sundaram for helpful conversations. She was supported by the NSF GRFP Award \# 2237827 and benefited from the NSF grant DMS-2053288 for related travel. The final portion was carried out while she was in residence at ICERM for the Categorification and Computation in Algebraic Combinatorics semester program, which was supported by the NSF grant DMS-1929284.
The second author was supported by the NSF grant DMS-2452324, a Simons Foundation Collaboration Grant, award number 849561, the Australian Research Council Grant DP230103184, and the Marsden Fund Grant MFP-VUW2411.  



\section{Background: LRBs and running examples}
\label{ch:background_lrbs}
A \textbf{semigroup} is a set equipped with a binary, associative multiplication operation. A \textbf{monoid} is a semigroup which has a (two-sided) identity element. Recall from the introduction that a finite semigroup $B$ is a \textbf{left regular band} (or an \textbf{LRB}) if it satisfies (i) $x^2 = x$ for all $x \in B$ and (ii) $xyx = xy$ for all $x, y \in B.$
\subsection{Example \#0: Meet-semilattices} 
The simplest, and the motivating, example of an LRB is a finite meet-semilattice. Further, as we will see in \cref{sec:general-lrb-rep-theory}, a fair amount of the structure of general LRBs can be lifted from this special case. Recall that a poset $P$ is a \textbf{meet-semilattice} if for any elements $p, q \in P$, there exists a unique greatest lower bound of $p$ and $q$; this unique element is called the \textbf{meet} of $p$ and $q$ and is written as $p \wedge q.$ The meet operation on on a meet-semilattice $L$ is associative and commutative. Further, for all $p, q \in L$, observe that $p \wedge p = p$ and $p \wedge q \wedge p = p \wedge q.$ Hence, every finite meet-semilattice forms a \textbf{commutative} LRB under the meet operation. In fact, it turns out that an LRB is commutative if and only if it is isomorphic to a semilattice under its meet operation (cf.~\cite[Proposition 2.2]{MSS}).

\subsection{General definitions}
\label{sec:background_lrbs_definitions}
 We begin by explaining two partially ordered sets associated to every LRB: the \textbf{support semilattice} and the \textbf{semigroup poset}. Throughout, let $B$ be an LRB.

\subsubsection{Posets associated to LRBs}
The \textbf{principal left ideal} of $B$ generated by a fixed element $b \in B$  is $Bb:= \{xb: x \in B\}.$ It turns out (\cite[Proposition 2.3]{MSS}), that $Ba\cap Bb=Bab$. Hence, the poset of principal left ideals of $B$ ordered under inclusion forms a meet-semilattice---where meet coincides with intersection---called the \textbf{support semilattice of $B$} and written $\supp(B) = (\supp(B), \leq)$. Define the \textbf{support map}  $\sigma$ by
\begin{align*}
    \sigma\colon B &\to \supp(B)\\
    b & \mapsto Bb.
\end{align*}

For $x \in B$, we call $\sigma(x)$ the \textbf{support} of $x$. We shall often use capital letters such as \textbf{$X$} and \textbf{$Y$} to denote supports. When an LRB is a monoid, the semilattice $\supp(B)$ is a lattice as the element $\sigma(1)$ forms an upper bound of all principal left ideals, and finite, bounded semilattices are always lattices \cite[Prop 3.3.1]{EC1}.

There is a straightforward way to tell if $\sigma(x) \leq \sigma(y)$ for $x, y \in B$, namely, for $x, y \in B$, 
\begin{equation}\label{lem:comparing-supports}
    \sigma(x) \leq \sigma(y)\ \text{if and only if}\ xy = x.
\end{equation}

We already observed that $\sigma(xy) = \sigma(x) \wedge \sigma(y)$; hence, viewing the semilattice $\supp(B)$ as an LRB, the map $\sigma\colon B \to \supp(B)$ is a semigroup homomorphism. In fact, $\sigma\colon B \to \supp(B)$ is the \textbf{abelianization of $B$}, meaning that any semigroup homomorphism from $B$ to a meet-semilattice factors uniquely through $\sigma\colon B \to \supp(B)$ (see \cite[Proposition 2.5]{MSS}).  Note in particular that this implies that for a meet-semilattice $L$, its support semilattice $\supp(L)$ is isomorphic to itself via the map $p \mapsto L \wedge p = \{q \in L : q \leq_L p\}.$ 

The second poset associated to an LRB $B$ is an order on the elements of $B$ themselves which we call the \textbf{semigroup order} $B = (B, \leq)$. In particular, for $x, y \in B$, we  write \begin{equation}
    x \leq y\ \text{if and only if}\ yx = x
\end{equation}
or equivalently, $xB\subseteq yB$\footnote{In the language of semigroup theory, this is Green's $\mathscr R$-order.}. If $L$ is a meet-semilattice, the semigroup poset of $L$ is also $L$ itself. 

 Notice that any homomorphism $\tau\colon B\to B'$ of LRBs is a poset homomorphism, as $yx=x$ implies $\tau(y)\tau(x)=\tau(yx)=\tau(x)$.      It follows that the support map $\sigma\colon B\to \Lambda(B)$ is a poset homomorphism.  The LRB axiom implies left translation $\lambda_a\colon B\to B$ by $a\in B$ is a homomorphism, as $\lambda_a(b)\lambda_a(c) = abac=abc=\lambda_a(bc)$. Thus the action of $B$ on the left of itself is by poset homomorphisms, cf.~\cite[Lemma~2.1]{MSS}.

We will often need to consider certain subposets of $B.$ For $y \in B$, we write $B^{<y}$ and $B^{\leq y}$ to be the subposets (in fact, subsemigroups) $\{x \in B: x < y\}$ and $\{x \in B: x \leq y\}$, respectively. For future use, we record a helpful fact about the action of elements with large-enough support on the subposet $B^{<y}.$

\begin{lemma}\label{lem:action-of-y'-on-By-poset-auts}
Let $\sigma(y) \leq \sigma(z).$ Then, the map $B^{<y} \to B^{<zy}$ given by left multiplication by $z$ is a semigroup, and hence poset, isomorphism.
\end{lemma}
\begin{proof}
First note that left multplication by $z$ on $B^{\leq y}=yB$ coincides with left multiplication by $zy$.   Since $\sigma(y)=\sigma(zy)$, left multiplication by $z$ yields an isomorphism of LRBs (and hence posets) $B^{\leq y}\to B^{\leq zy}$ taking $y$ to $zy$ by~\cite[Lemma~2.7]{MSS}.  The result follows, as $B^{<a}=B^{\leq a}\setminus \{a\}$ for $a\in B$.  
\end{proof}

A very important class of subposets of $B$ are called the \textbf{contractions of $B$:} 

\begin{definition}\rm\label{def:contractions-of-B}
    Given a support $X$ in $\supp(B)$, define the \textbf{contraction} \textbf{$B_{\geq X}$} to be the subsemigroup of $B$ consisting of the elements $\{b \in B: \sigma(b) \geq X\}$.  We may view it as a subposet of the semigroup poset.
\end{definition}

\subsubsection{Semigroup algebras and symmetry}
We assume throughout that $\kk$ is a field---we will sometimes also add assumptions about its characteristic. The \textbf{semigroup algebra}  $\kk B$ of $B$ is the $\kk$-vector space with formal basis $B.$ Multiplication is extended $\kk$-linearly, so that
\begin{align}\label{eqn:semigroup-algebra}
    \left(\sum_{b \in B}c_b b\right) \cdot \left(\sum_{b' \in B}d_{b'}b'\right) = \sum_{b, b' \in B}(c_bd_{b'})bb'.
\end{align}

Since $B$ may not have an identity, the semigroup algebra may not have a two-sided multiplicative identity; however, there are LRBs without an identity whose semigroup algebras do have such an identity. We shall provide the condition for $\kk B$ to have a $1$ (due to Margolis--Saliola--Steinberg) in \cref{prop:connected-unit}.  

We say $B$ carries the symmetries of a finite group $G$ if $G$ acts on $B$ by \textbf{semigroup automorphisms}, meaning $g(b\cdot b') = g(b)\cdot g(b')$ for all $g \in G$ and $b, b' \in B.$ When $G$ acts by $B$ by semigroup automorphisms, the induced actions on its related objects behave quite nicely.

\begin{prop}\label{prop:G-acts-nicely-on-the-other-objects}
    If a finite group $G$ acts on $B$ by semigroup automorphisms, then
    \begin{enumerate}
        \item[(i)] By extending its action $\kk$-linearly, $G$ acts by algebra automorphisms on $\kk B$.
        \item[(ii)] The induced action on the support semilattice $\supp(B)$ is by poset automorphisms and the support map $\sigma\colon B \to \supp(B)$ is $G$-equivariant.
        \item[(iii)] The induced action on the semigroup poset $B$ is by poset automorphisms.
    \end{enumerate}
\end{prop}
\begin{proof}
    Item (i) follows from  functoriality of the semigroup algebra construction. For (ii), note that $G$ takes principal left ideals to principal left ideals, since $g(Bb) = g(B)g(b) = Bg(b)$; this also proves that $\sigma$ is $G$-equivariant. The $G$-action is obviously intersection preserving, and thus is by semigroup, and hence poset, automorphisms. Note (iii) follows as semigroup isomorphisms are poset isomorphisms. 
\end{proof}

Given a support $X \in \supp(B)$, we write $\GX$ to denote the {subgroup of $G$ which \textbf{setwise} stabilizes $X$}. Observe that the restriction of the action of $\GX$ to the contraction $B_{\geq X}$ is by semigroup automorphisms.

Provided $G$ is understood, we write $[x]$ to denote the $G$-orbit of an element $x \in B$, and use $B / G$ to denote the set of $G$-orbits of $B.$ Analogously, we write $[X]$ to denote the $G$-orbit of a support $X \in \supp(B)$, and  use $\supp(B) / G$ to denote the set of $G$-orbits of supports in $\supp(B)$.

Invariant subalgebras are the central concept in this paper. Since $G$ acts on $\kk B$ by algebra automorphisms, the subspace of elements  which are pointwise fixed by the action of each element in $G$ form a subalgebra of $\kk B$ called the \textbf{$G$-invariant subalgebra,} which we  write as $\left(\kk B\right)^G$. 

It is straightforward to check that the orbit-sums $\left \{\sum_{x' \in [x]} x'\right\}$ as $[x]$ varies in $B / G$ form a $\kk$-basis of $(\kk B)^G$. Additionally, if we write $\kk G$ for the \textbf{group algebra} of $G$ (defined analogously to the semigroup algebra in  \cref{eqn:semigroup-algebra}), then
for any choice of elements $x \in \kk B, y \in \left(\kk B\right)^G$, and $z \in \kk G,$ it is straightforward to check that 
 \begin{align} \label{eqn:commuting-actions}
     z(xy) = (zx)y\ \text{and}\ z(yx) = y(zx)
 \end{align}
 where view $\kk B$ as a left $\kk G$-module by extending the $G$-action linearly. 

\subsection{Example \#1: Face semigroups of real, central hyperplane arrangements}
The most well-studied LRBs are semigroups arising from real, central, hyperplane arrangements. Indeed, these were the LRBs considered by Bidigare--Hanlon--Rockmore in \cite{BHR}. 
\subsubsection{General theory}\label{subsec:general-arrangements}
A \textbf{real, central hyperplane arrangement}  is a collection $\mathcal{A}$ of \textbf{hyperplanes} in $\mathbb{R}^n.$ Given an arrangement $\mathcal{A}$, pick some vector $v \in \mathbb{R}^n$ not on any hyperplane in $\mathcal{A}$. Each hyperplane $H \in \mathcal{A}$ determines three subsets of $\mathbb{R}^n$: the hyperplane itself \textbf{$H = :H^0$}, and the two half-spaces of $V$ on either side of $H$. Let \textbf{$H^+$} be the half space containing $v$ and let \textbf{$H^-$} be the other. Picking one of these three subsets for each $H \in \mathcal{A}$, then taking the intersection of the chosen subsets, gives either the empty set or a \textbf{face} of the arrangement. In other words, the faces of $\mathcal{A}$ are the nonempty intersections of the form 
\[F = \bigcap_{H \in \mathcal{A}} H^{\sgn_H(F)},\] where $\sgn_H(F) \in \{+, - , 0\}$. We call $\sgn(F):= (\sgn_H(F): H \in \mathcal{A})$ the \textbf{sign vector} of $F$. Note that if $v$ and $v'$ lie in the same \textbf{chamber} (connected component of $\mathbb{R}^n \setminus \mathcal{A})$, then picking $v'$ instead of $v$ as the starting vector results in the same sign vector for each face. Hence, the choice of $v$ only matters up to chamber -- and we shall see shortly that the choice of this chamber does not actually impact the LRB structure. For example, there are thirteen faces in the hyperplane arrangement in the left image of \cref{fig:hyperplane-pictures}: the origin, six rays, and six chambers. The face marked $F$ has $\sgn(F) = (-, 0, -)$.

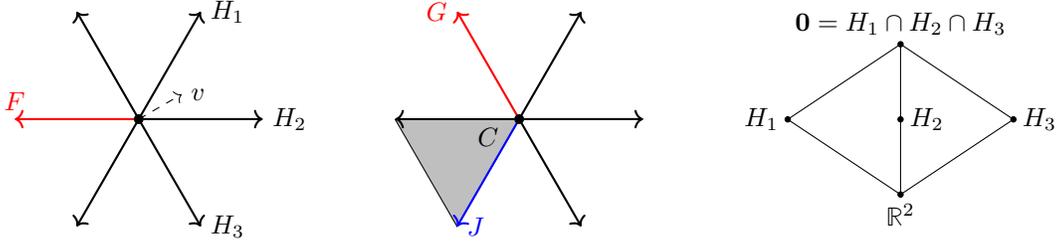
\begin{figure}
\centering
\begin{minipage}{.3\textwidth}
\centering
\begin{tikzpicture}[scale=1.65]
\draw[black, thick, <-> ] (-.5, -0.86602540378
) -- (.5,  0.86602540378) node[anchor=west] {$H_1$}
;
\draw[black, thick, <->] (-.5, 0.86602540378
) -- (.5,  -0.86602540378)
node[anchor=west] {$H_3$};
 \draw[red, thick, <-]  (-1, 0
)-- (0, 0);
\filldraw[red] (-1,0) circle (0pt) node[anchor=south] {$F$};
\draw[black, thick, ->](0, 0)-- (1, 0) node[anchor=west] {$H_2$}
;
\draw[black, dashed, ->] (0, 0
) -- (.3464, .2) node[anchor=west] {$v$}
;
\filldraw[black] (0,0) circle (1pt) ;
\end{tikzpicture}
\end{minipage}%
\begin{minipage}{.3\textwidth}
\centering
\begin{tikzpicture}[scale=1.65]
\coordinate (a) at (0,0);
\coordinate (b) at (-.5, -0.86602540378
);
\coordinate (c) at (-1,0);
\coordinate(d) at (-.25, 0);
\filldraw[fill=lightgray] (a)--(b)--(c) ;
\draw (a) -- (d) node[anchor=north]{$C$};
\draw[blue, thick, -> ]  (0, 0) -- (-.5, -0.86602540378
) node[anchor=west]{$J$}; 
\draw[black, thick, -> ] (0, 0) -- (.5,  0.86602540378) node[anchor=west] {};
\draw[black, thick, ->] (0,0) -- (.5,  -0.86602540378)
node[anchor=west] {};
\draw[red, thick, ->]  (0,0) -- (-.5, 0.86602540378
) node[anchor=east]{$G$};
 \draw[black, thick, <-]  (-1, 0
)-- (0, 0);
\filldraw[black, thick, ->] (-1,0) circle (0pt);
\draw[black, thick, ->](0, 0)-- (1, 0) node[anchor=west] {};
\filldraw[black] (0,0) circle (1pt) ;
\end{tikzpicture}
\end{minipage}%
\begin{minipage}{.3\textwidth}
\centering
    \begin{tikzpicture}[scale=1]
    \filldraw[black] (0, 0) circle (1pt) node[anchor=north]{$\mathbb{R}^2$};
    \filldraw[black] (-1.5, 1) circle (1pt) node[anchor=east]{$H_1$};
    \filldraw[black] (0, 1) circle (1pt) node[anchor=west]{$H_2$};
    \filldraw[black] (1.5, 1) circle (1pt) node[anchor=west]{$H_3$};
    \filldraw[black] (0, 2) circle (1pt) node[anchor=south]{$\mathbf{0} = H_1 \cap H_2 \cap H_3$};
    \draw[black] (0, 0)--(1.5, 1);
    \draw[black] (0, 0)--(-1.5, 1);
    \draw[black] (0, 0)--(0, 1);
    \draw[black] (-1.5, 1)--(0, 2);
    \draw[black] (0, 1)--(0, 2);
    \draw[black] (1.5, 1)--(0, 2);  
\end{tikzpicture}
\end{minipage}
\caption{The examples in \cref{subsec:general-arrangements}}
    \label{fig:hyperplane-pictures}
\end{figure}

We write \textbf{$\mathcal{F}(\mathcal{A})$} to be the set of faces of $\mathcal{A}.$ The faces turn out to carry a multiplication operation, as was first considered by Tits~\cite{tits1974buildings, TITS1976265}. To define it, we first introduce a small LRB called \textbf{$L$}. 

\begin{definition}[Definition of $L$, $L^k$]\label{def:L-LRB}\rm
    Define \textbf{$L$} to be the {left regular band monoid} with elements $\{-, 0, +\}$ and multiplication defined by 
\begin{align*}
    xy = \begin{cases}
        x &\text{if }x \in \{+, -\},\\
        y & \text{ if } x = 0.
    \end{cases}
\end{align*}

For $k \geq 0,$ the LRB monoid \textbf{$L^k$} consists of $k$-vectors in $\{-,0, +\}^k$ with componentwise multiplication. 
\end{definition}

If $\mathcal{A}$ has $k$ hyperplanes, then $F \mapsto \sgn(F)$ is an injective map from $\mathcal{F}(\mathcal{A})$ to $L^k,$ so it makes sense to define an inverse map $\sgn^{-1}\colon \sgn\left(\mathcal{F}(\mathcal{A})\right) \to \mathcal{F}(\mathcal{A})$ on the image $\sgn(\mathcal{F}(\mathcal{A})).$ 
Less obvious is the fact that $\sgn(\mathcal{F}(\mathcal{A}))$ is closed under the multiplication.  Assuming this we can now define Tits's multiplication operation.

\begin{definition}\rm\label{def:face-product}
Given faces $F$ and $G$ of $\mathcal{F}(\mathcal{A})$, define the product $F \cdot G = FG$ to be
\[F \cdot G = \sgn^{-1}\left(\sgn(F)\cdot \sgn(G)\right),\]
where multiplication on the right hand side is within $L^k.$
\end{definition}

Details that this is well defined can be found in~\cite[Section~4.1]{OrientedMatroids}, where the following geometric description of the product is given. 
\begin{prop}
\label{prop:geo-interpretation-of-face-mult}
  Let $F$ and $G$ be faces in $\mathcal{F}(\mathcal{A})$ and let $x_F, x_G$ be two arbitrary points in $F$ and $G$, respectively. Then, $FG$ is the unique face for which there exists some constant $c$ such that $FG$ contains $x_F + (x_G - x_F)\varepsilon$ for all $0 < \varepsilon < c.$
\end{prop}
Intuitively, $FG$ is the face one enters when taking a ``small'' step from $F$ towards $G.$  In particular, the product is independent of the choice of chamber defining the orientation. For example, in the middle image of \cref{fig:hyperplane-pictures}, the product $JG = C$.

It is straightforward that $\mathcal{F}(\mathcal{A})$ is a left regular band monoid as the intersection of all the hyperplanes in $\mathcal A$ has the all-zeros, identity sign vector. 

\subsubsection*{Support semilattice of $\mathcal{F}(\mathcal{A})$} The \textbf{intersection lattice $\mathcal{L}(\mathcal{A})$} of a central hyperplane arrangement $\mathcal{A}$ is the poset of all intersections of hyperplanes ordered by \textit{reverse} inclusion. In particular, the empty intersection $\mathbb{R}^n$ is the minimum of $\mathcal{L}(\mathcal{A})$. Note that the intersections of two different sets of hyperplanes can lead to the same intersection. It is well known that the lattice $\mathcal{L}(\mathcal{A})$ coincides with the support semilattice $\Lambda \left(\mathcal{F}(\mathcal{A})\right)$ (cf.~\cite{BrownonLRBs,MSS}). More precisely, there is a semigroup homomorphism $\mathcal{F}(\mathcal A)\to \mathcal{L}(\mathcal A)$ sending a face $F$ to its linear span (which is the, perhaps empty, intersection of all the hyperplanes of $\mathcal A$ containing it), and two faces have the same linear span precisely when they generate the same principal left ideal.  Thus, we can identify the support semilattice with the intersection lattice.

\begin{example}\rm
Let $\mathcal{A}$ and $F$ be the arrangement and face from \cref{fig:hyperplane-pictures}. Then, the intersection lattice $ \mathcal{L}(\mathcal{A}) \cong \supp(\mathcal{F}(\mathcal{A}))$ is drawn in the right picture of  \cref{fig:hyperplane-pictures}. Note that $\sigma(F)$ is the left ideal consisting of $F$, the other one-dimensional ray on the same hyperplane as $F$, and all six chambers. This left ideal maps to the intersection $H_2$ under the isomorphism of the support semilattice with the intersection lattice.
\end{example}

 \subsubsection*{Semigroup poset and contractions of $\mathcal{F}(\mathcal{A})$} It is well known (see \cite[Proposition 2.2.4]{Bidigare}) that $G\leq F$ in the semigroup poset if and only if the closure of $F$ is contained in the closure of $G.$ 

Every central hyperplane arrangement $\mathcal{A}$ has a convex polytope associated to it called the \textbf{zonotope polar to $\mathcal{A}$}, which is given by taking the Minkowski sum of the normal vectors for the hyperplanes in the (\textit{essentialization}) of $\mathcal{A};$ see \cite[\S 2.4.1]{MSS}. It turns out that the semigroup poset of $\mathcal{F}\left(\mathcal{A}\right)$ is isomorphic to the face poset of the zonotope. Throughout this paper, unless otherwise stated, we do not consider face posets as having an ``empty face''.

Moreover, for $X\in \mathcal{L}(\mathcal{A})$, the contraction $\mathcal{F}\left(\mathcal{A}\right)_{\geq X}$ is the face semigroup of a different central arrangement---the arrangement whose underlying space is the intersection $X \cong \mathbb{R}^{\dim_\mathbb{R} X}$ and whose hyperplanes are the intersections of hyperplanes of $\mathcal{A}$ with $X.$ Hence, we obtain the following:
\begin{prop}\label{prop:contractions-arrangements-polytopes}
    Let $\mathcal{A}$ be a real, central hyperplane arrangement and let $X$ be an intersection in $\mathcal{L}(\mathcal{A}).$ As posets, the contraction $\mathcal{F}(\mathcal{A})_{\geq X}$ is isomorphic to the face poset of a convex polytope.
\end{prop}

\subsubsection{Arrangements with symmetry}
Assume that $\mathcal{A}$ is a real, {central}  hyperplane arrangement and that $G$ is a finite subgroup of $GL_n(\mathbb{R})$ which preserves $\mathcal{A}.$ Then, $G$ acts by semigroup automorphisms on $\mathcal{F}(\mathcal{A})$:
\begin{lemma}\label{lem:g-acts-hyperplane-semigroup-auts}
    If $G$ and $\mathcal{A}$ are as above, then $G$ acts on $\mathcal{F}(\mathcal{A})$ by semigroup automorphisms.
\end{lemma}

\begin{proof}
Since $G$ permutes the hyperplanes of $\mathcal{A}$ (and thus the halfspaces), it must also permute the faces in $\mathcal{F}(\mathcal{A}).$ It is immediate from \cref{prop:geo-interpretation-of-face-mult} that $G$ acts by semigroup automorphisms. 
\end{proof}

In this situation, we also have that $G$ acts on $\mathcal{L}(\mathcal{A})$ by poset automorphisms. The isomorphism between $\supp(\mathcal{F}(\mathcal{A}))$ and $\mathcal{L}(\mathcal{A})$ becomes upgraded in this situation as a special case of \cref{prop:G-acts-nicely-on-the-other-objects}:

\begin{lemma}\label{lem:intersection-lattice-support-poset-G-equiv}
    If $G$ is a finite subgroup of $GL_n(\mathbb{R})$ preserving $\mathcal{A}$, then the isomorphism between $\supp(\mathcal{F}(\mathcal{A}))$ and $\mathcal{L}(\mathcal{A})$ is $G$-equivariant.
\end{lemma}
\begin{proof}
This follows because the map sending a face to its linear span is $G$-equivariant. 
\end{proof}

We note that the setwise stabilizer $\GX$ of a support $X$ of $\mathcal{F}(\mathcal{A})$ coincides with the the \textit{setwise} stabilizer of the corresponding intersection in $\mathcal{L}(\mathcal{A}).$

\subsection{Example \# 2: $\mathrm{CAT}(0)$-cube complex LRBs}\label{sec:CAT0-lrbs}
Our second running example is a new type of LRB, discovered by Margolis--Saliola--Steinberg \cite{MSS} and Bandelt--Chepoi--Knauer \cite{Bandelt-Hans-Juergen-Chepoi}, formed from the faces of a $\catzero$-cube complex. 

\subsubsection{Simplicial complexes}\label{sec:simplicial-complexes}
Before explaining $\catzero$-cube complexes, it will be useful to review some concepts from the theory of simplicial complexes. 

An (abstract) \textbf{simplicial complex} $K$ is a collection of nonempty sets  closed under taking {nonempty} subsets.  A subset $\sigma \in K$ is called a \textbf{face} of $K$, or a $(|\sigma| - 1)$-\textbf{simplex}. We often call the $0$-simplices \textbf{vertices} and the maximal faces \textbf{facets}. The \textbf{$1$-skeleton} of a simplex $\sigma \in K$ is the collection of all $1$-simplices in $\sigma.$

The \textbf{face poset} $\mathcal{P}(K)$ of $K$ consists of all faces of $K$ ordered under inclusion. On the other hand, we can obtain a simplicial complex from a poset $P$ called the \textbf{order complex of $P$}, written $\Delta(P)$, which has the nonempty \textbf{chains} of $P$ as its simplices. 

A simplicial complex can be realized as a topological space through its \textbf{geometic realization} $\|K\|$ within Euclidean space $\mathbb{R}^{\mathrm{Vertices}(K)}.$ This is constructed by taking the union of the convex hulls of the standard basis vectors corresponding to the vertices in $\sigma$, over all $\sigma \in K.$

\subsubsection{$\catzero$-cube complexes}\label{subsec:catzero-cone-1}
    The $k$-dimensional \textbf{unit cube} is defined as the space $I^k :=[0, 1]^k \subseteq \mathbb{R}^k$. The cube $I^k$ has $3^k$ \textbf{faces}, or \textbf{subcubes,} counted by deciding whether each coordinate $x_i$ for $1 \leq i \leq k$ is (i) fixed to be $1$, (ii) fixed to be $0$, or (iii) is permitted to vary between $0$ and $1.$ The dimension of a subcube is the number of coordinates belonging to category (iii). Note that $I^k$ also carries the structure of a \textbf{metric space}, inheriting the Euclidean metric from $\mathbb{R}^k.$
    
    A \textbf{cube complex} (or a cubical complex) is a space formed by gluing together a \textit{finite} number of unit cubes (of various dimensions) in a special manner. Although many of the cube complexes studied in geometric group theory are formed from gluing \textit{infinitely} many cubes, finite cubical complexes have enjoyed attention in recent years in combinatorics \cite{ArdilaSociety,Ardila-Owen-Sullivant,BillerHolmesVogtmann, Rowlands} and will be the setting we need as we are working with finite semigroups. Besides imposing that our cube complexes are finite, we use the conventions of Bridson--Haefliger \cite[Definition 7.3.2]{BridsonHaeflinger}. In particular, say we glue a set of unit cubes $\{I^{k_i} : 1 \leq i \leq n\}$ to obtain a space $\mathcal{C} := U / {\sim}$  where $U = \displaystyle \bigsqcup_{1 \leq i \leq n}I^{k_i}.$ Then, $\mathcal{C}$ is a cubical complex if it satisfies
    \begin{enumerate}
        \item[(i)] \textit{The unit cubes $\{I^{k_i}: 1 \leq i \leq n\}$ are embedded in $\mathcal{C}$}; formally, the restriction of the natural projection map $p\colon U \to \mathcal{C}$ to each unit cube $I^{k_i}$ is injective, and
        \item[(ii)] \textit{The intersection of any two faces is a face}; formally, if the intersection $p(I^{k_i}) \cap p(I^{k_j})$ is nonempty,  there are subfaces $T_i \subseteq I^{k_i}$ and $T_j \subseteq I^{k_j}$ such that {$p(T_i) = p(T_j) = p(I^{k_i}) \cap p(I^{k_j})$}   and there is an isometry $h:T_i \to T_j$ for which $p \circ h = p$ on $T_i.$
    \end{enumerate}
Although we will not use it, a cubical complex is also a metric space, and carries a metric induced by the Euclidean metric on each cube \cite[\S I.7.4]{BridsonHaeflinger}. 
    
A type of cubical complex which has especially nice properties is a \textbf{$\catzero$-cube complex}. Informally, a \textbf{$\catzero$-metric space} is one whose triangles are ``thinner'' than those in Euclidean space (see \cite[\S 2.1]{BridsonHaeflinger} for a precise description). For cube complexes, there is a combinatorial method for checking the $\catzero$-property called \textbf{Gromov's Criterion}. To explain it, we must first define a couple of concepts. 

The \textbf{link} of a vertex (or $0$-cube) $v$ in a cube complex $\mathcal{C}$, written $\mathrm{link}_{\mathcal{C}}(v)$, is the simplicial complex whose $(k-1)$-simplices are collections of $k$ $1$-cubes in $\mathcal{C}$ who are (i) adjacent to $v$ and (ii) all belong to a common cube. A simplicial complex $K$ is a \textbf{flag complex} if whenever the $1$-skeleton of a simplex belongs to $L,$ so does the full simplex.

\begin{thm}[Gromov's Criterion] \label{thm:Gromov's-Criterion}
    A cube complex $\mathcal{C}$ is a $\catzero$-cube complex if and only if it is simply connected and the link of every vertex in $\mathcal{C}$ is a flag complex.
\end{thm}

An important concept in the theory of $\catzero$-cube complexes is that of \textbf{hyperplanes}. We place an equivalence relation $\sim$ on the $1$-cubes (edges) of a $\catzero$-cube complex $\mathcal{C}$ by taking the transitive closure of the following relation: edges $e$ and $f$ are related if they are opposite each other in some $2$-cube. A \textbf{midcube} of an $n$-dimensional subcube $c \subseteq \mathcal{C}$ is an $(n - 1)$-cube (although \textit{not} itself a face of $\mathcal{C}$) which bisects all the $1$-cubes in $c$ within a fixed equivalence class.  The hyperplane corresponding to (or \textit{dual} to) the equivalence class $[e]$ is the union of all midcubes bisecting edges $f \sim e.$ The {hyperplanes} of the entire cube complex $\mathcal{C}$ biject with equivalence classes of $1$-cubes in $\mathcal{C}$. We write $\mathcal{H}(\mathcal{C})$ to denote the set of hyperplanes of $\mathcal{C}$. See \cref{fig:hyperplanes} for two examples of $\catzero$-cube complexes with their hyperplanes drawn in.

    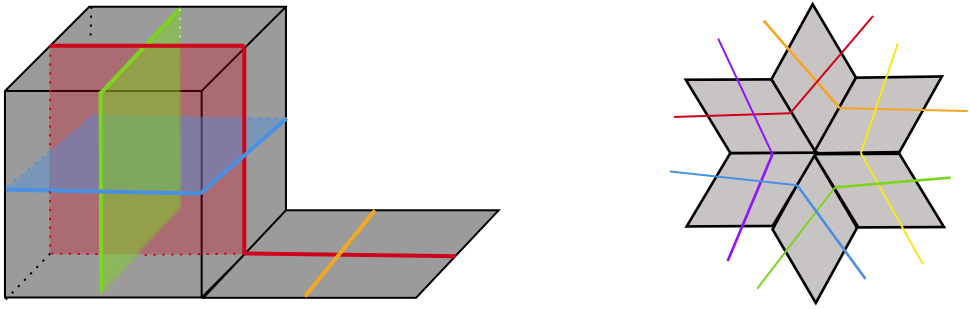
\begin{figure}
\begin{minipage}{.45\textwidth}

        \centering

\tikzset{every picture/.style={line width=0.75pt}} 

\begin{tikzpicture}[x=0.75pt,y=0.75pt,yscale=-1,xscale=1]

\draw  [fill={rgb, 255:red, 155; green, 155; blue, 155 }  ,fill opacity=1 ] (153,153.53) -- (195.52,111.01) -- (294.72,111.01) -- (294.72,215.24) -- (252.2,257.76) -- (153,257.76) -- cycle ; \draw   (294.72,111.01) -- (252.2,153.53) -- (153,153.53) ; \draw   (252.2,153.53) -- (252.2,257.76) ;
\draw  [dash pattern={on 0.84pt off 2.51pt}]  (196.58,111.01) -- (196.14,127.18) ;
\draw  [dash pattern={on 0.84pt off 2.51pt}]  (175.86,235.53) -- (153,259.25) ;
\draw  [fill={rgb, 255:red, 155; green, 155; blue, 155 }  ,fill opacity=1 ] (294.72,213.74) -- (401.99,213.74) -- (360.36,257.93) -- (253.09,257.93) -- cycle ;
\draw  [ color={rgb, 255:red, 208; green, 2; blue, 27 }  ,draw opacity=0.26 ][fill={rgb, 255:red, 208; green, 2; blue, 27 }  ,fill opacity=0.31 ] (175.86,130.7) -- (273.7,130.7) -- (273.7,235.53) -- (175.86,235.53) -- cycle ;
\draw  [ color={rgb, 255:red, 74; green, 144; blue, 226 }  ,draw opacity=0 ][fill={rgb, 255:red, 74; green, 144; blue, 226 }  ,fill opacity=0.42 ] (197.08,165.34) -- (295.12,167.35) -- (251.61,205.33) -- (153.57,203.33) -- cycle ;
\draw  [color={rgb, 255:red, 126; green, 211; blue, 33 }  ,draw opacity=0.29 ][fill={rgb, 255:red, 126; green, 211; blue, 33 }  ,fill opacity=0.51 ] (201.1,154.92) -- (241.17,112.04) -- (241.59,212.18) -- (201.52,255.06) -- cycle ;
\draw [color={rgb, 255:red, 126; green, 211; blue, 33 }  ,draw opacity=1 ][line width=1.5]    (201.25,153.63) -- (201.52,255.06) ;
\draw [color={rgb, 255:red, 126; green, 211; blue, 33 }  ,draw opacity=1 ][line width=1.5]    (241.17,112.04) -- (201.1,154.92) ;
\draw [color={rgb, 255:red, 208; green, 2; blue, 27 }  ,draw opacity=1 ][line width=1.5]    (273.7,130.7) -- (273.7,235.53) ;
\draw [color={rgb, 255:red, 208; green, 2; blue, 27 }  ,draw opacity=1 ][line width=1.5]    (175.86,130.7) -- (273.7,130.7) ;
\draw [color={rgb, 255:red, 74; green, 144; blue, 226 }  ,draw opacity=1 ][line width=1.5]    (251.61,205.33) -- (153.57,203.33) ;
\draw [color={rgb, 255:red, 74; green, 144; blue, 226 }  ,draw opacity=1 ][line width=1.5]    (295.12,167.35) -- (251.61,205.33) ;
\draw [color={rgb, 255:red, 74; green, 144; blue, 226 }  ,draw opacity=1 ][line width=0.75]  [dash pattern={on 0.84pt off 2.51pt}]  (175.75,183.88) -- (152.37,203.34) ;
\draw [color={rgb, 255:red, 74; green, 144; blue, 226 }  ,draw opacity=1 ][line width=0.75]  [dash pattern={on 0.84pt off 2.51pt}]  (295.12,167.35) -- (275.25,167.63) ;
\draw [color={rgb, 255:red, 184; green, 233; blue, 134 }  ,draw opacity=1 ][line width=0.75]  [dash pattern={on 0.84pt off 2.51pt}]  (241.17,112.04) -- (241.25,129.63) ;
\draw [color={rgb, 255:red, 126; green, 211; blue, 33 }  ,draw opacity=1 ][line width=0.75]  [dash pattern={on 0.84pt off 2.51pt}]  (220,236.38) -- (201.5,256.63) ;
\draw [color={rgb, 255:red, 208; green, 2; blue, 27 }  ,draw opacity=1 ][line width=0.75]  [dash pattern={on 0.84pt off 2.51pt}]  (175.86,130.7) -- (175.75,183.88) ;
\draw [color={rgb, 255:red, 208; green, 2; blue, 27 }  ,draw opacity=1 ][line width=0.75]  [dash pattern={on 0.84pt off 2.51pt}]  (175.75,204.13) -- (175.86,235.53) ;
\draw [color={rgb, 255:red, 208; green, 2; blue, 27 }  ,draw opacity=1 ][line width=0.75]  [dash pattern={on 0.84pt off 2.51pt}]  (199.75,235.88) -- (175.86,235.53) ;
\draw [color={rgb, 255:red, 208; green, 2; blue, 27 }  ,draw opacity=1 ][line width=0.75]  [dash pattern={on 0.84pt off 2.51pt}]  (273.7,235.53) -- (220,236.38) ;
\draw [color={rgb, 255:red, 208; green, 2; blue, 27 }  ,draw opacity=1 ][line width=1.5]    (380.21,236.93) -- (273.7,235.53) ;
\draw [color={rgb, 255:red, 245; green, 166; blue, 35 }  ,draw opacity=1 ][line width=1.5]    (339.64,213.79) -- (304.27,257.25) ;
\end{tikzpicture}

        \end{minipage}
\begin{minipage}{.45\textwidth}
\centering
      \tikzset{every picture/.style={line width=1pt}} 
   \begin{tikzpicture}[x=0.75pt,y=0.75pt,yscale=-1,xscale=1]

\draw  [fill={rgb, 255:red, 201; green, 196; blue, 196 }  ,fill opacity=1 ] (254.36,181) -- (275.16,143.06) -- (253.34,105.9) -- (232.54,143.84) -- cycle ;

\draw [fill={rgb, 255:red, 201; green, 196; blue, 196 }  ,fill opacity=1 ]  (318.98,142.69) -- (296.51,105.71) -- (253.34,105.89) -- (275.81,142.88) -- cycle ;

\draw  [fill={rgb, 255:red, 201; green, 196; blue, 196 }  ,fill opacity=1 ] (253.82,105.25) -- (274.62,67.31) -- (252.76,30.23) -- (231.96,68.17) -- cycle ;
\draw  [fill={rgb, 255:red, 201; green, 196; blue, 196 }  ,fill opacity=1 ] (253.2,105.48) -- (296.54,104.87) -- (317.96,66.7) -- (274.62,67.31) -- cycle ;
\draw  [fill={rgb, 255:red, 201; green, 196; blue, 196 }  ,fill opacity=1 ] (189.23,142.39) -- (232.59,142.16) -- (254.61,105.12) -- (211.26,105.34) -- cycle ;
\draw  [fill={rgb, 255:red, 201; green, 196; blue, 196 }  ,fill opacity=1 ] (254.43,105.16) -- (231.96,68.17) -- (188.79,68.36) -- (211.26,105.34) -- cycle ;
\draw [color={rgb, 255:red, 245; green, 166; blue, 35 }  ,draw opacity=1 ] (228.08,38.52) -- (266.45,82.68) ;
\draw [color={rgb, 255:red, 245; green, 166; blue, 35 }  ,draw opacity=1 ][line width=0.75]    (266.45,82.68) -- (331,84.15) ;
\draw [color={rgb, 255:red, 208; green, 2; blue, 27 }  ,draw opacity=1 ][line width=0.75]    (182.88,87.14) -- (241.42,85.26) ;
\draw [color={rgb, 255:red, 208; green, 2; blue, 27 }  ,draw opacity=1 ][line width=0.75]    (241.42,85.26) -- (283.25,36.19) ;
\draw [color={rgb, 255:red, 144; green, 19; blue, 254 }  ,draw opacity=1 ](209.89,159.86) -- (232.49,105.94) ;
\draw [color={rgb, 255:red, 144; green, 19; blue, 254 }  ,draw opacity=1 ][line width=0.75]    (232.49,105.94) -- (205.12,47.6) ;
\draw [color={rgb, 255:red, 248; green, 231; blue, 28 }  ,draw opacity=1 ][line width=0.75]    (295.82,49.89) -- (277.04,105.24) ;
\draw [color={rgb, 255:red, 248; green, 231; blue, 28 }  ,draw opacity=1 ][line width=0.75]    (277.04,105.24) -- (308.41,161.54) ;
\draw [color={rgb, 255:red, 126; green, 211; blue, 33 }  ,draw opacity=1 ]  (322.36,117.75) -- (263.99,122.66) ;
\draw [color={rgb, 255:red, 126; green, 211; blue, 33 }  ,draw opacity=1 ][line width=0.75]    (263.99,122.66) -- (224.77,173.84) ;
\draw [color={rgb, 255:red, 74; green, 144; blue, 226 }  ,draw opacity=1 ]   (279.39,168.87) -- (244.95,121.6) ;
\draw [color={rgb, 255:red, 74; green, 144; blue, 226 }  ,draw opacity=1 ][line width=0.75]    (244.95,121.6) -- (180.76,114.63) ;
\end{tikzpicture}
\end{minipage}
\caption{Hyperplanes of two $\catzero$-cube complexes.}
        \label{fig:hyperplanes}
    \end{figure}

Crucially, in $\catzero$-cube complexes, we have a notion of \textbf{halfspaces}:

\begin{thm}[{\cite[Lemma 2.7]{NibloReevesOlder}}]
    Let $\mathcal{C}$ be a $\catzero$-cube complex. Then, for any hyperplane $H\in \mathcal{H}(\mathcal{C})$, the complement $\mathcal{C} \setminus H$ has two connected components, called {halfspaces.}
\end{thm}

As they did for hyperplane arrangements, \textbf{intersections} of hyperplanes of $\catzero$-cube complexes will play an important role in the structure of the associated LRBs. One key reason is that they themselves are $\catzero$-cube complexes (see~\cite[Theorem~3.25]{MSS} and the references there):

\begin{prop}\label{prop:intersections-of-hyperplanes-are-catzero}
   Let $\mathcal{C}$ be a $\catzero$-cube complex. The intersection $X = \displaystyle\bigcap_{H \in A}H$ of a set of hyperplanes $A \subseteq \mathcal{H}(\mathcal{C})$  is itself a $\catzero$-cube complex with cubes $\{c \cap X \ : \ c \text{ a cube in }\mathcal{C} \text{ with }c \cap X \neq \emptyset\}.$
\end{prop}

 Like in the case of arrangements, the poset of all nonempty intersections of hyperplanes of $\mathcal{C}$, ordered under \textit{reverse} inclusion, including $\mathcal{C}$ viewed as the interesection of the empty collection of hyperplanes, is a meet-semilattice called the \textbf{hyperplane intersection semilattice $\mathcal{L}(\mathcal{C})$}. Fortunately, the structure of the hyperplane intersection semilattice for $\catzero$-cube complexes is---in several ways---much tamer than those of hyperplane arrangements. Specifically, unlike the intersection semilattices for arrangements:
\begin{enumerate}
    \item[(1)] Each intersection $X \in \mathcal{L}(\mathcal{C})$ has a \textit{unique} set of hyperplanes $A \subseteq \mathcal{H}(\mathcal{C})$ such that $X =  \displaystyle \bigcap_{H \in A}H.$ 
    \item[(2)] If $X \leq Y$ in $\mathcal{L}(\mathcal{C})$ with $X$ the intersection of $k$ hyperplanes and $Y$ the intersection of $\ell$ hyperplanes, then the subposet $[X, Y]$ is isomorphic to the Boolean algebra on $\ell - k$ elements.
\end{enumerate}

These properties are immediate from~\cite[Proposition~3.27]{MSS}, which proves that the intersection semilattice is the face poset (this time including an empty face) of the \textbf{nerve} of the set of hyperplanes, that is, the simplicial complex with vertex set the hyperplanes and simplices the collections of hyperplanes with nonempty intersection. 

\subsubsection{LRBs associated to $\catzero$-cube complexes}\label{subsec:general-catzero-cube-lrb}
Margolis--Saliola--Steinberg \cite{MSS} proved that every finite $\catzero$-cube complex has a natural LRB structure on its faces. Independently, in \cite{Bandelt-Hans-Juergen-Chepoi}, Bandelt--Chepoi--Knauer also discovered this LRB structure and showed further that there is an LRB structure on the faces of any $\catzero$-zonotopal complex whose cells are Coxeter zonotopes. We give a brief overview of how the $\catzero$-cube complex LRBs are constructed below, but the reader should consult \cite{MSS, Bandelt-Hans-Juergen-Chepoi} for more details.

Let $\mathcal{C}$ be a $\catzero$-cube complex and recall that $\mathcal{H}(\mathcal{C})$ denotes the set of hyperplanes of $\mathcal{C}.$ For each hyperplane $H \in \mathcal{H}(\mathcal{C})$ choose one of the halfspaces to be the \textbf{positive side of $H$} (written \textbf{$H^{+}$}) and the other to be the \textbf{negative side} (written \textbf{$H^-$}). 

We write $\mathbf{\mathcal{F}(\mathcal{C})}$ to denote the set of cubes (or faces) in $\mathcal{C}.$ Define for each $H \in \mathcal{H}(\mathcal{C})$ a sign function $\sgn_H\colon \mathcal{F}(\mathcal{C}) \to \{-, 0, +\}$ by
     \[   c \longmapsto \begin{cases}
            -, & \text{if}\ c \subseteq H^-\\
            0, & \text{if}\ c \cap H \neq \emptyset\\
            +, & \text{if}\ c \subseteq H^+.
        \end{cases}
    \]
Arbitrarily order the hyperplanes $H \in \mathcal{H}(\mathcal{C})$ as $\mathcal{H}(\mathcal{C}) = \{H_1, H_2, \ldots, H_k\}.$ Define a map $\theta\colon \mathcal{F}(\mathcal{C})\to L^{k}$ by
\[\theta(c)= \left(\sgn_{H_1}(c), \sgn_{H_2}(c), \ldots, \sgn_{H_k}(c)\right)\]
with $L^k$ from \cref{def:L-LRB}.

The following is part of Theorem~3.29 of~\cite{MSS}
\begin{thm}[Definition of LRB associated to a $\catzero$-cube complex]\label{cor:def-of-catzero-lrb}
Let $\mathcal{C}$ be a finite $\catzero$-cube complex. Then
$\theta$ is an order isomorphism of the face poset of $\mathcal{C}$ of with a right ideal in $L^k$, and hence the set $\mathcal{F}(\mathcal{C})$ is an LRB under the multiplication 
\begin{align*}
    c \cdot c' := \theta^{-1}\left(\theta(c) \cdot \theta(c')\right)
\end{align*}
with semigroup poset isomorphic to the face poset of $\mathcal{C}$ via $\theta$.
\end{thm}

A priori, one might worry that this multiplication definition depends on the sign choices we made earlier for the hyperplanes; this turns out to not be a problem.  Indeed, as pointed out in~\cite{MSS}, the product of two cubes $c\cdot c'$ is the face of $c$ consisting of those vertices closest to $c'$ in the path metric.  We prove this for completeness, as we shall use it later.   Note that the vertex set of a cube $c$ consists of all $0$-cubes with sign vector obtained from $\theta(c)$ by replacing all zeroes by $\pm$.

\begin{prop}\label{prop-geom-cube-mult}
Let $\mathcal C$ be a finite $\catzero$-cube complex and let $c,c'$ be cubes of $\mathcal C$.  Then $c\cdot c'$ is the face of $c$ with vertex set consisting of those vertices of $c$ closest to $c'$ in the path metric on the $1$-skeleton of $\mathcal C$.
\end{prop}
\begin{proof}
Note that a cube of $\mathcal C$ is determined by its vertices.
Let $d$ be the path metric on the one-skeleton $\mathcal C^1$.  We freely use that $d(v,v')$ is the number of hyperplanes separating $v,v'$ (cf.~\cite[Lemma~9]{Bandelt-Hans-Juergen-Chepoi}). It follows that the distance of a vertex $v$ to $c'$ is the number of hyperplanes separating $v$ and $c'$.  This is minimized on $c$ by  those vertices $v$ where $\theta(v)_i=\theta(c')_i$ whenever $\theta(c)_i=0$ and $\theta(c')_i\neq 0$.  But these are precisely the vertices of  $\theta^{-1}(\theta(c)\cdot \theta(c'))=c\cdot c'$. 
\end{proof}

The above definition of the multiplication is intrinsic, hence we have the following.

\begin{prop}
       The semigroup in \cref{cor:def-of-catzero-lrb} is independent of the sign choices for hyperplanes and their halfspaces.  
\end{prop}

In an abuse of notation, we write \textbf{$\mathcal{F}(\mathcal{C})$} to denote both the set of faces and the LRB.

\subsubsection*{{Support semilattice for $\mathcal{F}(\mathcal{C})$}}
Let $\mathcal{C}$ be a $\catzero$-cube complex. In \cite[Theorem 3.29(4)]{MSS}, Margolis--Saliola--Steinberg prove that the support semilattice of $\mathcal{F}({\mathcal{C}})$ is isomorphic to the hyperplane intersection semilattice $\mathcal{L}(\mathcal{C})$ via the mapping 
  \begin{align}
   \supp \left(\mathcal{F}(\mathcal{C})\right) &\longrightarrow \mathcal{L}(\mathcal{C}) \nonumber\\
  \mathcal{F}(\mathcal{C}
  )\cdot c &\longmapsto\bigcap_{\substack{H \in \mathcal{H}(\mathcal{C}):\\ \sgn_H(c) = 0}}H. \label{eqn:iso-support-intersection-cubes}
  \end{align} 

Note that if $G$ acts by isometries and permutes the cubes of $\mathcal{C},$ then this isomorphism is $G$-equivariant. In an abuse of notation, we write $\sigma(c)$ to denote either the principal left ideal or the corresponding intersection; which interpretation we are using should be clear from context.

\subsubsection*{{Semigroup posets and contractions of $\mathcal{F}(\mathcal{C})$}}
Let $\mathcal{C}$ be a $\catzero$-cube complex with $n$ hyperplanes.
We saw in \cref{cor:def-of-catzero-lrb} that the semigroup poset of $\mathcal{F}(\mathcal{C})$ is isomorphic to the face poset  $\mathcal{P}(\mathcal{C})$ of $\mathcal{C}.$ (Recall again that our face posets do not include a minimum element, or ``empty face.'') Something even stronger is true. Let $X$ be a support in $\supp \left(\mathcal{F}(\mathcal{C})\right),$ which by \eqref{eqn:iso-support-intersection-cubes} one can think of as an intersection in $\mathcal{L}(\mathcal{C}).$ Recall from \cref{prop:intersections-of-hyperplanes-are-catzero} that $X$ is itself a $\catzero$-cube complex. It turns out that the contraction $\mathcal{F}(\mathcal{C})_{\geq X}$ is isomorphic to the face poset of the $\catzero$-cube structure on $X$ via the map $ c\mapsto c \cap X.$

\begin{prop}[{[Theorem 3.29(5)]\cite{MSS}}]\label{prop:contractions-of-catzero-lrbs}
    Let $\mathcal{C}$ be a $\catzero$-cube complex and let $X$ be an intersection in its intersection semilattice $\mathcal{L}(\mathcal{C}).$ Then, the contraction $\mathcal{F}(\mathcal{C})_{\geq X}$ is isomorphic to the face poset of $X.$
\end{prop}

Since $\catzero$ cube complexes are simply connected, they are connected, and so the Hasse diagram of the contractions are always connected.

\subsubsection{$\catzero$-cube LRBs with symmetry}

It is natural to consider what types of group actions on $\catzero$-cube complexes will induce an action on the corresponding LRB by semigroup automorphisms. 

\begin{prop}
If $G$ acts cellularly on a $\catzero$-cube complex $\mathcal{C}$, permuting the cubes of $\mathcal{C}$, then the induced action of $G$ on the LRB $\mathcal{F}(\mathcal{\mathcal{C}})$ is by semigroup automorphisms. 
\end{prop}

\begin{proof}
Note that since $G$ permutes cubes, it acts by graph automorphisms on the $1$-skeleton of $\mathcal C$, and hence preserves the path metric.  It follows that $G$ acts by semigroup automorphisms by Proposition~\ref{prop-geom-cube-mult}.
\end{proof}

We will be able to say the most when the action of $G$ on $\mathcal{C}$ satisfies one extra condition. Following the terminology in \cite[Definition 3.3]{Varghese}, we define the notion of a \textbf{strongly simplicial action} on $\catzero$-cube LRBs as follows:

\begin{definition}\rm\label{def:lrb-action}
    We call the action of a finite group $G$ on a finite $\catzero$-cube complex $\mathcal{C}$ \textbf{strongly simplicial} if it satisfies the following: 
        \begin{itemize}
        \item $G$ acts by isometries
        and permutes the cubes of $\mathcal{C}.$
        \item For all cubes $c \in \mathcal{F}(\mathcal{C}),$ if $c$ is fixed \textit{setwise} by all $g \in G$, then $c$ is fixed \textit{pointwise} by all $g \in G.$
    \end{itemize}
\end{definition}

This second condition is a slight weakening of the more standard hypothesis for group actions on cell complexes that the setwise stabilizer of a cell is its pointwise stabilizer.

The following property of strongly simplicial actions will be useful later on.
\begin{lemma}\label{lem:fixed-cubes-by-support}
    If $G$ has a strongly simplicial action on a $\catzero$-cube complex $\mathcal{C}$, there exists some vertex $v$ of $\mathcal{C}$ fixed by all $g \in G$
\end{lemma}

\begin{proof}
 Since $\mathcal{C}$ is finite-dimensional, it is complete (see the discussion in \cite[Remark 1.7.33]{BridsonHaeflinger}). Thus, by \cite[Corollary II.2.8]{BridsonHaeflinger}, there exists at least one point $x \in \mathcal{C}$ which is fixed by all $g \in G.$ Let $c$ be the cube in $\mathcal{C}$ containing $x$ in its ``interior.'' Since $G$ permutes the cubes of $\mathcal{C}$, we must have that $G$ fixes $c$ setwise. Thus, by the second assumption of \cref{def:lrb-action}, all of $G$ must fix $c$ pointwise. Thus, there exists some vertex $v$ of $c$ which is fixed by $G$, as desired.
\end{proof}

One interesting and well-studied finite $\catzero$-cube complex is the \textit{space of phylogenetic $n$-trees}, which carries a natural strongly simplicial action of the symmetric group $S_n$. The phylogenetic tree LRB takes some work to explain, so we refer the interested reader to \cite{ComminsThesis} for the details of its definition and  structure. We will instead use a simpler example of a $\catzero$-cube complex carrying a strongly simplicial action.

\begin{example}[Cubulated $n$-gon] \label{ex:cubulated-ngon}\rm
    Let $n \geq 4.$ Consider the $\catzero$-cube complex formed from a regular $n$-gon by adding a vertex to the center and adding edges from the central vertex to the midpoints of each of the edges of the $n$-gon, which we now also consider vertices. The resulting cube complex $\mathscr{C}_n$ has $2n + 1$ vertices, $3n$ $1$-cubes, and $n$ $2$-cubes. See \cref{fig:hyperplanes} (right) for $\mathscr{C}_6$. Each link is a triangle-free graph (with the link at the center vertex an $n$-gon) and hence a flag complex. The complex $\mathscr{C}_n$ carries a natural action of the dihedral group $D_{2n}$ of order $2n$ and it is straightforward to check this action \textit{is} a strongly simplicial action. Note, however, that this action does not satisfy the more standard hypothesis that the setwise stabilizer of a cell is its pointwise stabilizer since reflections can setwise stabilize a $2$-cube without stabilizing it pointwise. It is straightforward to check that the cube complex $\mathscr{C}_n$ has $n$ hyperplanes, as illustrated for $n = 6$ in \cref{fig:hyperplanes}.
\end{example}

\subsection{Example \# 3: LRBs associated to geometric lattices}\label{sec:matroid_lrbs}
A \textbf{geometric lattice} is a lattice which is \textbf{atomic}  and \textbf{semimodular}. A lattice is atomic if every element is the join of \textbf{atoms} (atoms are elements which cover the minimal element). A poset $P$ is said to be \textbf{graded} if for all $p$ in $P$, the lengths of the unrefinable chains which end at $p$ agree. For an element $p$ of a graded poset $P$, we write $\rk(p)$ to denote the \textbf{rank} of $p$, which means the length of an unrefinable chain in $P$ ending at $p.$ The rank of an {interval} $x < y$ in a graded poset is $\rk(y) - \rk(x).$ A lattice is semimodular if it is graded and for each $x, y$, the ranks satisfy
\[\rk(x) + \rk(y)  \geq \rk(x \wedge y) + \rk(x \vee y).\]
Geometric lattices arise precisely as lattices of flats of matroids. We refer the reader to \cite{Oxley} for more details.

Three particularly nice geometric lattices we will use as examples are:
\begin{itemize}
    \item The \textbf{Boolean lattice} $B_n$ on $n$ letters, consisting of subsets of $\{1, 2, \ldots, n\}$ ordered under inclusion. The number of elements at rank $k$ is ${n \choose k}$
    \item  The \textbf{projective geometry} $\pg(n - 1, q)$ for $q$ a power of a prime (also known as the $q$-Boolean lattice), consisting of subspaces of $\mathbb{F}_q^{n}$ ordered under inclusion. The number of rank $k$ elements in $\pg(n - 1, q)$  is $\qbinom{n}{k}_q,$ where \begin{align*}
      \qbinom{n}{k}_q := \frac{[n]!_q}{[k]!_q\cdot [n-k]!_q}, &  & &
         [k]!_q := [k]_q\cdot [k-1]_q\cdots [1]_q, &
         [k]_q := 1 + q + q^2 + \cdots + q^{k - 1}.
      \end{align*}
    \item The \textbf{affine geometry} $\ag(n, q)$ consisting of the \textbf{affine subspaces} of $\mathbb{F}_q^n$ ordered under inclusion along with the empty set as an absolute minimum element. An affine subspace is a set of the form $v + V$ where $v \in \mathbb{F}_q^n$ and $V$ is a linear subspace. The number of elements of rank $k$ in $\ag(n, q)$ is  $1$ if $k = 0$ and $q^{n-k + 1}\qbinom{n}{k - 1}_q$ for $k > 0$.
\end{itemize}

There is an LRB $S(\mathcal{L})$ associated by Brown~\cite[\S6.2]{BrownonLRBs} to every geometric lattice $\mathcal{L}.$ The elements of $S(\mathcal{L})$ are essentially prefixes of complete flags of $\mathcal{L}$. More specifically, they are sequences of the form $(X_1, X_2, \ldots, X_k)$ as $k$ varies, where $X_1 \gtrdot \hat{0}$ and $X_i \lessdot X_{i + 1}$ in $\mathcal{L}$ for all $1 \leq i \leq k - 1$. We permit an empty chain $()$ as well. To multiply two flags, we follow the rule
\begin{align*}
(X_1, X_2, \ldots, X_k) &\cdot (Y_1, Y_2, \ldots, Y_\ell)
= (X_1, X_2, \ldots, X_k, X_k \vee Y_1, X_k \vee Y_2, \ldots, X_k \vee Y_\ell)^\wedge,
\end{align*}
where $\wedge$ denotes deleting any repeated flats, as reading from left to right, and $\vee$ means the join operation in $\mathcal{L}.$ The resulting flag is a valid element (a prefix of a complete flag) by virtue of $\mathcal{L}$ being a geometric lattice (namely, using semimodularity). It is then straightforward from the multiplication definition to check that $S(\mathcal{\mathcal{L}})$ is an LRB, and in fact an LRB monoid since $()$ is an identity element. This monoid is the right Rhodes expansion of $\mathcal{L}$ with respect to the join operation, cut to the atoms as generators. 

We point out that $S(B_n)$ is the well-studied \textbf{free left regular band}, which we will denote by $\free_n.$ Similarly, we will use $\free_n^{(q)}$ to denote $S(\pg(n - 1, q))$, also called the $q$-free LRB and studied in \cite{BrownonLRBs,BraunerComminsReiner}.

It is straightforward to check the following: 
\begin{prop} \label{prop:Matroids-with-symmetry}
Let $G$ be a finite group which acts on a geometric lattice $\mathcal{L}$ by poset automorphisms. Then, $G$ acts on the LRB $S(\mathcal{L})$ by monoid automorphisms via $g(X_1,\ldots, X_k)=(g(X_1),\ldots, g(X_k))$.
\end{prop}

\subsubsection*{Support semilattice of $S(\mathcal{L})$} Much like the situations for hyperplane arrangements and cubical complexes, the support semilattice of geometric lattice LRBs reflects the combinatorics of the associated object. It is straightforward to check the following lemma; a proof can be found in \cite[Lemma 2.5.7]{ComminsThesis}.

\begin{lemma}\label{lem:supp-post-for-matroids-is-lattice-of-flats}
    The map $\sigma((X_1, X_2,\ldots, X_k)) \mapsto X_k$ gives a well-defined poset isomorphism between the support semilattice $\supp(S(\mathcal{L}))$ and the dual lattice $\mathcal{L}^{\mathrm{opp}}.$ If $G$ acts on $\mathcal{L}$ by poset automorphisms, the isomorphism is $G$-equivariant.
\end{lemma}

\subsubsection*{Semigroup poset and contractions of $S(\mathcal{L})$}The structure of the semigroup posets of the geometric lattice LRBs $S(\mathcal{L})$ is quite nice.  This is a special case of the known description of the $\mathscr R$-order for the Rhodes expansion, but we include a proof for completness.
\begin{prop}\label{prop:matroids-are-rooted-trees}
   Let $\mathcal{L}$ be a geometric lattice and consider two elements of $S(\mathcal{L})$: $\mathbf{X} = (X_1, X_2, \ldots, X_k)$ and $\mathbf{Y} = (Y_1, Y_2, \ldots, Y_\ell)$. Then, in the semigroup poset $S(\mathcal{L}),$  $\mathbf{X} \leq \mathbf{Y}$ if and only if $\mathbf{Y}$ is a \textit{prefix} of $\mathbf{X},$ meaning $k \geq \ell$ and $X_i = Y_i$ for $i \leq \ell.$ Further, the Hasse diagrams of $S(\mathcal{L})$ and their contractions $S(\mathcal{L})_{\geq X}\cong S([\hat 0_\mathcal L,X])$ are rooted trees.
\end{prop}
\begin{proof}
Let $\mathbf{X} = (X_1, X_2, \ldots, X_k)$ and $\mathbf{Y} = (Y_1, Y_2, \ldots, Y_\ell)$ be elements of $S(\mathcal{L}).$ If $k \geq \ell$ and $X_i = Y_i$ for $i \leq \ell,$ then it is clear from the multiplication definition that $ \mathbf{Y}\mathbf{X} = \mathbf{X},$ meaning that $\mathbf{X} \leq \mathbf{Y}.$ On the other hand, if $ \mathbf{Y}\mathbf{X} = \mathbf{X},$ then since multiplication can only increase the length of a chain, $k \geq \ell.$ Furthermore, since the first $\ell$ letters in the chain $\mathbf{Y}\mathbf{X}$ belong to $\mathbf{Y},$ the first and second claims follow.
\end{proof}

\subsection{Classes of LRBs}\label{sec:classes_lrbs}
We now introduce a few different classes of LRBs needed to explain our results. 
\subsubsection{Connected LRBs}
An LRB $B$ is said to be \textbf{connected} if the Hasse diagrams of the contractions $B_{\geq X}$ for all $X \in \supp(B)$ are connected as graphs. By \cref{prop:contractions-arrangements-polytopes}, \cref{prop:contractions-of-catzero-lrbs}, and \cref{prop:matroids-are-rooted-trees}, our primary examples---semilattices, hyperplane arrangement face monoids, $\catzero$-cube complex LRBs, and geometric lattice LRBs ---are all connected LRBs. It turns out that most of the LRBs in the literature are also connected, including those associated to affine real hyperplane arrangements and complex hyperplane arrangements. A simple example of a \textit{disconnected} LRB is given by an LRB defined on $n>1$ elements, with multiplication $xy = x$ for all elements $x,y,$ whose semigroup poset consists of $n$ isolated points.

Since LRBs are \textit{semigroups} rather than monoids, they may not always contain an identity element; for instance a $\catzero$-cube complex LRB will have an identity element if and only if it is a single $n$-cube and its faces. However, this does not necessarily preclude the associated semigroup algebra from being unital. Margolis--Saliola--Steinberg determined when an LRB semigroup algebra is unital.

\begin{prop}[{\cite[Theorem 4.15]{MSS}}]\label{prop:connected-unit}
    Let $\kk$ be a commutative unital ring, and let $B$ be an LRB. Then, the semigroup algebra $\kk B$ is unital if and only if $B$ is connected.
\end{prop}
So, despite rarely having an identity, $\catzero$-cube complex LRBs always have unital semigroup algebras.

\subsubsection{CW LRBs}

In \cite{browndiac}, Brown and Diaconis used the algebraic topology of the zonotope associated to a hyperplane arrangement to prove the diagonalizability of random walks on the chambers of the arrangement. Their constructions ended up playing a key role in Saliola's work on the representation theory of the face semigroup algebra of a hyperplane arrangement \cite{saliolaquiverdescalgebra, saliolafacealgebra}. 

Later on, Margolis--Saliola--Steinberg generalized Saliola's work to a large class of LRBs which have analogous CW complexes to draw from, which they call CW LRBs \cite{MSS}. In terms of representation theory, the most seems to be known for this class of LRBs; see, for instance, \cite[Theorem 1.1, Theorem 9.7]{MSS}.

\subsubsection*{CW complexes and CW posets}
Before explaining CW LRBs, we give a brief overview of CW complexes, regular CW complexes, and CW posets. We mainly follow the exposition of \cite{LundellWeingram, hatcher, MSS, bjorner}; however, our definition will only consider \textit{finite-dimensional} CW complexes.

For $i \geq 1,$ let an $\bf{i}$\textbf{-cell} be an $i$-dimensional open ball. An ${n}$-dimensional \textbf{\bf{CW} complex ${X}$} is defined to be a topological space of the form $X^n$, where ${X^i}$ is defined inductively as follows. 
\begin{itemize}
    \item[(i)] $X^0$ is a discrete set of points, called $0$-cells.
    \item[(ii)] $X^i$ is formed from $X^{i - 1}$ by attaching $i$-cells, $e_\alpha^i$. Each $i$-cell $e_\alpha^i$ is glued via a map of the form $\varphi_\alpha\colon S^{i - 1} \to X^{i - 1}$, where $S^{i - 1}$ is the $(i-1)-$dimensional sphere. One should think of $\varphi_{\alpha}(S^{i - 1})$ as the boundary of the open $i$-cell $e_\alpha^i$. Letting ${E_\alpha^i}$ be the closed $i$-dimensional Euclidean ball, we can consider $X^i$ as a quotient space $X^i = \left(X^{i - 1} \sqcup_{\alpha \in A} E_\alpha^i \right)/ {\sim}$, where $x \sim \varphi_\alpha(x)$ for $x$ on the boundary of $E_\alpha^i$. Setwise,  $X^i = X^{i - 1} \sqcup_{\alpha \in A} e_\alpha^i$.
\end{itemize}

 A \textbf{characteristic map} of a cell $e_\alpha^i$ in a CW complex $X$ is a map $\Phi_\alpha\colon E_\alpha^i \to X$ which (i) extends $\varphi_\alpha$ and (ii) is a homeomorphism from the interior of the ball $E_\alpha^i$ to $e_\alpha^i$. In particular, one can always choose
 \begin{equation}\label{eqn:char_map_choice}
    \Phi_\alpha\colon E_\alpha^i \hookrightarrow X^{i - 1} \bigsqcup_{\alpha'\in A} E_{\alpha'}^i \to X^{i} \hookrightarrow X,
\end{equation} where the first and third maps are inclusions and the middle map is the projection map \[X^{i - 1}\bigsqcup_{\alpha' \in A}E_{\alpha'}^i \to X^i = \left(X^{i - 1} \bigsqcup_{\alpha'\in A} E_{\alpha'}^i \right)/ {\sim}.\] We call the images $\Phi_\alpha(E_\alpha^i)$ \textbf{closed cells}.

A \textbf{regular $\bf{CW}$ complex} is a CW complex where the characteristic maps can be chosen to be embeddings (or equivalently, homeomorphisms onto their images). Note that this condition implies that the attaching maps $\varphi_\alpha$ must be embeddings and that the closed cells $\Phi_\alpha(E_\alpha^i)$ are homeomorphic to closed balls $E^i$. 

Given a CW complex $X$, let ${\mathcal{P}(X)}$ be the \textit{\textbf{face poset}} of $X$, defined to be the poset of closed cells ordered under inclusion. Recall that our definition of a face poset does \textit{not} include an empty face. There is a close relationship between regular cell complexes and the order complexes of their face posets. Namely, for a regular CW complex $X$, the geometric realization $\|\Delta(\mathcal{P}(X))\|$ is homeomorphic to $X$ and the closed balls are the geometric realizations of principal lower sets (\cite[Proof of Theorem 1.7]{LundellWeingram}.

Bj\"{o}rner \cite{bjorner} studied the posets which are face posets of regular CW complexes. Recall that for a poset $P$ and element $y \in P,$  
the subposet $P^{<y}:= \{x \in P: x < y\}.$ A \textbf{{CW poset}} is a \textit{nonempty}  poset $P$ such that for all $y \in P,$ the geometric realization of the order complex $\|\Delta(P^{<y})\|$ is homeomorphic to a sphere. (Note that we consider the order complex of the empty poset $P^{<y}$ for some minimal element $y \in P$ to be homeomorphic to the ``$(-1)-$sphere.'')  A CW poset is automatically graded. 

The close relationship between regular CW complexes and CW posets is as follows.
 \begin{thm}[{Bj\"{o}rner, \cite[Proposition 3.1]{bjorner}}]
     A poset $P$ is a CW poset if and only if it is the face poset of a regular CW complex.
 \end{thm}

 Bj\"{o}rner's proof involves creating a regular CW complex $\Sigma(P)$ from a CW poset $P$ whose underlying space is $\|\Delta(P)\|$ with closed $n$-cells of the form $\|\Delta(P^{\leq p})\|$  and open $n$-cells $\|\Delta(P^{\leq p})\|\setminus \|\Delta(P^{<p})\|$.  He mentions \textit{en passant} that $\Sigma$ is a functor from CW posets to regular CW complexes with respect to special maps of CW posets and regular CW complexes termed \textit{admissible} and \textit{semiregular}, respectively. Margolis--Saliola--Steinberg refined this functoriality to show that $\Sigma$ is a functor with respect to poset maps they termed \textit{cellular} and \textit{regular cellular} maps (\cite[Lemma 3.5]{MSS}); we will discuss the notion of cellular maps on CW posets in \cref{sec:CWLRBs}.

Finally, we record a lemma that will be useful later on, which appears in \cite[Lemma V.4.1]{LundellWeingram}.

\begin{lemma}\label{lem:cw-poset-rank2-ints}
    Every open rank-two interval $(x, y)$ in a CW poset $P$ has exactly two elements. Additionally, if $\rk(y) = 1$, then $|\{p \in P: p < y\}| = 2.$
\end{lemma}

\subsubsection*{Definition of CW LRBs}
In \cite{MSS}, Margolis--Saliola--Steinberg unite the theory of many LRBs in the literature by introducing the concept of a CW LRB. An LRB $B$ is a \textbf{$\mathbf{CW}$ LRB} if for all supports $X \in \supp(B)$, the contraction $B_{\geq X}$ is a CW poset. 

Since convex polytopes and cubical complexes are regular CW complexes, we have from \cref{prop:contractions-arrangements-polytopes} and \cref{prop:contractions-of-catzero-lrbs} that the face monoids of real, central hyperplane arrangements and $\catzero$-cube LRBs are both examples of CW LRBs. It turns out that many more of the LRBs studied in the literature are CW LRBs. For example, the face semigroups of \textit{affine} arrangements, (affine) oriented matroids, $\mathrm{COM}$s, face monoids of complex hyperplane arrangements, oriented interval greedoids, and $\catzero$-zonotopal complexes are all CW left regular bands; see \cite[Proposition 3.16]{MSS}. In contrast, the free LRB and other LRBs associated to nontrivial geometric lattices are \textit{not} examples of CW LRBs. 

Importantly, Margolis--Saliola--Steinberg prove that the support semilattice $\supp(B)$ of a CW LRB is always graded (just like its semigroup poset must be). 

\begin{thm}[{\cite[Theorem 6.2 and Theorem~7.14]{MSS}}]\label{thm:graded-support-poset-cw}
    For a CW LRB $B$, the support semilattice $\supp(B)$ is graded, with each open interval a Cohen-Macaulay poset.
\end{thm}

\subsubsection{Hereditary LRBs}

A \textbf{$\kk$-hereditary LRB} is a \textit{connected} LRB whose semigroup algebra $\kk B$ is \textit{hereditary}, which is a concept we will explain in \cref{sec:Elementary-algebras-quivers-hereditary}. It is a consequence of~\cite[Corollary 5.22]{MSS} that hereditary left regular bands admit the following combinatorial characterization.

\begin{thm}[Margolis--Saliola--Steinberg]
    Let $B$ be a connected left regular band and $\kk$ a field. Then, $\kk B$ is hereditary if and only if for each $X < Y \in \supp(B),$ the connected components of the realization of the order complex $\|\Delta(B_{\geq X}^{<y})\|$ are acyclic over $\kk$, meaning their reduced homology groups are all zero over $\kk$. Here, $y$ is allowed to be any element in $B$ with support $Y$.
\end{thm}

Margolis--Saliola--Steinberg (see \cite[Corollary 5.24]{MSS}) proved the following, generalizing Brown's result for the free LRB \cite[Theorem 8.1]{saliolaquiverlrb}.

\begin{thm}[Margolis--Saliola--Steinberg]
    Let $\kk$ be a field. If the semigroup poset of an LRB $B$ is a rooted tree, then $\kk B$ is a hereditary algebra.
\end{thm}

By \cref{prop:matroids-are-rooted-trees}, the LRB $S(\mathcal{L})$ associated to any geometric lattice $\mathcal{L}$ is a rooted tree and is thus 
$\kk$-hereditary  for every field $\kk$. 

\section{Background: Representation theory of LRB algebras}
\label{ch:fin-dim-algebras}

Throughout this section, we assume that $\kk$ is a field and that $B$ is a connected LRB.

\subsection{Representation theory of finite dimensional algebras}\label{subsec:rep-theory-fin-dim-algebras} We give a very brief summary of the finite dimensional algebra theory we will need, and recommend the references \cite[Appendix $D$]{Aguiar-Mahajan}, \cite[Ch. 1]{assem}, \cite[Ch. 9]{etingof}, 
and \cite[Ch. 6, 7]{webbrepntheory} for more details.

 A \textbf{$\kk$-algebra $A$}, or just an algebra,
 is a $\kk$-vector space with an associative bilinear multiplication. We will assume throughout that $A$ is a \textbf{finite dimensional} unital algebra, unless otherwise stated. A \textbf{representation} of an algebra $A$ is a finite dimensional (left) $A$-module; we will use the language of modules and representations interchangeably. 

A nonzero $A$-module is \textbf{indecomposable} if it cannot be written as a nontrivial direct sum of two $A$-submodules and it is \textbf{simple} if it has no nontrivial $A$-submodules. An $A$-module $U$ is \textbf{projective} if there is a free $A$-module $F$ and another $A$-module $U'$ such that $F \cong U \oplus U'$ as $A$-modules.

The \textbf{idempotents} of $A$ play a key role in its representation theory. A family of nonzero idempotents $\{e_1, e_2, \ldots, e_n\}$ of $A$ is \textbf{complete} if $e_1 + e_2 + \cdots + e_n = 1$; \textbf{orthogonal} if $e_ie_j = 0$ for all $i \neq j$; and \textbf{primitive} if for all $i,$ $e_iAe_i$ contains only the idempotents $0,e_i$, or equivalently  is a local ring. 

We will be particularly interested in \textbf{complete families of primitive, orthogonal idempotents}, which we call \textbf{cfpois} for brevity. Cfpois $\{e_1, e_2, \ldots, e_n\}$ of $A$ biject with decompositions of $A$ into indecomposable left $A$-modules via $
    A = \bigoplus_{i = 1}^n A e_i$.
 Up to re-indexing and $A$-module isomorphism, the decomposition of $A$ into indecomposable modules is unique.  So, for any two cfpois $\{e_i: 1 \leq i \leq n\} $ and $\{f_j: 1 \leq j \leq m\}$ of $A,$ we know that $m = n$ and there exists at least one reordering $\sigma \in S_n$ such that $Ae_i \cong Af_{\sigma(i)}$ as $A$-modules.  Further, up to isomorphism, these are the only projective indecomposable $A$-modules, and so we call them the \textbf{projective indecomposable modules of $A$}.

We will use the following fact several times. 
 \begin{remark}\rm\label{rem:conjugation-idems} 
The cfpois form a single class orbit under conjugation by $A^\times$, where $A^\times$ is the group of units of $A$ (see~\cite[Exercise~21.17]{Lam}).
Conjugating any cfpoi of $A$ by a unit of $A$ produces another cfpoi of $A$. Conversely, any two cfpois of $A$ are conjugate by some unit of $A$.
\end{remark}

The \textbf{radical} \textbf{$\rad(A)$} of $A$ is the intersection of all of its maximal left ideals, which is a nilpotent, two-sided ideal of $A$. 
More generally, the \textbf{radical of an $A$-module} $U,$ written \textbf{$\rad(U)$}, is the intersection of all its maximal submodules. One can equivalently define it as $\rad(A) U$. Up to isomorphism, the set of simple $A$-modules is $\left\{Ae_i/ \rad(A)e_i: 1 \leq i \leq n \right\}$ for any choice of cfpoi $\{e_i\}.$ The projective indecomposable module $Ae_i$ is a \textbf{projective cover} for the simple module $Ae_i / \rad(A)e_i$. Furthermore, $Ae_i/ \rad(A e_i) \cong A e_j / \rad(A) e_j$ if and only if $Ae_i \cong Ae_j.$ In general, it is possible that $Ae_i \cong Ae_j$ for $i \neq j$.  The multiplicity of $Ae_i$ as a summand in $A$ is the same as that of $Ae_i/\rad(A)e_i$ in $A/\rad(A)$. We shall frequently use that an idempotent $e\in A$ is primitive if and only if $e+\rad(A)$ is primitive.

 An algebra $A$ is \textbf{semisimple} if every $A$-module can be decomposed into a direct sum of simple $A$-modules. By the Wedderburn--Artin theorem, $A$ is semisimple if and only if $\rad(A) = 0$. Maschke's theorem implies that group algebras $\kk G$ for finite groups $G$ are semisimple precisely when $\Char(\kk)\nmid |G|$. 

The following well-known fact is standard, so we omit its straightforward proof.
 
\begin{prop}\label{prop:yoneda}
    Let $A$ be an algebra and let $e,f,f'$ be idempotents.   Then there is a vector space isomorphism $\hom_A(Ae,Af)\cong eAf$ given by the mutually inverse maps
   \begin{gather*}
    \varphi\longmapsto \varphi(e)=e\varphi(e)f\\
    (x\mapsto xeaf) \longmapsfrom eaf
    \end{gather*}
    Moreover, under this isomorphism, if $\varphi\colon Ae\to Af$ and $\psi\colon Af\to Af'$, then \[\psi\circ \varphi\longmapsto \varphi(e)\psi(f) = e\varphi(e)f\psi(f)f',\] since $\psi(\varphi(e)) = \psi(\varphi(e)f) = \varphi(e)\psi(f).$
\end{prop}

When $A$ is semisimple, one might aim to understand an $A$-module by understanding its decomposition into a direct sum of simple modules. For non-semisimple algebras, one needs an alternative approach; we will take the approach of counting \textit{composition multiplicities}. A \textbf{composition series} of a finite dimensional $A$-module $U$ is a sequence of $A$-modules $0 \subsetneq U_1 \subsetneq U_2 \subsetneq \ldots \subsetneq U_{n - 1} \subsetneq U$ such that each successive quotient $U_i / U_{i - 1}$ (called a \textbf{composition factor}) is a simple $A$-module. The Jordan--H\"{o}lder theorem guarantees that up to isomorphism, the composition factors and their multiplicities are independent of the composition series. The \textbf{composition multiplicity} $[U: M]$ of a simple $A$-module $M$ in an $A$-module $U$ is the multiplicity of the isomorphism class of $M$ as a composition factor in any composition series of $U$.
 
 If a cfpoi for $A$ is well-understood, one can theoretically compute the composition multiplicities of an $A$-module without constructing an explicit composition series.
Specifically, let $U$ be a finite dimensional $A$-module, fix $M = Ae / \rad(A)e$ to be a simple $A$-module, and let $P = Ae$ be its projective cover. 
Then, the composition multiplicity of $M$ in $U$ is given by
\begin{align}\label{eqn:composition-mults}
     [U: M] = \dim_\kk \mathrm{Hom}_A\left(P, U\right)/\dim_\kk D = \dim_\kk \mathrm{Hom}_A\left(Ae, U\right)/\dim_\kk D = \dim_\kk eU/\dim_\kk D
 \end{align}
where $D=\End_A(M)$, cf.~\cite[Chap.~II, Exer.~1]{ARS}. 

 A particularly important family of composition multiplicities  to understand is the composition multiplicities of the projective indecomposable $A$-modules. Let $\{M_\alpha: \alpha \in \mathcal{I}\}$ and $\{P_\alpha: \alpha \in \mathcal{I}\}$ be the isomorphism classes of simple $A$-modules and their corresponding projective indecomposables. The composition multiplicities $[P_\beta: M_\alpha]$ as $\alpha, \beta$ vary in $\mathcal{I}$ are called the \textbf{Cartan invariants} of $A.$ By \cref{eqn:composition-mults}, 
\begin{align}\label{eqn:Cartan-Invariants.gen}
     [P_\beta: M_\alpha] = \dim_{\kk} e_\alpha A e_\beta/\dim_\kk \End_A(M_\alpha).
 \end{align}

\subsection{Elementary algebras, quivers, and hereditary algebras}\label{sec:Elementary-algebras-quivers-hereditary}
A finite-dimensional $\kk$-algebra $A$ is an \textbf{elementary algebra} if as algebras, $A / \rad(A) \cong \kk^n$ for some $n$; this is equivalent to every simple $A$-module being one-dimensional. Every finite dimensional algebra over an algebraically closed field is Morita-equivalent to a unique elementary algebra. 
  Note that if $A$ is elementary, then since $A/\rad(A)\cong \kk^n$ has $n$ non-isomorphic simple summands, $A$ is a direct sum of $n$ pairwise non-isomorphic projective indecomposable modules.

If $A$ is elementary, then $\End_A(S)=\kk$ for every simple $A$-module $S$, and so \cref{eqn:Cartan-Invariants.gen} simplifies to 
\begin{align}\label{eqn:Cartan-Invariants}
     [P_\beta: M_\alpha] = \dim_{\kk} e_\alpha A e_\beta
 \end{align}
which we shall use throughout.

A fact which will be useful to us is that the invariant subalgebras of elementary algebras are themselves also elementary algebras.

\begin{prop}[{\cite[Lemma D.46]{Aguiar-Mahajan}}]\label{prop:inv-sub-of-elementary}
Let $G$ be a finite group acting by algebra automorphisms on an elementary $\kk$-algebra $A$. Then, the invariant subalgebra $A^G$ is elementary with $\rad(A^G) = \left(\rad(A)\right)^G.$ If $\Char (\kk) \nmid |G|$, then $A^G / \rad(A^G) = \left(A / \rad(A)\right)^G.$
\end{prop}

 Elementary algebras have close ties with the theory of quivers, and the link between \textit{hereditary} elementary algebras and quivers is especially strong. An algebra $A$ is \textbf{hereditary} if each of its left ideals is a projective $A$-module.
 A \textbf{quiver} is a (finite) directed graph, with loops and multiple edges permitted. Given a field $\kk$ and a quiver $\mathcal{Q},$ the \textbf{path algebra} $\kk \mathcal{Q}$ is the $\kk$-algebra with $\kk$-basis consisting of (directed) \textit{paths} $e_k\cdots e_2e_1$, read right-to-left,  \[v_{i_k} \xleftarrow{\,e_k\,} v_{i_{k-1}}\longleftarrow \cdots \longleftarrow v_{i_1} \xleftarrow{\,e_1\,} v_{i_0}\] in $\mathcal{Q}$ including an empty path at each vertex.  Multiplication $\kk$-linearly extends path composition
    \[(f_\ell\cdots f_1)(e_k\cdots e_1) = \begin{cases} f_\ell\cdots f_1\cdot e_k\cdots e_1, & \text{if}\ f_1e_k\ \text{is a path}\\0, & \text{else.}\end{cases}\]
Note that $\kk \mathcal Q$ is finite dimensional if and only if $\mathcal Q$ is acyclic.

A (two-sided) ideal $I \subseteq \kk \mathcal{Q}$ is called an \textbf{admissible ideal} if $J^n \subseteq I \subseteq J^2$ for some $n \geq 2$, where $J$ is the (two-sided) ideal generated by the arrows of $\mathcal{Q}$. 

\begin{thm}[Gabriel]\label{thm:quiver-elementary}
Let $A$ be an elementary algebra. Then, there is a unique quiver $\mathcal{Q}(A)$ for which 
\[A \cong \kk \mathcal{Q}(A) / I,\] for some (potentially non-unique) admissible ideal $I$. Moreover, $A$ is hereditary if and only if $I=0$, i.e., $A\cong \kk \mathcal Q(A)$.
\end{thm}

The quiver $\mathcal{Q}(A)$ from \cref{thm:quiver-elementary} turns out to always have the property that there is an algebra surjection $\rho\colon \kk \mathcal{Q}(A) \to A$ for which:
\begin{itemize}
    \item the vertices are mapped bijectively by $\rho$ to your chosen cfpoi $\{e_i: i \in \mathcal{I}\}$ of $A$,
    \item letting $v_{e_i}$ be the vertex which maps to $e_i$ under $\rho$, the edges of the form $v_{e_j} \leftarrow v_{e_i}$ are mapped by $\rho$ to a basis of  $e_{j}\left[\rad(A) / \rad^2(A)\right]e_i$,
    \item the kernel of $\rho$ is $I$,    
    \item the arrow ideal $J$ is mapped \textit{onto} $\rad(A)$ by $\rho$.
\end{itemize}

\subsection{Primitive idempotents of LRB algebras}\label{sec:rep-theory-of-lrb-algebras}
We continue to assume that $B$ is a connected LRB. 
\subsubsection{Brief summary of general theory}\label{sec:general-lrb-rep-theory}

The structure of an LRB algebra is simplest for a meet-semilattice $\Lambda$: Solomon~\cite{SolomonBurnside} established an isomorphism $\kk \Lambda\to \kk^{\Lambda}$ mapping $X\in \Lambda$ to the indicator function $\delta_{\Lambda^{\leq X}}$ of $\Lambda^{\leq X}$.  The primitive idempotents of $\kk^{\Lambda}$ are given by the indicator functions $\delta_X$, with $X\in \Lambda$, of singletons.  For a \textit{general} connected LRB $B$, the representation theory of $\kk B$ is a fair bit more complex. However, $\kk B$ still inherits a lot of nice structure from its split-semisimple cousin $\kk \supp(B).$ For instance, the radical of $\kk B$ is completely understood, as explained by the following proposition (see \cite[Corollary 4.12]{MSS}).
\begin{prop} \label{prop:split-semisimple-quotient}
    The radical of $\kk B$ is $\ker  \sigma$. Consequently, $\kk B / \rad(\kk B)\cong \kk \supp(B)\cong \kk^{\supp(B)}$, and so $\kk B$ is an elementary algebra. Moreover, the irreducible character of $B$ associated to $X\in \supp(B)$ (i.e., corresponding to the primitive idempotent $\delta_X\in \kk^{\supp(B)}$) is the indicator function $\varepsilon_X\colon B\to \kk$ of $B_{\geq X}$. 
\end{prop}

\begin{remark}[Indexing of cfpois]\label{rem:indexing-cfpois}\rm
   \cref{prop:split-semisimple-quotient} implies that the members of any cfpoi for $\kk B$ biject with the elements of $\supp(B).$ Given such a cfpoi $\{E_X : X \in \supp(B)\}$ for $\kk B$, we write $P_X$ for the projective indecomposable module $\kk B\cdot E_X$ and $M_X$ for its associated simple module $(\kk B\cdot E_X) / \rad(\kk B) E_X.$ Since $\kk B$ is elementary, we have that as $\kk B$-modules $P_X \ncong P_Y$ for $X \neq Y.$  
  
When we talk about a general cfpoi of $\kk B$, we shall \textbf{always} assume the indexing is such that $M_X$ has character $\varepsilon_X$. Further, since any two cfpois for $\kk B$ are conjugate by an invertible element of $\kk B$ by \cref{rem:conjugation-idems}, and conjugate idempotents generate isomorphic $\kk B$-modules, we know that for any two cfpois $\{E_X: X \in \supp(B)\}$ and $\{F_X:X \in \supp(B)\}$ of $\kk B$, there exists an invertible element $z \in \kk B$ so that for all $X \in \supp(B)$, $F_X = zE_Xz^{-1}$ with our indexing convention.
\end{remark}

Saliola computed the Cartan invariants of $\kk B$ for any LRB \textit{monoid} $B$ in \cite{saliolaquiverlrb} in terms of the M\"{o}bius function of the support semilattice $\supp(B)$. These formulas were generalized to all connected LRBs in \cite[Theorem 4.18]{MSS}. The Cartan invariants of connected CW LRBs are especially nice, as proved by Margolis--Saliola--Steinberg in \cite[Theorem 9.7]{MSS} (the case of hyperplane arrangements was done earlier by Saliola~\cite{saliolafacealgebra}):
\begin{thm}[Margolis--Saliola--Steinberg]\label{t:cartan.lrb} Let $B$ be a connected CW LRB. Then, the Cartan invariants $[P_X: M_Y]$ of $\kk B$ are zero unless $X \leq Y$, in which case $[P_X:M_Y] = \dim_{\kk}E_Y \cdot \kk B \cdot  E_X = |\mu_{\supp(B)}(X, Y)|.$
\end{thm}
 
Later on, we will prove an equivariant analogue of the equality above; see \cref{cor:CWLRBs-as-poset-topology}.

More information, such as projective resolutions of the simple $\kk B$-modules and descriptions of the Ext-spaces between any pair of simples, can be found in~\cite{MSS}.

\subsubsection{Idempotents that behave well with symmetry}\label{sec:general-lrb-symmetry-idempotents}

In this section, we show that when the characteristic of $\kk$ does not divide $|G|,$ there exist cfpois for $\kk B$ whose members are permuted by $G.$ We do not claim any originality here, as our construction essentially copies and merges the group-equivariant work of Saliola on face algebras of reflection arrangements in \cite{saliolaquiverdescalgebra} and the idempotents of LRBs of Margolis--Saliola--Steinberg in \cite{MSS}, with some minor streamlining.  

Let $\theta\colon \kk B\to \kk^{\Lambda(B)}$ be the composition of $\sigma$ with the Solomon isomorphism, so that $\theta(b)=\delta_{\supp(B)^{\leq \sigma(b)}}$. In particular, $\ker\theta=\ker \sigma=\rad(\kk B)$. Note that 
\begin{equation}\label{eqn:theta}
\theta(b)\delta_X=\delta_{\Lambda(B)^{\leq \sigma(b)}}\delta_X = \varepsilon_X(b)\delta_X
\end{equation}
for $b\in B$ and $X\in\supp(B)$.

We shall use throughout that $\rad(\kk B)$ contains no nonzero idempotents, being nilpotent.  We shall also use the well-known facts that an idempotent $e$ of a finite dimensional algebra $A$ is primitive if and only if $e+\rad(A)$ is primitive (see~\cite[Proposition~21.22]{Lam}), and that a family of orthogonal idempotents of $A$ is complete if and only if its image modulo the radical is complete (easily checked).

\begin{prop} \label{cor:supports-of-kb-idemps}
   For $X\in \supp(B)$, let $e_X\in \kk B$ be any idempotent (necessarily primitive) such that $\theta(e_X)=\delta_X$, i.e., $\kk Be_X/\rad(\kk B)e_X\cong M_X$.  Then the following hold.
    \begin{enumerate}
        \item[(i)] If $b \in B$ with $\sigma(b) \ngeq X$, then $be_X = 0.$
        \item[(ii)]
        The idempotent $e_X \in \mathrm{span}_\kk \{b \in B: \sigma(b) \leq X\}$. 
        \item[(iii)]  If $Y \ngeq X,$ and $e_Y$ is a primitive idempotent with $\kk Be_Y/\rad(\kk B)e_Y\cong M_Y$, then  $e_Y\kk Be_X = 0.$ 
        \item[(iv)] Writing $e_X = \sum_{b \in B}k_bb,$ we have that \[\sum_{b: \sigma(b) = X}k_b = 1.\]
        \item[(v)] If $b \in B$ with $\sigma(b) = X$, then there are constants $c_{b'} \in \kk$ for which, \[be_X = b + \sum_{b': \sigma(b') < X}c_{b'} b',\] 
        and  $b'<b$ whenever $c_{b'}\neq 0$.
    \end{enumerate}
\end{prop}

\begin{proof}
For (i), note that $\theta(be_X)=\delta_{\supp(B)^{\leq \sigma(b)}}\cdot \delta_X=0$, and so $be_X\in \rad(\kk B)$. On the other hand, left multiplication by $b$ is a ring homomorphism $\kk B\to \kk bB$, so $be_X$ is idempotent.  Thus $be_X=0$.   Turning to (ii), choose $b\in B$ with $\sigma(b)=X$. Then since $\{b' \in B: \sigma(b') \leq X\} =Bb=BbB$, it suffices to show that $e_X\in \kk BbB$.   Since $be_X$ is idempotent, $(e_Xbe_X)^2=e_X(be_X)^2=e_Xbe_X$.  Therefore $e_X-e_Xbe_X$ is an idempotent.  But $\theta(e_X-e_Xbe_X) = \delta_X-\delta_X\cdot \delta_{\supp(B)^{\leq X}}\cdot \delta_X=0$.  Therefore, $e_X-e_Xbe_X\in \rad(\kk B)$, and so $e_X=e_Xbe_X\in \kk BbB$.  Item (iii) is immediate from (ii) applied to $e_Y$ and (i).  Item (iv) follows from (ii) because $\theta(e_X)=\delta_X$ implies $\varepsilon_X(e_X)=1$ by \cref{eqn:theta}. Finally, to prove (v), note that $e_X=f+(e_X-f)$ where $f\in \mathrm{span}_\kk \{a \in B: \sigma(a) = X\}$.  One verifies directly from (iv) that $bf=b$ and from (ii) that $b(e_X-f)\in \mathrm{span}_\kk \{b' \in bB: \sigma(b') <X\}$, from which (v) follows.
\end{proof}

We now establish a sort of converse to the previous proposition, with some equivariance thrown in.

\begin{prop} \label{prop:construct-idems}
For each support $X \in \supp(B),$ fix a $\kk$-linear combination $\ell_X$ of elements of $B$ with support $X$ whose coefficients sum to $1$; in particular the $\ell_X$ are idempotent.  Define $E_{\widehat{0}} = \ell_{\widehat{0}}.$ Inductively define \[E_X:= \ell_X - \sum_{Y < X}\ell_X E_Y.\]
The following  statements hold. \begin{itemize}
    \item[(i)]  $E_X$ is an idempotent with $\theta(E_X)=\delta_X$. 
    \item[(ii)] The family $\{E_X: X \in \supp(B) \}$ forms a cfpoi for $\kk B$.
\end{itemize}
Moreover, if $B$ admits an action via automorphisms by a finite group $G$ such that $g(\ell_X) = \ell_{g(X)}$ for each $g \in G$ and $X \in \supp(B)$, then:
\begin{itemize}
    \item[(iii)] For all $g \in G$ and $X \in \supp(B)$, $g(E_X) = E_{g(X)}$.
\end{itemize}
\end{prop}
\begin{proof}  We prove (i) and (ii) together.  Let $B'$ be obtained from $B$ by adjoining an identity $1$  and note that $\supp(B') = \supp(B)\cup \{\hat 1\}$ with $\hat 1$ a maximum.  Set $\ell_{\hat 1}=1$.  

We prove by induction on the length of the shortest chain from $\hat{0}$ to $X$ that $\{E_Y: Y\leq X\}$ consists of pairwise orthogonal idempotents and $\theta(E_X)=\delta_X$. Clearly, $\theta(\ell_{\hat 0})=\delta_{\hat 0}$ and $\ell_{\hat 0}$ is idempotent.  Now, let $\hat 0\neq X \in \supp(B')$ and assume by induction the statement holds for all $Y \in \Lambda(B')$ with $Y < X$.   
Note that $\kk B'\ell_X = \mathrm{span}_\kk \{b \in B': \sigma(b) \leq X\}$. If $Y<X$, then  $E_Y \in \mathrm{span}_\kk \{b \in B': \sigma(b) \leq X\}$ by the inductive hypothesis and \cref{cor:supports-of-kb-idemps}(ii), so we can rewrite $E_Y$ as $E_Y = E_Y\ell_X$.  Thus if $Y_1,Y_2<X$, then 
\[\ell_XE_{Y_1}\ell_X E_{Y_2} = \ell_X E_{Y_1}E_{Y_2}.\]  If $Y_1\geq Y_2$, then by induction we have $\ell_X E_{Y_1}E_{Y_2}=\delta_{Y_1,Y_2}\cdot \ell_X E_{Y_1}$ (Kronecker $\delta$).   If $Y_1\ngeq Y_2$, then $\ell_X E_{Y_1}E_{Y_2}=0$ by applying  \cref{cor:supports-of-kb-idemps}(iii) and induction.  It follows that $\sum_{Y<X}\ell_XE_Y$ is an idempotent in $\ell_X \kk B'\ell_X$.  Therefore, $E_X = \ell_X - \sum_{Y < X} \ell_X E_Y$ is an idempotent orthogonal to all $\ell_XE_Y$ with $Y<X$.  Recall however that if $Y_1,Y_2\leq X$, then $E_{Y_1}E_{Y_2} = E_{Y_1}\ell_X E_{Y_1}\ell_X E_{Y_2}=\delta_{Y_1,Y_2}\cdot E_{Y_1}$, and so we deduce that $\{E_Y:Y\leq X\}$ is a set of orthogonal idempotents. By induction \[\theta(E_X) = \delta_{\supp(B')^{\leq X}} - \sum_{Y<X}\delta_{\supp(B')^{\leq X}}\cdot \delta_Y =\delta_{\supp(B')^{\leq X}} -\sum_{Y<X}\delta_Y = \delta_X.\]

In particular, since $X<\hat 1$ for all $X\in \supp(B)$, we have proved (i), as well as the idempotent and orthogonality claims of (ii). Using the discussion preceding \cref{cor:supports-of-kb-idemps}, the rest of (ii) follows because $\kk B$ is unital, $\{E_X:X\in \supp(B)\}\subseteq \kk B$, each $\theta(E_X)=\delta_X$, $\ker\theta=\rad(\kk B)$, and $\{\delta_X:X\in \supp(B)\}$ is a cfpoi for $\kk^{\supp(B)}$. 
    
Item (iii) is a straightforward induction on the length of the shortest chain from $\hat{0}$ to $X.$   Since $g(\hat 0)=\hat 0$  and $E_{\hat{0}} = \ell_{\hat{0}}=g(\ell_{\hat{0}})=g(E_{\hat 0})$, the base case follows. 
 Now, assume that $X > \hat{0}.$ By induction, for $g \in G$,
    \begin{align*}
        g(E_X) = g(\ell_X) - \sum_{Y < X}g(\ell_X)g(E_Y)
        = \ell_{g(X)} - \sum_{Y < X}\ell_{g(X)}E_{g(Y)} 
        = \ell_{g(X)} - \sum_{Y' < g(X)}\ell_{g(X)}E_{Y'}     = E_{g(X)}.
    \end{align*}
\end{proof}

It will often be very useful for us to use a $\kk$-basis for $\kk B$ which is slightly different from the standard basis $\{b: b \in B\}.$   The following result is essentially~\cite[Corollary 4.4]{MSS}, but allowing any cfpoi.  Nonetheless, the proof only uses the properties in \cref{cor:supports-of-kb-idemps}. 

\begin{prop}[{\cite[Corollary 4.4]{MSS}}]\label{prop:bE_b-form-basis}
 Let $\{F_X: X \in \supp(B)\}$ be any cfpoi for $\kk B.$ The set \[\{bF_{\sigma(b)}: b \in B\}\] forms a basis for the vector space $\kk B$.
\end{prop}

We shall use the following observation later.
\begin{prop}\label{p:deletions}
Let  $\{E_X: X\in \supp(B)\}$ be a cfpoi, $X\in \supp(B)$ and $x\in B$ with $\sigma(x)=X$.  Putting $F_X=\sum_{Y\leq X}E_Y$, one has $F_X\cdot\kk B=F_X\cdot\kk B\cdot F_X\cong \kk B^{\leq x}$ as $\kk$-algebras via the mutually inverse isomorphisms
\begin{gather*}
    F_Xb\longmapsto xF_Xb=xb\\ F_Xb=F_Xxb \longmapsfrom xb. 
\end{gather*}
\end{prop}
\begin{proof}
First note that $F_Xx=F_X$ by \cref{cor:supports-of-kb-idemps}(ii) and that $xF_X = x$ since \[x=x\cdot \sum_{Y\in \supp(B)}E_Y =x\cdot \sum_{Y\leq X} E_Y= xF_X\] by \cref{cor:supports-of-kb-idemps}(i).  Thus, $F_X\cdot \kk B=F_Xx\cdot \kk B=F_Xx\cdot\kk B\cdot x=F_Xx\cdot\kk B\cdot xF_X = F_Xx\cdot\kk B\cdot F_X=F_X\cdot\kk B\cdot F_X$.  It is routine to verify that the two maps are inverse algebra isomorphisms.
\end{proof}
 
 We say a cfpoi $\{E_X: X \in \supp(B)\}$ for $\kk B$ is \textbf{invariant with respect to $G$} if $g(E_X) = E_{g(X)}$ for all $g \in G$ and $X \in \supp(B)$, as in \cref{prop:construct-idems}(iii). So far, our discussion of them has been hypothetical. In the next proposition, we show that so long as the $\Char(\kk) \nmid |G|$, there do exist such cfpois.

\begin{prop} \label{prop:permuted-idempotents}
If $\Char(\kk) \nmid |G|$, then there exists a cfpoi for $\kk B$ which is invariant with respect to $G$. 
\end{prop}
\begin{proof}
By \cref{prop:construct-idems} (iii), it suffices to prove that under these conditions, we can build a family of elements $\{\ell_X: X \in \supp(B)\}$ for which $g(\ell_X) = \ell_{g(X)}$ and each $\ell_X$ is a $\kk$-linear combination of elements in $B$ with support $X$ whose coefficients sum to $1.$  

Fix a set $T$ of orbit representatives for  $\supp(B) / G$.  For $X\in T$,
let $G_X$ be the setwise stabilizer of $X$ and choose an orbit $\mathcal O_X$ from the orbits of $G_X$ on $\sigma^{-1}(X)$.  Note that $|\mathcal O_X|$ is a unit of $\kk$, as it divides $|G|$ by the orbit-stabilizer and Lagrange theorems.  Then \[\ell_X:=\frac{1}{|\mathcal O_X|}\sum_{b\in \mathcal O_X}b\in \kk B\] is a linear combination of elements with support $X$ with coefficients summing to $1$.  Now put $\ell_{g(X)} = g(\ell_X)$. This is well defined because if $g(X)=g'(X)$, then $g^{-1}g'\in G_X,$ and so $g(\mathcal O_X)=g(g^{-1}g')(\mathcal O_X)=g'(\mathcal O_X)$.  Trivially, $g(\ell_X)$ has coefficients summing to $1$ and is in the span of elements of support $g(X)$.  Moreover, $h(\ell_{g(X)}) = hg(\ell_X) = \ell_{hg(X)}$, for $h\in G$, as required.  
\end{proof}

Finally, it is straightforward to check that conjugating a cfpoi which is invariant with respect to $G$ by an invertible element of $(\kk B)^G$ produces another invariant cfpoi---a fact we will use in \cref{ch:invariant_subalgebras}.
\begin{lemma}\label{lem:conjugating-an-invariant-family}
Let $\{F_X: X \in \supp(B)\}$ be a cfpoi for $\kk B$ which is invariant with respect to $G$. Let $z$ be an invertible element of $(\kk B)^G$. Then, $\{zF_Xz^{-1}: X \in \supp(B)\} $ is also an invariant cfpoi with respect to $G$.
\end{lemma}



\section{Invariant subalgebras of LRB algebras}
\label{ch:invariant_subalgebras}
  As mentioned in the introduction, there are several combinatorial algebras which can be realized as invariant subalgebras of semigroup algebras, including Solomon's descent algebra, the Mantaci--Reutenauer algebra, and the random-to-top algebra. Since viewing Solomon's descent algebra within its ambient LRB semigroup algebra has been a useful, we hope that this perspective can be useful for other algebras, too.

  We begin by explaining the representation theoretic information of the invariant subalgebra of an LRB algebra that is inherited from its ambient semigroup algebra in \cref{subsec:general-theory-inv-subalg}. We then explain and prove \cref{intro:thmA} in \cref{sec:rep-theory-invariant-subalgebras}. Finally, we give a necessary and sufficient condition for an orbit-sum of a single element in an LRB to generate the invariant subalgebra in \cref{prop:diag}.

Throughout this section, we assume that $B$ is a connected LRB carrying the action of a finite group $G$  by semigroup automorphisms, and that $\kk$ is a 
field whose characteristic does not divide $|G|.$

\subsection{Representation theory of invariant subalgebras}\label{subsec:general-theory-inv-subalg}
Cfpois for Solomon's descent algebra are well-studied. Initially, they were studied from the perspective of Coxeter groups \cite{Bergerons-hyper1,Bergerons-general,  Bergerons-hyper2, garsia-reutenauer}. Using Bidigare's theorem, Saliola provided a semigroup theoretic perspective by constructing a cfpoi \cite{saliolaquiverdescalgebra} from the cfpois he discovered for the ambient face algebra \cite{saliolaquiverdescalgebra, saliolafacealgebra}. Aguiar--Mahajan further study Saliola's idempotents and generalizations extensively in \cite{Aguiar-Mahajan}.

Refining Saliola's approach, we acquire a fair amount of representation theoretic information on the invariant subalgebra $(\kk B)^G$ from that of $\kk B$. Namely, in this way, cfpois, the projective indecomposable modules, and the simple modules of $(\kk B)^G$ are not challenging to access.   

Let $X$ be a set acted upon by $G$.  Then $G$ acts by automorphisms on $\kk^X$ via $(gf)(x)=f(g^{-1}x)$.  Note that $(\kk^X)^G$ consists of those mappings constant on orbits, and may be identified with $\kk^{X/G}$. The following proposition provides an alternative proof that $(\kk B)^G$ is elementary.
\begin{prop}\label{prop-semisimple-quot}
    As algebras, $(\kk B)^G/\rad((\kk B)^G)\cong (\kk \supp(B))^G\cong (\kk^{\supp(B)})^G\cong \kk^{\supp(B)/G}$.  
\end{prop}
\begin{proof}
    Since $\sigma\colon \kk B \to \kk \supp(B)$ is a $G$-equivariant surjective map with kernel $\rad \left(\kk B\right)$ (by \cref{prop:split-semisimple-quotient}), it induces a homomorphism $\sigma^G\colon (\kk B)^G\to (\kk\supp(B))^G$ with nilpotent kernel, which is surjective since the characteristic of $\kk$ does not divide $|G|$.  
    The Solomon isomorphism  $\kk\supp(B)\to \kk^{\supp(B)}$, $X\mapsto \delta_{\supp(B)^{\leq X}}$, is clearly $G$-equivariant. Therefore, $(\kk \supp(B))^G\cong (\kk^{\supp(B)})^G\cong \kk^{\supp(B)/G}$ is semisimple, whence $\ker\sigma^G=\rad((\kk B)^G)$.
\end{proof}

For a general connected LRB $B$, we show that given our assumptions on $\kk$, the cfpois for the invariant subalgebra $(\kk B)^G$ arise as orbit-sums of the cfpois for $\kk B$ which are invariant with respect to $G$. We remind the reader of our conventions on indexing general cfpois of $\kk B$ by supports in $\supp(B)$ as described in \cref{rem:indexing-cfpois} and of the notion of $G$-invariant cfpois from the discussion preceding \cref{prop:permuted-idempotents}. Also recall $\theta\colon \kk B\to \kk^{\supp(B)}$ which composes $\sigma$ with the Solomon isomorphism.
\begin{lemma} \label{lem:orbit-sums-new-family}
 Let $\{F_X:X \in \supp(B)\}$ be a cfpoi for $\kk B$ which is invariant with respect to $G$.  Then, \[\left\{F_{[X]}:=\sum_{X' \in [X]}F_{X'}: [X] \in \supp(B) / G\right\}\] is a cfpoi for $(\kk B)^G.$
\end{lemma}
\begin{proof}
It is straightforward to check that the given elements are complete, orthogonal idempotents using that the elements $\{F_X:X \in \supp(B)\}$ are. Primitivity follows because $\theta(F_{[X]})$  is the indicator function of the orbit $[X]$, which is primitive in $(\kk^{\supp(B)})^G\cong \kk^{\supp(B)/G}$. 
\end{proof}

\begin{remark}\label{rem:prim}\rm
It is known (cf.~\cite[Proposition~11.1]{saliolaquiverlrb} or~\cite[Proposition~4.5]{MSS}) that $F_X\cdot \kk B\cdot F_X=\kk F_X\cong \kk$ for $x\in \supp(B)$.  Now if $X\in \supp(B)$ and $g\in G$ with $g(X)\neq X$, then $X$ and $g(X)$ are incomparable, and so $F_X\cdot \kk B\cdot F_{g(X)}=0$ by \cref{cor:supports-of-kb-idemps}.  Therefore, if $F_{[X]}=\sum_{X'\in [X]}F_{X'}$, then \[F_{[X]}\cdot \kk B\cdot F_{[X]} = \bigoplus_{X'\in [X]}F_{X'}\cdot \kk B\cdot F_{X'}= \bigoplus_{X'\in [X]}\kk F_{X'}.\] 
The subalgebra $F_{[X]}\cdot \kk B\cdot F_{[X]}$ is invariant under the $G$-action, and $G$ permutes the $\{F_{X'}: X'\in [X]\}$, and so $F_{[X]}(\kk B)^GF_{[X]}= (F_{[X]}\cdot \kk B\cdot F_{[X]})^G= \left(\bigoplus_{X' \in [X]} \kk F_{X'}\right)^G = \kk F_{[X]}\cong \kk$, as $F_{[X]}$ is the sum of the $F_{X'}$.
\end{remark}

\begin{lemma}
\label{prop:idems-for-kbg-sums-of-idems-kb}
Any cfpoi for $(\kk B)^G$ can be written in the form $\left \{F_{[X]}: {[X] \in \supp(B) / G} \right\}$,
so that each \[F_{[X]} = \sum_{X' \in [X]}F_{X'}\] for some cfpoi $\{F_X: X \in \supp(B)\}$ for $\kk B$ which is invariant with respect to $G$ and respects our standard indexing, as described in \cref{rem:indexing-cfpois}.
\end{lemma}

\begin{proof}
By \cref{prop:permuted-idempotents}, there exists some cfpoi $\{E_X: X \in \supp(B)\}$ for $\kk B$ which is invariant with respect to $G$. If we define $E_{[X]}:= \sum_{X' \in [X]}E_{X'},$ then $\{E_{[X]}: [X] \in \supp(B) / G\}$ defines a cfpoi for $(\kk B)^G$ by \cref{lem:orbit-sums-new-family}.

Now, fix some general cfpoi of $(\kk B)^G$. By \cref{rem:conjugation-idems}, there exists an invertible element $z \in (\kk B)^G$ for which our fixed cfpoi is $\{zE_{[X]}z^{-1} : [X] \in \supp(B) / G\}.$ Since 
\[zE_{[X]}z^{-1} = \sum_{X' \in [X]}zE_{X}z^{-1},\] and $\kk B\cdot E_{X} \cong \kk B \cdot zE_{X'}z^{-1}$, setting $F_{X'} := zE_{X'}z^{-1}$ and applying \cref{lem:conjugating-an-invariant-family} completes the proof.
\end{proof}

Let $P$ be a finite poset and assume that a {finite} group $G$ acts on $P$ by poset automorphisms. Then, there is a well-defined poset \textbf{$P / G$} which has as its elements the $G$-orbits $[p]$ of elements of $P$ and the order is defined by $[p] \leq [q]$ if there exists elements $p' \in [p]$ and $ q' \in [q]$ such that $p' \leq q'.$ We record an invariant analogue of the Saliola lemma (\cref{cor:supports-of-kb-idemps}(i)) which will be useful later on.

\begin{lemma}\label{lem:invariant-Saliola-lemma}
Let  $\left\{E_{[X]} : [X] \in \supp(B) / G\right\}$ be a cfpoi for $\left(\kk B\right)^{G}$,  as indexed in \cref{prop:idems-for-kbg-sums-of-idems-kb}. Then, for any $b \in B$ with $[\sigma(b)] \ngeq [X],$ we have  $bE_{[X]}= 0$.
    \end{lemma}
    \begin{proof}
        By \cref{prop:idems-for-kbg-sums-of-idems-kb}, there exists a cfpoi $\{E_{X'}:X' \in \supp(B)\}$ which is invariant with respect to $G$ for  which $E_{[X]} = \sum_{X' \in [X]}E_{X'}.$ By \cref{cor:supports-of-kb-idemps}(i), $b E_{[X]} = \sum_{X' \in [X]}bE_{X'} = 0$.
    \end{proof}
    
\begin{convention}\rm 
For the remainder of this paper, we fix some appropriately indexed cfpoi $\{E_{[X]}: [X] \in \supp(B) / G\}$ for $\left(\kk B\right)^G$ and let $\{E_X: X \in \supp(B) \}$ be a cfpoi for $\kk B$ which is invariant with respect to $G$, matches our standard indexing, and satisfies  $E_{[X]} = \sum_{X' \in [X]}E_{X'}$ for each $[X] \in \supp(B) / G.$
\end{convention}
 
\subsubsection*{Simples and projective indecomposables}
Recall from \cref{prop:inv-sub-of-elementary} (or \cref{prop-semisimple-quot}) that $(\kk B)^G$ is also an elementary algebra. Hence, 
each idempotent in our cfpoi $\{E_{[X]}: [X] \in \supp(B) / G\}$ generates a distinct (up to isomorphism) projective indecomposable $(\kk B)^G-$module. We will write $P_{[X]}$ for the projective indecomposable $(\kk B)^G\cdot E_{[X]}$ and $M_{[X]}$ for its associated simple quotient $(\kk B)^G\cdot E_{[X]} / \rad \left((\kk B)^G\right)\cdot E_{[X]}.$
The Cartan invariants are of the form $[P_{[X]}: M_{[Y]}] = \dim_{\kk} E_{[Y]} \cdot (\kk B)^G \cdot E_{[X]}$ and are explored in \cref{ch:invariant_peirce_general}. 

\subsection{\cref{intro:thmA} and characterization of when an orbit-sum generates $(\kk B)^G$}\label{sec:rep-theory-invariant-subalgebras}

We now prove \cref{intro:thmA} from the introduction, rewritten as \cref{cor:when-is-invariant-semisimple} below. 

\begin{thm} \label{cor:when-is-invariant-semisimple}
Assume $\Char(\kk) \nmid |G|.$ The following are equivalent:
\begin{enumerate}
    \item The invariant subalgebra $(\kk B)^G$ is semisimple.
    \item $ |\supp(B) / G| = |B / G|.$
    \item  The invariant subalgebra $(\kk B)^G$ is commutative.
\end{enumerate}
\end{thm}
\begin{proof} 
 $((1) \iff (2))$: This is immediate from \cref{prop-semisimple-quot} and dimension counting as $\dim (\kk B)^G=|B/G|$ and $\dim (\kk \supp(B))^G=|\supp(B)/G|$.
 
$((1) \implies (3))$: Under the hypothesis, $(\kk B)^G\cong (\kk\supp(B))^G$ by \cref{prop-semisimple-quot}  and $\kk\supp(B)$ is commutative.  

$((3) \implies (1))$:  Fixing a cfpoi $\{E_{[X]}: X\in \supp(B)/G\}$ of $(\kk B)^G$, we have the Peirce decomposition 
\begin{align*}
(\kk B)^G &= \bigoplus_{[X],[Y]\in \supp(B)/G} E_{[Y]}\cdot (\kk B)^G\cdot E_{[X]}= \bigoplus_{[X],[Y]\in \supp(B)/G} (\kk B)^G\cdot E_{[Y]}E_{[X]} \\ &= \bigoplus_{[X]\in \supp(B)/G} E_{[X]} \cdot (\kk B)^G \cdot E_{[X]} =\bigoplus_{[X]\in \supp(B)/G} \kk E_{[X]}
\end{align*}
by Remark~\ref{rem:prim}, and so $(\kk B)^G\cong \kk^{\supp(B)/G}$ is semisimple.
\end{proof}

The proof  $((3)\implies (1))$ adapts to prove the abelianization of $(\kk B)^G$ is $(\kk \supp(B))^G\cong \kk^{\supp(B)/G}$.

\begin{example}\label{ex:semisimple-commutative-invariant-rings} \rm  We test out whether the LRBs we have seen so far give rise to semisimple and commutative invariant subalgebras.
\begin{itemize}
    \item \textbf{Semilattice with symmetry}: Let $L$ be a finite meet-semilattice with a group $G$ acting by poset automorphisms. Since the support semilattice $\supp(L)$ is $G$-equivariantly isomorphic to $L,$ the cardinality $|\supp(L) / G| = |L / G|$ and the invariant subalgebra $(\kk L)^G$ is always semisimple and commutative. 
    \item \textbf{Arrangements with symmetry}: Let $\mathcal{A}$ be a hyperplane arrangement in $V$ and let $G \leq GL(V)$ be a finite group which preserves the arrangement. In general, there are  more $G$-orbits of faces of $\mathcal{A}$ than there are $G$-orbits of the intersections. So, $\left(\kk \mathcal{F}(\mathcal{A})\right)^{G}$ is usually not semisimple. This is why Solomon's descent algebra is nonsemisimple and noncommutative for the majority of Coxeter types; one notable exception is for the dihedral Coxeter type $I_2(m)$ for $m$ even.
    \item \textbf{Cubulated $n$-gons}: Recall the cubulated $n$-gon $\mathscr{C}_n$ from \cref{ex:cubulated-ngon}. There are three $D_{2n}$-orbits of $\supp(\mathscr{C}_n)$: the empty intersection, each individual hyperplane, and the intersections of any two hyperplanes. On the other hand, there are strictly more $D_{2n}$-orbits of $\mathscr{C}_n$: three orbits of $0$-cubes, two of $1$-cubes and one of $2$-cubes. Hence, $(\kk \mathscr{C}_n)^{D_{2n}}$ is neither semisimple nor commutative. Similarly, one can also check that the algebra $(\kk \mathscr{C}_n)^{C_n}$ is nonsemisimple and noncommutative, where $C_n$ is the cyclic subgroup of $D_{2n}$ consisting of the rotations.
\end{itemize}

We determine when the invariant subalgebras for geometric lattice LRBs are semisimple and commutative.
\begin{cor}\label{cor:matroid-semisimple}
   Let $G$ act by poset automorphisms on a geometric lattice $\mathcal{L}$ and assume that $\Char(\kk) \nmid |G|$. Then, the following statements are equivalent:
   \begin{enumerate}
       \item $(\kk S(\mathcal{L}))^G$ is semisimple.
       \item $(\kk S(\mathcal{L}))^G$ is commutative.
       \item $G$ acts transitively on the set of complete flags  of $\mathcal{L}$.
   \end{enumerate}
\end{cor}
\begin{proof}
The equivalence of (1) and (2) follows from \cref{intro:thmA}.

We now show (3) is equivalent to (1). Assume $(\kk S(\mathcal{L}))^G$ is semisimple. By \cref{intro:thmA} (or \cref{cor:when-is-invariant-semisimple}), the cardinality $|\mathcal{L} / G| = |S(\mathcal{L})/ G|.$ In particular, the natural surjection $\sigma_G\colon S(\mathcal{L}) / G \to \mathcal{L} / G$ must be an injection. Since $\sigma_G$ maps all complete flags to the orbit $[\hat{1}] \in \mathcal{L} / G$, all of the complete flags in $\mathcal{L}(\mathcal{L})$ must belong to the same $G$-orbit. Conversely, suppose that $G$ acts transitively on the complete flags of $\mathcal{L}.$ This assumption implies that $G$ must act transitively on each rank of $\mathcal{L}.$ Hence, $|\mathcal{L} / G|$ is $k + 1$, if $k$ is the rank of $\mathcal{L}.$ More strongly, the assumption also implies that any two prefixes of complete flags of the same length are in the same $G$-orbit; hence, the $G$-orbits of $S(\mathcal{L})$ are indexed by the possible lengths $0 \leq \ell \leq k + 1$ of the flags. Hence, $|S(\mathcal{L}) / G| = k + 1$, too, and \cref{intro:thmA} implies that $(\kk S(\mathcal{L}))^G$ is semisimple, as desired.
\end{proof}

In general, it is uncommon for the automorphism group of a geometric lattice to act transitively on the set of complete flags the lattice. All such examples were classified by  Kantor \cite{Kantor}, and aside from a few sporadic examples, arise as truncations of the following three examples of geometric lattices with symmetry groups that \textit{do} act transitively on the complete flags: 
the symmetric group $S_n$ acting on the Boolean lattice $B_n$, the finite general linear group $\mathrm{GL}_n:= GL_n(\mathbb{F}_q)$ acting on $\pg(n-1, q)$, and the finite affine general linear group $A_n:= GL_n(\mathbb{F}_q) \ltimes \mathbb{F}_q^n$ acting on $\ag(n,q)$.  Hence, \cref{cor:matroid-semisimple} explains that the resulting invariant subalgebras $(\kk \free_n)^{S_n}$, $(\kk \free_n^{(q)})^{\mathrm{GL}_n}$, and $(\kk S(\ag(n, q)))^{A_n}$ are each semisimple. Their invariant subalgebras turn out to have a very simple structure,  as a consequence of the following more general result.

\begin{thm}\label{prop:diag}
    Suppose that $B$ is a connected LRB admitting an action by a group $G$ and $\kk$ is a field with $\Char(\kk)\nmid |G|$.   Let $\mathcal O\subseteq B$ be a $G$-invariant subset and fix $x=\sum_{b\in \mathcal O}b$. Then the following are equivalent:
    \begin{enumerate}
        \item $x$ generates $(\kk B)^G$ as a $\kk$-algebra;
        \item  $|B/G|=|\supp(B)/G|$, and $|\{b\in \mathcal O:\sigma(b)\geq X\}|\equiv|\{b\in \mathcal O:\sigma(b)\geq Y\}| \mod \Char(\kk)$  implies $[X]=[Y]$ for $X,Y\in \supp(B)$.
    \end{enumerate} 
   If these equivalent conditions occur, then as algebras, \[(\kk B)^G\cong \kk[x]/(m(x)),\] where \[m(x)=\prod_{[X]\in \supp(B)/G}\left(x-\lambda_{[X]}\right) \ \text{ and } \ \lambda_{[X]} = |\{b\in \mathcal O\mid \sigma(b)\geq X\}|.\]
\end{thm}
\begin{proof}
    Note that if $x$ generates $(\kk B)^G$ as a $\kk$-algebra, then $(\kk B)^G$ is commutative, and so $|B/G|=|\supp(B)/G|$ by \cref{cor:when-is-invariant-semisimple}.  It then follows from the proof of \cref{cor:when-is-invariant-semisimple}, that under either (1) or (2), we have $(\kk B)^G\cong (\kk \supp(B))^G\cong \kk^{\supp(B)/G}$.  We show that $m(x)$ defined above is the characteristic polynomial of $x$ under either of these hypotheses.
    
    Note that $\kk \Lambda(B)$ has a unique cfpoi, being commutative, and this family is $G$-invariant automatically.  Let $e_Y$ be the primitive idempotent of $\kk\supp(B)$ corresponding to $Y\in \supp(B)$ under our standard indexing.  Then \[\sigma(x)e_Y = \varepsilon_Y(x)e_Y =  |\{b\in \mathcal O:\sigma(b)\geq Y\}|\cdot e_Y.\]
    Since $\mathcal O$ is $G$-invariant, $|\{b\in \mathcal O:\sigma(b)\geq Y\}|=|\{b\in \mathcal O:\sigma(b)\geq g(Y)\}|$ for $g\in G$.   Thus $e_{[Y]} = \sum_{Y' \in [Y]}e_{{Y'}}\in (\kk \Lambda(B))^G$ is an eigenvector of $\sigma(x)$ with eigenvalue $\lambda_{[Y]}$.  Since $\{e_{[Y]}\mid [Y]\in \supp(B)/G\}$ is a $\kk$-basis for $(\kk\supp(B))^G$, it follows that $\sigma(x)$ (and hence $x$) is diagonalizable with characteristic polynomial $m(x)$ on $(\kk \supp(B))^G\cong (\kk B)^G$.  By dimension considerations $x$ generates $(\kk B)^G$ if and only if its minimal polynomial is its characteristic polynomial, which occurs if and only if the eigenvalues $\lambda_{[Y]}$ are distinct.  But this is precisely the second part of (2).
\end{proof}

As mentioned, \cref{prop:diag} applies to $B = \free_n, \free_n^{(q)}$, and $S(\ag(n, q))$. This can be seen by setting $\mathcal O$ to be the orbit of an element of rank $1$, since for these cases, the rank of a flat is determined by the cardinality of the number of elements of rank $1$ below it and the groups are flag-transitive. Writing $\lambda_i$ to be the eigenvalue $\lambda_{[X]}$ for $X$ of rank $i$, one can easily compute that $\lambda_i = i$ for $\free_n$, $\lambda_i = [i]_q$ for $\free_n^{(q)}$, and $\lambda_i = q^{i - 1}$ for $S(\ag(n, q))$ with $i \geq 1$. We point out \cref{prop:diag} generalizes and provides a simpler proof of the results in \cite[Theorem 1.3]{BraunerComminsReiner}. 
 \end{example}

\section{Invariant Peirce decompositions and poset topology background}
\label{ch:invariant_peirce_general}

Throughout, let $B$ be a connected LRB carrying the action of a finite group $G$ by semigroup homomorphisms and let $\kk$ be a 
field whose characteristic does not divide $|G|.$

The goal of this section is to present tools for computing the $G$-representation structure of the invariant Peirce components of an LRB algebra. In \cref{sec:inv_pierce_background_motivation}, we motivate why one might want to understand these representations. In \cref{sec:inv-peirce-general}, we provide some (limited) tools for understanding these representations for general LRBs. The background on poset topology is provided in \cref{sec:poset-topology-conventions}.

\subsection{Motivation}\label{sec:inv_pierce_background_motivation}
In this section, we study the $G$-representation structure of the decomposition of $\kk B$ into its invariant Peirce components: \[\kk B = \bigoplus_{\substack{[X] , [Y]\\ \in \supp(B) / G}}E_{[Y]} \cdot \kk B\cdot   E_{[X]}.\] Our motivation for understanding the $G$-representations $E_{[Y]} \cdot \kk B\cdot  E_{[X]}$ is twofold:
\begin{enumerate}
    \item \textbf{Cartan invariants of $(\kk B)^G$}: The multiplicities of the trivial representation within each component $ E_{[Y]}\cdot \kk B\cdot  E_{[X]}$, written $\langle \mathbb{1}, E_{[Y]}\cdot \kk B\cdot  E_{[X]}\rangle_G$, give the Cartan invariants of $(\kk B)^G$, see \cref{conversion-of-algebra-rep-to-g-rep}. 
    \item \textbf{Invariant theory of $\kk B$}: More generally, understanding the $G$-representation structure of each summand $E_{[Y]} \cdot \kk B\cdot  E_{[X]}$ helps one understand $\kk B$ as a module over $G$ and $\left( \kk B\right)^G$ simultaneously. In particular, as we explain in \cref{prop:decomposing-monoid-algebra} and \cref{conversion-of-algebra-rep-to-g-rep}, it reveals information about the $(\kk B)^G$-module structure of each $G$-isotypic component of $\kk B.$
\end{enumerate}

For an irreducible character $\chi$ of $G$ and a $G$-representation $V$, we write $V^\chi$ to denote the $\chi$-isotypic subspace of $V$.  These subspaces are accessed by the $\chi$-isotypic projectors \[\pi_\chi:= \frac{\chi(1)}{d_{\chi}\cdot |G|}\sum_{g \in G}
{\chi(g^{-1})}g \in \kk G\]
where $d_{\chi}$ is the dimension of the endomorphism algebra of the simple module affording $\chi$.
\begin{remark}\label{rem-iso-typ}\rm 
Note that $v \in V$ satisfies $v \in V^\chi$ if and only if $\pi_\chi v = v.$  In particular, if $f\colon V\to V$ is any $\kk G$-module homomorphism, then $f(V^\chi)\subseteq f(V)^\chi$ as $\pi_\chi f(v) = f(\pi_\chi v)=f(v)$ for $v\in V^\chi$.  
\end{remark}

To explain our motivations (1) and (2), we need two lemmas.
\begin{lemma} \label{lem:rep-structure-of-cols}
    The map $\kk$-linearly extending $b \mapsto b E_{\sigma(b)}$ gives an isomorphism of $G$-representations \begin{align*}
        \mathrm{span}_\kk \left \{b \in B: \sigma(b) \in [X] \right\} \cong  \kk B\cdot  E_{[X]}.
    \end{align*}
\end{lemma}
\begin{proof}
   The map is $G$-equivariant since $g\left(bE_{\sigma(b)}\right) = g(b)g\left(E_{\sigma(b)}\right) = g(b) E_{g(\sigma(b))} = g(b) E_{\sigma(g(b))}.$ It is a vector space isomorphism by \cref{prop:bE_b-form-basis} and the vector space decomposition $\kk B\cdot   E_{[X]} = \bigoplus_{X' \in [X]}\kk B\cdot  E_{X'}$.
\end{proof}

\begin{lemma}\label{prop:decomposing-monoid-algebra}
Let $\chi$ be an irreducible character of $G.$ As $\left(\kk B\right)^G-$modules, \[\left(\kk B\right)^\chi = \bigoplus_{[X] \in \supp(B) / G}\left(\kk B\cdot  E_{[X]}\right)^\chi.\] 
The $\kk-$dimension of each $(\kk B)^G-$submodule is \[\dim_{\kk} \left(\kk B\cdot  E_{[X]}\right)^\chi = \chi(1) \cdot \left \langle\chi, \Psi_{[X]}\right \rangle, \]
where $\Psi_{[X]}$ is the $G$-character carried by the representation $\mathrm{span}_\kk\{ b \in B: \sigma(b) \in [X]\}.$
\end{lemma}
\begin{proof}
By orthogonality and completeness of the $\{E_{[X]}\}$, there is a $(\kk B)^G$-module decomposition \[\left(\kk B \right)^\chi = \bigoplus_{[X] \in \supp(B) / G}\left(\kk B\right)^\chi \cdot E_{[X]}.\] By~\cref{eqn:commuting-actions} and \cref{rem-iso-typ}, $\left(\kk B\right)^\chi \cdot E_{[X]} = \left(\kk B\cdot  E_{[X]}\right)^\chi.$ The statement on the $\kk$-dimension of each submodule follows from \cref{lem:rep-structure-of-cols}.
\end{proof}

\cref{prop:decomposing-monoid-algebra} only explains the $\kk$-dimensions of the submodules $(\kk B\cdot  E_{[X]})^{\chi}$, not their structure as modules over $(\kk B)^G$. We  now explain how invariant Peirce components can help one better understand these module structures. Recall the notion of the orbit poset $\supp(B) / G$ from the discussion before \cref{lem:invariant-Saliola-lemma}.

\begin{prop}\label{conversion-of-algebra-rep-to-g-rep}
    Fix $[X]$ and $[Y]$ in $\supp(B) / G.$ Let $\chi$ be an irreducible character of $G.$ The composition multiplicity of the $(\kk B)^G$-simple $M_{[Y]}$ in the module $\left(\kk B\cdot   E_{[X]}\right)^\chi$ is $0$ unless $[Y] \geq [X],$ in which case,
    \begin{align*}
        \left[\left(\kk B\cdot  E_{[X]}\right)^\chi: M_{[Y]}\right] = \chi(1) \cdot \left \langle \chi, E_{[Y]}\cdot \kk B\cdot  E_{[X]}\right \rangle_G.
    \end{align*}
    In particular, the Cartan invariants of $(\kk B)^G$ are given by \[[P_{[X]}: M_{[Y]}] = \langle \mathbb{1}, E_{[Y]} \cdot  \kk B\cdot  E_{[X]}\rangle_G.\]
\end{prop}
\begin{proof}
Since $(\kk B)^G$ is elementary by \cref{prop:inv-sub-of-elementary}, the dimension of $D = \mathrm{End}_{(\kk B)^G}(M)$ is one for any simple $(\kk B)^G$-module $M$. Hence, by \cref{eqn:commuting-actions}, \cref{rem-iso-typ},  and \cref{eqn:composition-mults},
\begin{align*}
     \left[\left(\kk B \cdot E_{[X]}\right)^\chi: M_{[Y]}\right] = \dim_{\kk}E_{[Y]}\cdot \left(\kk B\cdot  E_{[X]}\right)^\chi &= \dim_\kk \left(E_{[Y]} \cdot  \kk B\cdot  E_{[X]}\right)^\chi\\
     &= \chi (1) \cdot \left \langle \chi, E_{[Y]}\cdot \kk B\cdot  E_{[X]}\right \rangle_G.
\end{align*}
Since $E_{[Y]} = \sum_{Y' \in [Y]}E_{Y'}$, using orthogonality one deduces \[E_{[Y]} \cdot \kk B\cdot  E_{[X]} = \bigoplus_{\substack{Y' \in [Y]}}E_{Y'} \cdot \kk B\cdot   E_{[X]}.\] By \cref{cor:supports-of-kb-idemps}(iii) and \cref{lem:invariant-Saliola-lemma}, this space is zero if $[Y] \ngeq [X].$
\end{proof}

\subsection{General theory}\label{sec:inv-peirce-general}

We reduce understanding the $G$-module structure of the invariant Peirce component $E_{[Y]} \cdot \kk B\cdot  E_{[X]}$ to understanding the $G_X \cap G_Y$-module structure of the ordinary Peirce component $E_Y \cdot \kk B\cdot  E_X$.

\begin{prop}\label{prop:invariant-Peirce-decomp}
Fix $[X']$ and $[Y']$ in $\supp(B) / G$. As $G$-representations, \[E_{[Y']} \cdot \kk B\cdot   E_{[X']} \cong \bigoplus_{\substack{[X \leq Y] : \\ X \in [X'], Y \in [Y']}}E_{Y} \cdot \kk B\cdot  E_{X} \Big \uparrow_{\GX \cap \GY}^G,\] 
where $[X\leq Y]$ denotes the $G$-orbit of a comparable pair $X\leq Y$.
\end{prop}
\begin{proof}
  Since $E_{[Y']} = \sum_{Y \in [Y']}E_Y$ and $E_{[X']} = \sum_{X \in [X']}E_X$, and the $\{E_Z\}$ are orthogonal, as vector spaces:
  \[E_{[Y']} \cdot \kk B\cdot   E_{[X']} = \bigoplus_{X \in [X'], Y \in [Y']}E_{Y} \cdot \kk B\cdot  E_{X}.\]

  By \cref{cor:supports-of-kb-idemps}(iii), $E_Y \cdot \kk  B \cdot E_X = 0$ if $X \nleq Y.$ Hence, we can simplify the above to:
 \[E_{[Y']} \cdot \kk B\cdot   E_{[X']} = \bigoplus_{\substack{X \leq Y:\\ X \in [X'], Y \in [Y']}}E_{Y} \cdot \kk B\cdot  E_{X}.\]

Notice that $g(E_Y\cdot \kk B\cdot E_X) = E_{g(Y)}\cdot\kk B\cdot E_{g(X)}$, whence the collection of subspaces \[\{E_Y\cdot \kk B\cdot E_X: X \leq Y, X \in [X'], Y \in [Y']\}\] is a system of imprimitivity and $G_X\cap G_Y$ is the stabilizer of $E_Y\cdot \kk B\cdot E_X$ in this system.
 The proposition then follows from standard properties of induced representations---specifically, see~\cite[Proposition 4.3.2]{webbrepntheory}.
\end{proof}

An application of \cref{conversion-of-algebra-rep-to-g-rep}, \cref{prop:invariant-Peirce-decomp}, and Frobenius reciprocity then yields:

\begin{cor}\label{c:cartan.formula}
Fix $[X']$ and $[Y']$ in $\supp(B) / G$.  Then the Cartan invariant \[\dim_{\kk}E_{[Y']}\cdot (\kk B)^G\cdot E_{[X']}=\sum_{\substack{
[X \leq Y] : \\ X \in [X'], Y \in [Y']}}\langle \mathbb 1,E_Y\cdot \kk B\cdot E_X\rangle_{G_X\cap G_Y}.\]
\end{cor}

The invariant Peirce components for meet-semilattices are particularly simple.

 \begin{prop}\label{prop:meet-semilattice-invariant-peirce}
     Let $L$ be a meet-semilattice. For $[X]$ and $[Y] \in \supp(L) / G \cong L / G$, the space $E_{[Y]} \cdot \kk L \cdot E_{[X]} = 0$ unless $[X] = [Y]$, in which case, as $G$-representations, $E_{[X]} \cdot \kk L \cdot E_{[X]} \cong \kk\left[G / G_X\right].$
 \end{prop}
 \begin{proof}
     Since $\kk L$ is commutative and has $\{E_X: X \in \supp(B)\}$ as a $\kk$-basis, we have by the orthogonality of the $\{E_X\}$ that
     \begin{align*}
     E_Y \cdot  \kk L\cdot   E_X = \begin{cases}
         0 & \text{ if }X \neq Y,\\
         \mathrm{span}_{\kk}(E_X) & \text{ otherwise}.
     \end{cases}
     \end{align*}
     The claim then follows from \cref{prop:invariant-Peirce-decomp}. 
 \end{proof}

 For general connected LRBs, the invariant Peirce components are more complex; however, the special case $[X] = [Y]$ is still straightforward. 

 \begin{prop}\label{prop:inv-decomp-x=x}
     Fix $[X] \in \supp(B) / G.$ As $G$-representations, 
     $E_{[X]} \cdot \kk B\cdot  E_{[X]} \cong \kk \left[G / G_X\right].$ 
 \end{prop}

 \begin{proof}
     Note that if $X<Y$, then $[X]\neq [Y]$, as the principal left ideal $X$ has strictly smaller cardinality than $Y$.  Thus, by \cref{prop:invariant-Peirce-decomp}, 
     \begin{align*}
         E_{[X]}\cdot  \kk B\cdot E_{[X]} \cong E_X \cdot \kk B\cdot  E_X \Big \uparrow_{G_X}^{G}.
     \end{align*}
     The Cartan invariants for a connected LRB are computed in~\cite[Theorem 4.18]{MSS}), and their result gives that $E_X \cdot \kk B\cdot   E_X$ is one-dimensional. Thus, $E_{X}\cdot  \kk B\cdot  E_X = \kk E_X$
     from which the claim follows.
 \end{proof}

By taking invariants, we recover the fact from \cref{rem:prim} that $E_{[X]}\cdot (\kk B)^G\cdot E_{[X]}$ is one-dimensional.

 Outside of the semilattice case and the special cases $[X] =  [Y]$ and $[X] \nleq [Y]$, understanding invariant Peirce components becomes much more difficult. The goal for the remainder of this paper is to leverage \textit{poset topology} to understand the invariant Peirce components of two large families of LRBs: $CW$ LRBs and hereditary LRBs.  Hence, we explain the basics on poset topology in the next subsection.


\subsection{Background: poset topology}
\label{sec:poset-topology-conventions}
Given that poset topology will play a key role in the remainder of this paper, we review poset homology and cohomology here. For more details, see \cite{WachsPosetTopology} as an excellent source. Throughout, let $P$ be a finite (and potentially even  empty) poset and let $R$ be a commutative ring with unit. Although we will end up almost always setting $R$ to be our ambient field $\kk$, we define several of the concepts within this more general context since there are a few times we appeal to homology over the integers.

An \textbf{$i$-chain} in $P$ is a sequence of elements of $P$ of the form $(p_0 < p_2 < \cdots < p_{i}).$ We define the space $C_i(P; R)$---or $C_i(P)$ when $R$ is understood---to be the free $R$-module of $i$-chains of $P$. Note that we allow empty chains\footnote{We use the convention that an empty poset $P$ still has an empty chain.} so that $C_{-1}(P) = R,$ the free $R$-module spanned by the empty chain $()$ and that for all $i <-1$, $C_{i}(P) = 0.$ The \textbf{boundary map }\textbf{$\partial_i\colon C_i(P) \longrightarrow C_{i - 1}(P)$} is the $R$-linear map sending
\begin{align*}
    (p_0 < p_1 < \cdots < p_{i}) \mapsto \sum_{j = 0}^{i}(-1)^j (p_0 < p_1 < \cdots < p_{j - 1}< p_{j+1} < \cdots  < p_{i}).
\end{align*}
We obtain a chain complex 
\begin{align*}
 \cdots \xrightarrow{\,\partial_3\,} C_2(P)\xrightarrow{\,\partial_2\,} C_1(P)\xrightarrow{\,\partial_1\,} C_0(P) \xrightarrow{\,\partial_0\,} R\xrightarrow{\partial_{-1}} 0
\end{align*}
This is just the augmented ordered simplicial chain complex for $\Delta(P)$.  

The \textbf{$i$-th reduced homology group of $P$}, written \textbf{$\widetilde{H}_i(P; R)$}---or $\widetilde{H}_i(P)$ when $R$ is understood---is \[\widetilde{H}_i(P): = \ker(\partial_i) / \im(\partial_{i + 1}).\]

Although we will almost always work with \textit{reduced} homology, there are a few occasions where we will work with just the \textit{standard} {homology groups} of $P$, written $H_i(P).$ Note that
\begin{align*}
    H_i(P) = \begin{cases}
        \widetilde{H}_i(P) & \text{if } i \geq 1,\\
        C_0(P) / \im(\partial_1) & \text{if } i = 0,\\
        0 & \text{otherwise.}\\
    \end{cases}
\end{align*}

The \textbf{reduced Euler characteristic} of $\Delta(P)$ (using the coefficient field $\kk$) is 
\[\widetilde{\chi}(\Delta(P))= \sum_j (-1)^j\dim_\kk C_j(P) =\sum_j (-1)^j\dim_\kk \widetilde H_j(P).\] P.~Hall's theorem says that $\widetilde{\chi}(\Delta(p,q)) = \mu_P(p,q)$, for $p<q$, where $\mu_P$ is the M\"obius function of $P$.

We also have the dual notion of poset cohomology. One puts $C^i(P)=\hom_R(C_i(P),R)$ and sets $\partial_i^*\colon C^{i-1}(P)\to C^{i}(P)$ to be the dual of $\partial_i\colon C_i(P)\to C_{i-1}(P)$.  Then $C^\bullet(P)$ is a cochain complex.  Of course, this gives the augmented ordered simplicial cochain complex for $\Delta(P)$.    Since $P$ is finite and $C_i(P)$ has a distinguished basis consisting of the $i$-chains $(p_0<\cdots<p_i)$, we can identify $C^i(P)$ (as a free $R$-module) with $C_i(P)$ by identifying our distinguished basis for $C_i(P)$ with its dual basis.  With this identification the coboundary map \textbf{$\partial_i^\ast\colon C_i(P) \longrightarrow C_{i + 1}(P)$} is the $R$-linear map sending
\begin{align*}
    (p_0 < p_1 < \cdots < p_{i}) \longmapsto\sum_{j = 0}^{i + 1} \,\, \sum_{p_{j - 1} < p' < p_j}(-1)^j (p_0 < \cdots < p_{j - 1} < p' < p_j  < \cdots <  p_i).
\end{align*}

We define the \textbf{$i$-th reduced cohomology group of $P$}, written \textbf{$\widetilde{H}^i(P)$} to be \[\widetilde{H}^i(P): = \ker(\partial_i^\ast) / \im(\partial_{i - 1}^\ast).\]
Similarly, the standard cohomology groups are \begin{align*}
    H^i(P; R) = H^i(P) = \begin{cases}
        \widetilde{H}^i(P) & \text{if } i \geq 1,\\
        \ker(\partial_0^\ast)  & \text{if } i = 0,\\
        0 & \text{otherwise.}\\
    \end{cases}
\end{align*}

\begin{remark}\label{rem:top-cohomology}\rm
    Assume that $P$ is a graded poset and $p < q$ in $P.$ Set $r -2$ as the length of the maximal chains of the open interval $(p, q).$ Then, there is precisely one position in which any $(r - 3)$-chain in $C_{r - 3}((p, q))$ can be refined. Thus, $\partial_{r - 3}^\ast$ sends any $(r - 3)$-chain $p_0 \lessdot  p_1 \lessdot  \ldots \lessdot p_i \color{blue}{<}\normalcolor  p_{i + 1} \lessdot \ldots \lessdot   p_{r - 3}$ in $(p, q)$ to 
    \begin{align*}
        (-1)^{i + 1} \sum_{p_i \lessdot p' \lessdot p_{i + 1}}p_0 \lessdot \ldots \lessdot p_i \lessdot \color{blue}{p'} \normalcolor \lessdot p_{i + 1} \lessdot \ldots \lessdot p_{r-3}. 
    \end{align*}
    Hence, in this case, the $R$-module $\im \left(\partial_{r - 3}^\ast\right)$ is a free $R$-module spanned by elements of the above form as one varies over all $(r - 3)$-chains of $(p, q).$
\end{remark}

\begin{remark} \label{rem:group-poset}\rm
The complexes $C_\bullet(P)$ and $C^{\bullet}(P)$ are functorial  with respect to order preserving maps, with the former covariant and the latter contravariant.  So if a finite group $G$ acts on 
$P$ by poset automorphisms, then $C_\bullet(P)$ is a chain complex of $G$-modules (note that $G$ acts trivially on  $C_{-1}(P)$) and $C^\bullet(P)$ is a cochain complex of right $G$-modules, which we can (and do) view as left $G$-modules via the inversion.  Note that the standard $G$-action on $C_i(P)$ coincides with the action obtained from identifying $C_i(P)$ with $C^i(P)$, after making it a left action.
With this understanding,
$\widetilde{H}_i\left(P\right), {H}_i\left(P\right)$ and $\widetilde{H}^i\left(P\right), {H}^i\left(P\right)$ are all $G$-modules. If we take $R$ to be our field $\kk$, then the pair $\widetilde{H}_i\left(P\right)$ and $\widetilde{H}^i\left(P\right)$ (as well as the pair ${H}_i\left(P\right)$ and ${H}^i\left(P\right)$)  are \textit{contragredient} $G$-representations. In particular, the vector space $\widetilde{H}^i(P)$ can be identified with the dual vector space of $\widetilde{H}_i(P)$, and if $g \in G$ acts on $\widetilde{H}_i(P)$ by a matrix $M_g$, then $g$ acts on $\widetilde{H}^i(P)$ via $(M_{g^{-1}})^T$.
\end{remark}

\begin{remark}\label{rem:empty-interval}\rm
    As an extension of \cref{rem:top-cohomology} and \cref{rem:group-poset}, we set the convention that for $p \in P$, the nonsensical ``open interval'' $(p, p)$ carries reduced homology and cohomology only in dimension $-2$, in which case $\widetilde{H}_{-2}((p, p)) \cong \widetilde{H}^{-2}((p, p)) \cong R$. Further, if $G$ fixes $p$, then we assume $\widetilde{H}_{-2}((p, p); \kk) \cong \widetilde{H}^{-2}((p, p); \kk) \cong \kk$ carries the trivial $G$-representation. Compare this carefully to the open interval $(p, q)$ for $p \lessdot q$, which is the empty poset and thus carries reduced (co)homology only in degree $-1$. This convention may seem a bit strange, but it is not unheard of (for instance, see \cite[\S 1.2]{saliolafacealgebra}). The reason for this convention is so that $\mu(p,q) =\widetilde{\chi}(\Delta(p,q))$, even when $p=q$, and so that we don't need a special case for $X = Y$ in \cref{cor:CWLRBs-as-poset-topology}. 
\end{remark}

\begin{remark}\rm
    The reduced homology and cohomology of a poset agree with the simplicial and singular reduced homology and cohomology of its order complex. In particular, letting $\widetilde{H}_i^{\mathrm{sing}}$, $\widetilde{H}^i_{\mathrm{sing}}$ and $\widetilde{H}_i^{\Delta}$, $\widetilde{H}^i_\Delta$ be the singular and simplicial reduced homology and cohomology of a poset $P$, we have
    \begin{align*}
        \widetilde{H}_i(P) &\cong \widetilde{H}_i^{\mathrm{sing}}(\|\Delta(P)\|) \cong \widetilde{H}_i^{\Delta}(\|\Delta(P)\|), \text{ and }\\
         \widetilde{H}^i(P) &\cong \widetilde{H}^i_{\mathrm{sing}}(\|\Delta(P)\|) \cong \widetilde{H}^i_{\Delta}(\|\Delta(P)\|).
    \end{align*}
\end{remark}

\subsubsection*{{Joins}}
We refer the reader to~\cite[Chapter~V]{CookeFinney} for details on joins of CW complexes. 
 Let $P,Q$ be posets.  Their \textbf{join} $P\ast Q$ is the poset $P\sqcup Q$, extending the poset structures on $P$ and $Q$ by making every element of $P$ less than every element of $Q$.  Then $\Delta(P\ast Q)=\Delta(P)\ast \Delta(Q)$ and $\|\Delta(P\ast Q)\|\cong \|\Delta(P)\|\ast \|\Delta(Q)\|$, as is well known.  
It is also well known, and straightforward to check, that \[C_q(P\ast Q)\cong\bigoplus_{\substack{-1\leq i,j\leq q:\\i+j=q-1}}C_i(P)\otimes_R C_j(Q)\] with differential
$d_q(c_i\otimes c_j') = d_i(c_i) \otimes c_j'+(-1)^{i+1}c_i\otimes d_j(c_j')$
(with our usual conventions on degree $-1$), via the map 
\[(p_0<\cdots <p_i<q_0<\cdots <q_j)\longmapsto (p_0<\cdots<p_i)\otimes (q_0<\cdots <q_j).\]
Notice that the join construction and the above isomorphism is natural in $P$ and $Q$ with respect to poset maps.  The K\"{u}nneth formula computes the homology and cohomology of $P\ast Q$ in terms of that of $P$ and $Q$. 

\begin{thm}\label{thm:joins}
Let $P$, $Q$ be finite posets.  Assuming that  $R$ is a PID,  there are natural short exact sequences
\begin{align*}
    0\to \bigoplus_{\substack{-1\leq i,j\leq q:\\i+j=q-1}} \widetilde{H}_i(P)\otimes_R \widetilde{H}_j(Q)\to \widetilde{H}_q(P\ast Q)\to \bigoplus_{\substack{-1\leq i,j\leq q-1:\\i+j=q-2}}\mathrm{Tor}^R_1(\widetilde{H}_i(P),\widetilde{H}_j(Q))\to 0\\
     0\to \bigoplus_{\substack{-1\leq i,j\leq q:\\i+j=q-1}} \widetilde{H}^i(P)\otimes_R \widetilde{H}^j(Q)\to \widetilde{H}_q(P\ast Q)\to \bigoplus_{\substack{-1\leq i,j\leq q-1:\\i+j=q-2}}\mathrm{Tor}^R_1(\widetilde{H}^i(P),\widetilde{H}^j(Q))\to 0\\
\end{align*}
that split, but not naturally.
\end{thm}

 The join construction was generalized to left regular bands in~\cite{MSS}. If $B,B'$ are left regular bands, their join $B\ast B'$ is $B\sqcup B'$ with the products in $B$ and $B'$ extended by $b'b=b=bb'$ for all $b\in B$ and $b'\in B'$.  
 The semigroup poset of $B\ast B'$ is the join of the semigroup posets of $B$ and $B'$.  Moveover, if $B$ is a CW LRB which is the face poset of a sphere and $B'$ is any CW LRB, then $B\ast B'$ is a CW LRB~\cite[Proposition~3.12]{MSS}.

\subsubsection*{{Relative homology groups}}
We will occassionally need the concept of relative homology. Let $Q$ be a subposet of $P.$ Then $C_\bullet(Q)$ is a subcomplex of $C_\bullet(P)$, and so we have a quotient chain complex.
\begin{align*}
 \ldots \xrightarrow{\,\overline{\partial_3}\,} C_2(P) / C_2(Q) \xrightarrow{\,\overline{\partial_2}\,} C_1(P) / C_1(Q) \xrightarrow{\,\overline{\partial_1}\,} C_0(P) / C_0(Q)\longrightarrow 0.
\end{align*}
The (nonreduced) homology groups of this chain complex form the \textbf{relative homology groups} $H_i(P, Q).$ Specifically, 
\begin{align*}
   H_i(P, Q; R) =  H_i(P, Q):= \ker(\overline{\partial_i}) / \im(\overline{\partial_{i + 1}}).
\end{align*}
This is of course the relative simplicial homology $H_i(\Delta(P),\Delta(Q);R)$. Thus we have:

\begin{prop}\label{prop:facts-relative-homology}
    Let $Q$ be a nonempty subposet of $P.$ There is a long exact sequence of pairs of the form
\begin{align}\label{eqn:longexactpairs}
   \cdots \longrightarrow \widetilde{H}_n(Q) \xrightarrow{\,i_*\,} \widetilde{H}_n(P) \xrightarrow{\,j_*\,} H_n(P, Q) \xrightarrow{\,\partial\,} \widetilde{H}_{n - 1}(Q) \longrightarrow \cdots,
\end{align}
where $i\colon C_n(Q) \to C_n(P)$ is inclusion and $j\colon C_n(P) \to C_n(P, Q)$ is the natural projection. 
\end{prop}

\section{Invariant Peirce components for $CW$ LRBs}\label{sec:CWLRBs}

 Throughout this section, we assume $B$ is a connected $CW$ LRB. We mantain our assumptions that $G$ is a finite group acting on $B$ by automorphisms and $\kk$ is a field with $\Char(\kk) \nmid |G|.$ Our goal is to generalize a known, useful connection between the Peirce components of hyperplane face monoid algebras and poset topology to \textit{all} $CW$ LRB algebras. We begin by explaining this motivating connection.

\subsection{Previously known case: real, central arrangements with symmetry} \label{subsec:motivation-real-arrangements-symmetry} 
Assume that $\mathcal{F}$ is the face monoid of a real, {central} \normalcolor hyperplane arrangement $\mathscr{A}$ in $\mathbb{R}^n$ and let $G$ be a finite subgroup of $GL(\mathbb{R}^n)$ which preserves $\mathscr{A}.$ By \cref{lem:g-acts-hyperplane-semigroup-auts}, $G$ acts on $\mathcal{F}$ by semigroup automorphisms.

For a support $X \in \supp(\mathcal{F})$---which corresponds to an intersection in  $\mathcal{L}(\mathcal{A})$---we define a character $\det(X)$ of the \textbf{setwise stabilizer} $\GX$ by sending $g \in \GX$ to the determinant of its action on the intersection associated to $X$. Since $g \in G$ has finite order and
real entries, the value of $\det(X)(g)$ is always in $\{+1, -1\}.$

 The useful connection between the Peirce components of $\kk \mathcal{F}$ and poset topology is as follows. It is implicit in the work of Saliola on reflection arrangements \cite{saliolaquiverdescalgebra} and appears more explicitly in Aguiar--Mahajan's monograph \cite[Proposition 14.44]{Aguiar-Mahajan}. 

\begin{thm}\label{prop:double-idem-spaces-homology} 
    Let $X \leq Y$ be intersections in $\mathcal{L}(\mathcal{A})$ and let $(X, Y)$ be the subposet given by its open interval in $\supp(\mathcal{F}).$ Set $k = \rk(Y) - \rk(X).$ As $\GX \cap \GY$-representations,
    \[E_Y \cdot \kk \mathcal{F} \cdot E_X \cong \det(X) \otimes \widetilde{H}^{\mathrm{k - 2}}((X, Y); \kk ) \otimes  \det(Y).\]
\end{thm}

This connection was very useful in the first author's work computing the symmetric group representation structure on the invariant Peirce components for the face algebra of the braid arrangement  \cite{commins2024invariant}.

One challenge in generalizing \cref{prop:double-idem-spaces-homology} to all $CW$ LRBs is determining the appropriate analogue of the determinant character twists. Indeed, most $CW$ LRBs do not have associated vector spaces readily available. Instead, we will take advantage of the topology of regular CW complexes.

\normalcolor

\subsection{Background: cellular homology}

In \cref{sec:poset-topology-conventions}, we described a homology theory for posets which agrees with the simplicial and singular homology of their order complexes. \textbf{Cellular homology} is a homology theory which agrees with singular homology and covers not just simplicial complexes, but also $CW$ complexes, and is still  substantially simpler to compute than singular homology. We will not work so much with the \textit{cellular homology groups}, but will instead heavily use the structure of their \textit{cellular chain complexes}.

\subsubsection{Cellular chain complexes and orientations} 

We remind the reader of some basic topological notions and refer to~\cite[Chapter~V]{LundellWeingram} for details.
Let $R$ be a commutative ring, with all homology --- unless otherwise specified --- understood to be taken with coefficients in $R$. If $X$ is a CW complex, then the cellular chain complex of $X$ is given by $C_q(X) = H_q^{\mathrm{sing}}(X^q,X^{q-1})$ with boundary map $d_q\colon C_q(X)\to C_{q-1}(X)$ given by the composition 
\[H^{\mathrm{sing}}_q(X^q,X^{q-1})\xrightarrow{\,\,\partial\,\,} H_{q-1}^{\mathrm{sing}}(X^{q-1})\xrightarrow{\,\,j_*\,} H_{q-1}^{\mathrm{sing}}(X^{q-1}, X^{q-2}).\] Here, the first map is the boundary map from the long exact sequence in relative homology of the pair $(X^q,X^{q-1})$ and the second map is induced by the inclusion $j\colon (X^{q-1},\emptyset)\to (X^{q-1},X^{q-2})$.  The homology of $C_\bullet(X)$ is well-known to be isomorphic to the singular homology of $X$.  

Let $X(q)$ denote the set of $q$-cells of $X$.  Then
 for all $q \geq 0,$ as $R$-modules, \[H_q^{\mathrm{sing}}(X^q, X^{q-1}) \cong \bigoplus_{e \in X(q)}  H_q^{\mathrm{sing}}(\overline e, \dot e) \cong \bigoplus_{e\in X(q)}H^{\mathrm{sing}}_q(E^q,S^{q-1})\cong \bigoplus_{e\in X(q)}R,\] where $\dot e=\overline e\setminus e.$

An \textbf{orientation} of a cell $e \in X(q)$ is an isomorphism $\varphi_e\colon H_{q}(\overline e, \dot e; \Z) \to  \Z$. By the Universal Coefficient Theorem, there is a natural isomorphism $\phi_e$ 
\begin{center}
    \begin{tikzcd}
         H_{q}(\overline e, \dot e; \mathbb{Z}) \otimes R \arrow[r, "\phi_e"] & H_{q}(\overline e, \dot e; R).
    \end{tikzcd}
\end{center}
Hence, given an orientation $\varphi_e$, there is a {unique} $R$-isomorphism \begin{align}\label{eqn:eps-=maps}
    \epsilon_e\colon H_{q}(\overline e, \dot e; R) \to R
\end{align} 
compatible with our orientation via the natural map $\mathbb Z\to R$.
Thus if we fix an orientation of each cell of $X$, then we have a canonical basis for $C_q(X;R)$, consistent with extension of scalars, consisting of the $\hat e:=\epsilon_e^{-1}(1)$ with $e\in X(q)$.

Throughout the remainder of this subsection, let $P$ be a $CW$ poset.   For an integer $0 \leq q \leq \rk(P),$ define $P^q$ be the subposet of $P$ consisting of the elements of rank at most $q,$ i.e., $P^q := \{p \in P: \rk(p) \leq q\}.$ Let $\Sigma(P)$ be the regular CW complex associated to $P$.  Then the $q$-cells in $\Sigma(P)$ are in bijection with elements $p \in P$ of rank $q$.  The corresponding closed cell is $\|\Delta(P^{\leq p})\|$ and the open cell is $\|\Delta(P^{\leq p})\|\setminus \|\Delta(P^{<p})\|$. 

We have a natural isomorphism $C_q(\Sigma(P))=H_q^{\mathrm{sing}}(\Sigma(P)^q,\Sigma(P)^{q-1})\cong H_q(P^q,P^{q-1})$, and so we identify these spaces from now on.  This isomorphism identifies the cellular boundary map of degree $n$ with the map $d_n$ from $H_n(P^n, P^{n-1})$ to $H_{n-1}(P^{n-1}, P^{n-2})$ that is the composition of $\partial\colon H_n(P^n, P^{n-1}) \to H_{n-1}(P^{n-1})$  coming from the long exact sequence of the pair $(P^{n}, P^{n-1})$ and the map $j_\ast\colon H_{n-1}(P^{n-1}) \to H_{n-1}(P^{n-1}, P^{n-2})$ induced by the inclusion $j\colon (P^n,\emptyset)\to (P^n,P^{n-1})$. 
Note that \[H_q(P^q, P^{q-1}) = \bigoplus_{\substack{p \in P:\\ \rk(p) = q}}  H_q(P^{\leq p}, P^{<p}) \cong \bigoplus_{\substack{p \in P:\\ \rk(p) = q}} H^{\mathrm{sing}}_q(E^q,S^{q-1})\cong \bigoplus_{\substack{p \in P:\\ \rk(p) = q}}R,\] which will allow us to think of our cellular chain groups as having an $R$-basis \textit{labelled} by the elements of $P.$
In particular, with our fixed orientations of $\Sigma(P)$, if $p\in P$ with $\rk(p)=q$, then $\hat p$ will denote the basis element of $H_q(P^q,P^{q-1})$ corresponding to the oriented cell $\|\Delta(P^{\leq p})\|\setminus \|\Delta(P^{<p})\|$.

\subsubsection{Cellular maps and degrees}

A mapping $g\colon X\to Y$ of CW complexes is \emph{cellular} if $g(X^q)\subseteq Y^q$ for all $q\geq 0$.  There is then an induced chain map $g_\ast\colon C_\bullet(X)\to C_\bullet(Y)$ of cellular chain complexes. If $e\in X(q)$, then 
\[g_\ast(\hat e) = \sum_{f\in Y(q)} [g:e:f]\widehat{f}\] for certain integers $[g:e:f]$.  One can describe $[g:e:f]$ as the degree of the map of oriented $q$-spheres \[\overline e/\dot e\to \bigvee_{x\in X(q)}\overline x/\dot x\cong X^q/X^{q-1}\to Y^{q}/Y^{q-1}\cong \bigvee_{y\in Y(q)}\overline y/\dot y\to \overline f/\dot f.\]  
That is, $H_q(\overline e,\dot e)\to H_q(X^q,X^{q-1})\to H_q(Y^q,Y^{q-1})\to H_q(\overline f,\dot f)$ sends $\hat e$ to $[g:e:f]\widehat f$.

A map $g\colon X\to Y$ is said to be a \textbf{regular cellular map}~\cite{LundellWeingram} if for each cell $e\in X(q)$ one has that $g(e)$ is a cell of $Y$ of dimension at most $q$.  In this case, $g(\overline e)= \overline {g(e)}$ is a closed cell, and $[g:e:f]=0$ unless $f=g(e)$.  Thus we define the \textbf{degree}
\begin{equation}\label{eqn:degree-definition}
\deg(e,g) =\begin{cases}[g:e:g(e)], & \text{if}\ g(e)\in Y(q)\\ 0, & \text{else}\end{cases}
\end{equation}
so that one has  \[g_\ast(\widehat e) = \deg(e,g)\widehat{g(e)}.\]

The following are immediate from the functoriality of the induced map on cellular chain complexes.
\begin{lemma}\label{lem:facts-about-cellular-maps}
Let $f\colon X\to Y$ and $g\colon Y\to Z$ be regular cellular maps between regular CW complexes. 
Then, 
\begin{enumerate}
    \item[(1)] $\deg(x, f) \in \Z$, 
    \item[(2)] $\deg(x, f)\cdot \deg(f(x), g) = \deg(x, g \circ f),$
    \item[(3)] If $f$ is invertible, $\deg(x, f) \in \{1, -1\}$.
\end{enumerate}
\end{lemma}

A map $f\colon P \to Q$ between two $CW$  posets is a \textbf{cellular map} if  it is order-preserving and, for all $p \in P,$ $f(P^{\leq p}) = Q^{\leq f(p)}.$  
It is shown in~\cite[Lemma~3.5]{MSS} that if $f\colon P\to Q$ is cellular, then the induced simplicial mapping $\Delta(f)\colon \Delta(P)\to \Delta(Q)$ yields a regular cellular map $\Sigma(f)\colon \Sigma(P)\to \Sigma(Q)$.

\subsection{Degrees of the actions of $G$ and $B$ on the semigroup poset}\label{sec:degree-character}
The actions of both $G$ and $B$ (via left multiplication) on the semigroup poset $B$ are by cellular maps.
\begin{lemma}\label{lem:G-B-cellular-maps}
    Let $G$ be a finite group acting on a $CW$ LRB $B$ by semigroup automorphisms. Then, $G$ and $B$ both act on the $CW$ poset $B$ by cellular maps.
\end{lemma}
\begin{proof}
Since $G$ acts by poset automorphisms on $B$, it is clear that $G$ acts by cellular maps on $B$, whereas the action of $B$ on itself is cellular by \cite[Proposition~5.3]{MSS}.
\end{proof}

\subsubsection{Degrees of the action of $B$: choosing ``good'' orientations}
In \cite[Proof of Theorem 5.11]{MSS}, Margolis--Saliola--Steinberg provide a choice of orientations for any $CW$ LRB $B$ so that the left action of $B$ on $\Sigma(B)$ is orientation preserving. 
Namely, they make a choice of orientation $\hat x$ of the cell corresponding to each $x\in B$ so that for $z \in B$,
\begin{equation}\label{eqn:action-of-B-preserves-orientations}
 z_*\hat{x} =\begin{cases}
     \widehat{zx}, & \text{if}\ \sigma(z)\geq \sigma(x)\\
     0, & \text{else.}
 \end{cases}  
\end{equation}
Therefore, with this choice of orientations, we have 
 \begin{align}\label{eqn:deg-one}
    \deg(x, z) = \begin{cases}
        1, & \text{if}\ \sigma(z)\geq \sigma(x)\\
        0, & \text{else.}
    \end{cases} 
 \end{align}

\begin{convention}\rm\label{conv:b-orientations}
    For the remainder of the paper, we assume that we have chosen our orientations as Margolis--Saliola--Steinberg do, so that \cref{eqn:action-of-B-preserves-orientations} and \cref{eqn:deg-one} hold.
\end{convention}

\subsubsection{Degrees of the action of $G$} 
We now show that given \cref{conv:b-orientations} and elements $g \in G$ and $y \in B$, the degree $\deg(y, g)$  only depends on the support $\sigma(y)$ of $y$. We will then study the well-defined degrees $\deg(Y, g)$ for $Y \in \supp(B).$ Importantly, we explain how to compute $\deg(Y, g)$ when $g \in G_{Y}$ via poset topology (\cref{lem:simpler-det-stabilizer}). 
\begin{lemma} \label{lem:deg-well-defined}
    Let $Y \in \supp(B)$ have rank $k.$ Then, for any $y, y' \in B$ with $\sigma(y) = \sigma(y')$ and any $g \in G$
    \begin{align*}
        \deg (y, g) = \deg(y', g).
    \end{align*}
\end{lemma}

\begin{proof}
Since $G$ acts on $B$ by semigroup automorphisms, the following diagram commutes.
\begin{center}
\begin{tikzcd}
    (B^{\leq y}, B^{<y}) \arrow[r, "g"] \arrow[d, "y' \cdot -"] & (B^{\leq g(y)}, B^{<g(y)})\arrow[d, "g(y') \cdot -"]\\
    (B^{\leq y'}, B^{<y'}) \arrow[r, "g"] &(B^{\leq g(y')}, B^{<g(y')})
\end{tikzcd}
\end{center}
By naturality, the induced maps on the relative homology groups also commute. Therefore, by \cref{lem:facts-about-cellular-maps}(2), \[\deg(y, g)\deg(g(y), g(y')) = \deg(y, y')\deg(y', g).\]
By \cref{eqn:deg-one}, $\deg(g(y), g(y')) = \deg(y, y') = 1$, implying that $\deg(y, g) = \deg(y', g)$. 
\end{proof}

Since $g \in G$ is an invertible cellular map on $B$, by \cref{lem:facts-about-cellular-maps} and  \cref{lem:deg-well-defined},  we can define a degree function $\deg\colon \supp(B) \times G \to \{1, -1\}$ by 
 $ \deg(Y, g):= \deg (y, g)$
where $y$ is any element in $B$ with $\sigma(y) = Y.$ 

Crucially, we note that the map \begin{align*}
\deg(Y)\colon \GY &\longrightarrow \{+1,-1\}\\
g &\mapsto \deg(Y, g)
\end{align*} defines a one-dimensional character on $\GY;$  indeed, \cref{lem:facts-about-cellular-maps}(2) and \cref{lem:deg-well-defined} imply
\begin{align*} \deg(Y)(g)\cdot \deg(Y)(g') = \deg(y, g)\cdot \deg(g(y), g') = \deg(y, g'g) = \deg(Y)(g'g).\end{align*}

There is a straightforward way to compute the values of the character $\deg(Y)$ which takes advantage of an unusual action of $\GY$ on $B^{<y}$ for $y$ with $\sigma(y) = Y,$ defined in the following lemma. This action makes sense for any LRB, not just CW LRBs.  We will write this action as the $\ast$-action, $g \ast b$, and will often write \begin{equation}\label{eqn:By-Tilde}
    \widetilde{B^{<y}}
\end{equation} to indicate we are viewing the poset $B^{<y}$ under the $\ast$- action of $\GY.$
\begin{lemma}\label{prop:reinterpretation-degree}
Let $Y \in \supp(B)$ and fix an element $y \in B$ for which $\sigma(y) = Y.$  Then, $\GY$ acts on the pair of posets $(B^{\leq y}, B^{<y})$ by poset automorphisms via the action \[g \ast x := y\cdot g(x).\]
\end{lemma}

\begin{proof} 
 Let $g \in \GY.$ First note that since $\sigma(g(y))=\sigma(y)$, left multiplication by $y$ is a semigroup isomorphism $B^{\leq g(y)}\to B^{\leq y}$ by \cref{lem:action-of-y'-on-By-poset-auts}.  The map $x\mapsto g\ast x$ is the composition of the semigroup isomorphism $g\colon B^{\leq y}\to B^{\leq g(y)}$ with the previous isomorphism, and hence is a semigroup automorphism of $B^{\leq y}$ and thus a poset automorphism.
It remains to show that this is a group action. Observe that for $g, h \in \GY$ and $x \leq y,$
\begin{align*}
   g\ast( h\ast x) = g\ast\left(yh(x)\right) = yg(yh(x)) = yg(y)g(h(x)) = yg(h(x)) = (gh)\ast x
\end{align*}
as $\sigma(y)=\sigma(g(y))$.
\end{proof}

We now give our topological interpretation of the $\GY$-character $\deg(Y)$.

\begin{lemma}\label{lem:simpler-det-stabilizer} 
 Let $Y \in \supp(B)$ and fix an element $y \in B$ for which $\sigma(y) = Y.$ As $G_Y$-representations,
    \begin{align*}
        \deg(Y) \cong \widetilde{H}_{\rk(y) - 1}(\widetilde{B^{<y}}).
    \end{align*}
In other words, $\deg(Y): G_Y \to \{1, -1\}$ maps $g \in G_Y$ to the degree of its $\ast$-action on the sphere $\|\Delta(\widetilde{B^{<y}})\|$.
\end{lemma}
\begin{proof}
Fix some $g \in \GY$ and let $k$ be the rank of $Y.$ 
Since the diagram \begin{center}
    \begin{tikzcd}
        \left(B^{\leq y}, B^{< y}\right) \arrow[d, "g\ast-" left] \arrow[r, "g"] &  \left(B^{\leq g(y)}, B^{< g(y)}\right)  \arrow[dl, "y \cdot -"]\\
         \left(B^{\leq y}, B^{< y}\right).
    \end{tikzcd}
\end{center}
commutes by definition, we know $\deg(y, g)\cdot \deg(g(y), y) = \deg(y, g \ast -).$ Further, since $\deg(g(y), y) = 1$ by \cref{eqn:deg-one}, we have that \[\deg(Y, g) := \deg(y, g) = \deg(y, g \ast -).\] Thus, to prove that $\deg(Y) \cong \widetilde{H}_{k - 1}(\widetilde{B^{<y}}),$ it suffices to show that as $\GY$-representations, $H_{k}(\widetilde{B^{\leq y}}, \widetilde{B^{<y}}) \cong \widetilde{H}_{k - 1}(\widetilde{B^{< y}}).$ 
 Since $\|\Delta(B^{\leq y})\|$ is contractible, the connecting morphism in the long exact sequence of the pair $(B^{\leq y},B^{<y})$  in reduced homology is an isomorphism $H_{k}(\widetilde{B^{\leq y}}, \widetilde{B^{<y}}) \cong \widetilde{H}_{k - 1}(\widetilde{B^{< y}})$ for $k\geq 0$, and moreover it is a $G$-isomorphism  by naturality of the connecting map.
\end{proof}

\subsection{Proof of \cref{intro:thmC}}

We work up to proving \cref{intro:thmC} in the introduction -- the generalization of \cref{prop:double-idem-spaces-homology}.
 \subsubsection{Realizing cellular chain complexes within $CW$ LRB algebras}\label{subsec:realizing-cell-homology-within-B}
 We explain how to realize the cellular chain complex of $\Sigma(B)$ within the LRB algebra $\kk B.$ From now on, we reset our homology coefficient ring $R$ to be our ambient field $ \kk.$

Fix a cfpoi $\{E_X: X\in \supp(B)\}$. \cref{eqn:action-of-B-preserves-orientations} implies there is a $\kk B$-module isomorphism 
\begin{equation} \label{eqn:chain-complex-CB-iso}
    \Phi\colon  C:=\bigoplus_{q\geq 0} C_q(\Sigma(B))=\bigoplus_{\substack{b \in B}}{H}_{\rk(b)}\left(B^{\leq b}, B^{<b}\right) \to  \bigoplus_{\substack{X \in \supp(B)}}\kk B \cdot E_X=\kk B
\end{equation}
given by $\widehat b \mapsto bE_{\sigma(b)}$; see \cite[Eqn.~(6.2)]{MSS} and the ensuing discussion.  Note that $\Phi$ depends upon the choice of cfpoi.   Since $G$-acts on $\Sigma(B)$ by cellular maps, $C$ is a $G$-representation,  where $g\in G$ acts as $g_*$, and the boundary map $d$ on $C$ is $G$-equivariant.   For $X\in \supp(B)$, put 
\[C_X:=\bigoplus_{\substack{\sigma(b)=X}}{H}_{\rk(X)}\left(B^{\leq b}, B^{<b}\right)\] and note that $\Phi$ restricts to a $\kk B$-module isomorphism of this submodule with $\kk B \cdot E_X$. In order to make $\Phi$ respect the action of $G$, we must twist by degrees, as we explain in the following lemma.

\begin{lemma}\label{lem:twist.hom}
Suppose that the chosen cfpoi is $G$-invariant. The restriction
$\Phi_X\colon C_X\otimes \deg(X)\to \kk B \cdot E_X$ of $\Phi$ is a $G_X$-equivariant isomorphism of $\kk B$-modules.    
\end{lemma}
\begin{proof}
    We just need to check equivariance.  Note that if $\sigma(x)=X$,  then \[\Phi(g_*(\hat x)) = \Phi (\deg(X,g)\widehat{g(x)})=\deg(X,g)g(x)E_{g(X)} =\deg(X,g)g(xE_X)=\deg(X,g)g(\Phi(\hat x))\] as required.
\end{proof}

 Define the map $\partial\colon \kk B \to \kk B$ to be the isomorphic image of the cellular boundary map $d$ viewed on the right side of \cref{eqn:chain-complex-CB-iso}, i.e.,
 \begin{equation}\label{eqn:def-partial}
     \partial :=\Phi \circ d \circ \Phi^{-1}.
 \end{equation}
 Since both $\Phi$ and $d$ are $\kk B$-module homomorphisms ($d$ is because $B$ acts cellularly on $\Sigma(B)$), so is $\partial$.

If $H$ is a group acting by automorphisms on an algebra $A$, then an $A$-module $V$ is {\bf $H$-semilinear} if $V$ is an $H$-representation satisfying $h(av)=h(a)h(v)$ for all $h\in H$, $a\in A$ and $v\in V$.
If $X,Y\in \supp(B)$, and $U,V$ are $\kk B$-modules which are $G_X$, $G_Y$-semilinear, respectively, then $\hom_{kB}(U,V)$ is a $G_X\cap G_Y$-representation under the conjugation action $g\star f(u) = gf(g^{-1}u)$. Note that both $\kk B\cdot E_Z$ and $C_Z$ are $G_Z$-semilinear for $Z\in \supp(B)$, the former being obviously so, and the latter following from \cref{lem:twist.hom} and the former case.
\begin{prop}\label{prop:main.isos}
Let $X,Y\in \supp(B)$ with $X\leq Y$. 
Assuming that our cfpoi is $G$-invariant, the vector space isomorphisms
\begin{align*}
    E_Y\cdot \kk B\cdot E_X \cong \hom_{\kk B}(\kk B \cdot E_Y,\kk B \cdot E_X) 
                \cong \deg(X)\otimes \hom_{\kk B}(C_Y,C_X)\otimes \deg(Y)
\end{align*}
are $G_X\cap G_Y$-equivariant, where the first isomorphism is from \cref{prop:yoneda} and the second is $f\mapsto \Phi^{-1}_X\circ f\circ \Phi_Y$.
\end{prop}
\begin{proof}
For the first isomorphism, we check equivariance of the reverse direction $\hom_{\kk B}(\kk B \cdot E_Y, \kk B \cdot E_X) \to E_Y \cdot \kk B \cdot E_X$ mapping $f\in \hom_{\kk B}(\kk B \cdot E_Y,\kk B \cdot E_X)$ to $ f(E_Y).$ If $g\in G_X\cap G_Y$, then $g^{-1}(E_Y) = E_Y$, and so $g\star f(E_Y) = gf(g^{-1}(E_Y)) = gf(E_Y)$, as required. For equivariance of the second isomorphism, \cref{lem:twist.hom} implies that $\Phi_Zg=\deg(Z,g)g\Phi_Z$ for all $Z\in \supp(B)$. It follows that if $f\colon \kk B\cdot E_Y\to \kk B\cdot E_X$ is $\kk B$-linear, then $\Phi_X^{-1}\circ g\star f\circ\Phi_Y= \Phi^{-1}_X \circ g f g^{-1} \circ \Phi_Y= \deg(X,g)\cdot g (\Phi_X^{-1}\circ f\circ \Phi_Y)g^{-1}\cdot\deg(Y,g)=\deg(X,g)\cdot g\star(\Phi_X^{-1}\circ f\circ \Phi_Y)\cdot\deg(Y,g)$ since $\deg(X),\deg(Y)$ are $\pm 1$-valued.
\end{proof}

Let $\pi_X\colon C\to C_X$ be the projection.  If $Y\gtrdot X_{k-1}\gtrdot \cdots \gtrdot X_1\gtrdot X$ is a maximal chain, define 
\begin{equation}\label{eqn:define.scr.F}
   \mathscr F(X,X_1,\ldots, X_{k-1},Y):= \pi_X\circ d\circ \pi_{X_1}\circ d\circ \cdots\circ \pi_{X_{k-1}}\circ d\colon  C_Y\to C_X.
\end{equation}

\begin{prop}\label{prop-gen-end-C}
If $X,Y\in \supp(B)$ with $X\leq Y$, then $\hom_{\kk B}(C_Y,C_X)$ is spanned by the mappings $\mathscr F(X,X_1,\ldots, X_{k-1}, Y)$ subject to the relations that if $X \lessdot  X_1 \lessdot  \cdots \lessdot X_i \color{blue}{<}\normalcolor  X_{i + 2} \lessdot \cdots \lessdot  X_{k-1}\lessdot Y$ with $\rk(X_{i + 2}) - \rk(X_i) = 2$, then
    \begin{align*}
        \sum_{X_i \lessdot X' \lessdot X_{i + 2}}\mathscr F(X,X_1,\ldots,X_i, \color{blue}{X'} \normalcolor, X_{i + 2}, \ldots, X_{k-1}, Y)=0. 
    \end{align*}
\end{prop}
\begin{proof}
   The proof of~\cite[Theorem~6.1]{MSS} shows that there is a choice of cfpoi $\{F_X:X\in \supp(B)\}$ such that $F_Y\cdot \kk B\cdot F_X$ is spanned by the elements
   \[\partial(F_Y)F_{X_{k-1}}\partial(F_{X_{k-1}})F_{X_{k-2}}\cdots \partial(F_{X_1})F_X\] with
   $Y\gtrdot X_{k-1}\gtrdot \cdots \gtrdot X_1\gtrdot X$ a maximal chain subject to the relations
     \begin{align*}
        \sum_{X_i \lessdot X' \lessdot X_{i + 2}}\partial(F_Y)F_{X_{k-1}}\partial(F_{X_{k-1}})F_{X_{k-2}}\cdots \color{blue}{\partial(F_{X_{i+2}})F_{X'}\partial(F_{X'})F_{X_i}}\normalcolor \cdots  \partial(F_{X_1})F_X=0
    \end{align*}
    where $X \lessdot  X_1 \lessdot  \cdots \lessdot X_i \color{blue}{<}\normalcolor  X_{i + 2} \lessdot \cdots \lessdot  X_{k-1}\lessdot Y$ with $\rk(X_{i + 2}) - \rk(X_i) = 2.$  

Let $\rho_Z\colon \kk B\to \kk B \cdot F_Z$ be given by linearly extending $\rho_Z(b) = bF_Z$.  Then for $Z'\geq Z$, we have $\rho_Z\circ \partial(F_{Z'}) = \partial(F_{Z'})F_Z$.   Also note that $\Phi^{-1}\circ \partial\circ \Phi=d$ and $\Phi^{-1}\circ \rho_Z\circ \Phi=\pi_Z$.  It follows from \cref{prop:yoneda} that the composition of the vector space isomorphisms $F_Y\cdot \kk B\cdot F_X\to \hom_{\kk B}(\kk B\cdot F_Y,\kk B\cdot F_X)\to \hom_{\kk B}(C_Y,C_X)$ is given by
\[\partial(F_Y)F_{X_{k-1}}\partial(F_{X_{k-1}})F_{X_{k-2}}\cdots \partial(F_{X_1})F_X\longmapsto \mathscr F(X,X_1,\ldots, X_{k-1},Y)\] for
   $Y\gtrdot X_{k-1}\gtrdot \cdots \gtrdot X_1\gtrdot X$ a maximal chain.  The result follows.
\end{proof}

We finally obtain our generalization of \cref{prop:double-idem-spaces-homology}, which implies \cref{intro:thmC} from the introduction. Note that for supports $X, Y \in \supp(B)$ when we write $\widetilde{H}^{\bullet}_{\supp(B)}(X, Y)$ or $\widetilde{H}^{\bullet}(X, Y)$, we mean the cohomology of the open interval poset $(X, Y)$ within $\supp(B)$. In particular, we do \textit{not} mean \textit{relative} cohomology. 

\begin{thm}\label{cor:CWLRBs-as-poset-topology}
    Let $X \leq Y$ be in $\supp(B)$ with $\mathrm{rk}(Y) - \mathrm{rk}(X) = k.$ As $\GX\cap \GY$-representations, 
    \[E_Y \cdot \kk B \cdot  E_X \cong \deg(X) \otimes \widetilde{H}^{k - 2}_{\supp(B)} \left(X, Y\right) \otimes  \deg(Y).\]
\end{thm}
\begin{proof}
By \cref{prop:main.isos}, it suffices to find a $G_X\cap G_Y$-equivariant isomorphism $\hom_{\kk B}(C_Y,C_X)\cong \widetilde{H}^{k - 2}_{\supp(B)} \left(X, Y\right)$.  
By  \cref{rem:top-cohomology} and \cref{prop-gen-end-C}, there is a $\kk$-vector space isomorphism $\widetilde{H}^{k - 2}_{\supp(B)} \left(X, Y\right)\to \hom_{\kk B}(C_Y,C_X)$ sending the class of a maximal chain $X_1\lessdot \cdots \lessdot X_{k-1}$ in $(X,Y)$ to $\mathscr F(X,X_1,\ldots, X_{k-1},Y)$.  It remains to check $G_X\cap G_Y$-equivariance.  But in light of \cref{eqn:define.scr.F} and the computations $g_*\circ d\circ g_*^{-1} = d$ and $g_*\circ \pi_Z\circ g_*^{-1} = \pi_{g(Z)}$, we have that, for $g \in G_X \cap G_Y$, \begin{align*}g\star \mathscr F&(X,X_1,\ldots, X_{k-1},Y)\\
&= g_\ast \circ \pi_X \circ g_\ast^{-1} \circ g_\ast \circ d \circ g_\ast^{-1} \circ g_\ast \circ \pi_{X_1} \circ g_\ast^{-1} \circ g_\ast  \circ d \circ g_\ast^{-1} \circ \cdots \circ g_\ast \circ \pi_{X_{k - 1}} \circ g_\ast^{-1} \circ g_\ast \circ d \circ g_\ast^{-1}\\
&= \mathscr F(X,g(X_1),\ldots, g(X_{k-1}),Y)\end{align*}
as required.
\end{proof}

By taking dimensions, \cref{cor:CWLRBs-as-poset-topology} recovers the Cartan invariants of $\kk B$ (\cref{t:cartan.lrb}). Indeed, since $(X,Y)$ is Cohen-Macaulay, $\dim_{\kk} \widetilde{H}^{k-2}_{\supp(B)}(X, Y)=|\mu_{\Lambda(B)}(X,Y)|$. Further, \cref{t:cartan.lrb} tells us about the Cartan invariants of $(\kk B)^G.$ Together with \cref{c:cartan.formula} and the fact that $\deg(X)\otimes \deg(Y)$ is a real-valued degree one character, it yields:

\begin{cor}\label{cor:cartan.cw}
Fix $[X']$, $[Y']$ in $\supp(B) / G$ and set $k = \rk(Y') - \rk(X')$.  Then the $(\kk B)^G-$ Cartan invariant \[\dim_{\kk}E_{[Y']}\cdot (\kk B)^G\cdot E_{[X']}=\sum_{\substack{
[X \leq Y] : \\ X \in [X'], Y \in [Y']}}\langle \deg(X)\otimes \deg(Y),\widetilde{H}^{k-2}(X,Y)\rangle_{G_X\cap G_Y}.\]
\end{cor}

\subsection{Using \cref{intro:thmC} to recover the previously known case}
  The goal of this subsection is to explain how specializing \cref{cor:CWLRBs-as-poset-topology} to the case of hyperplane face monoids with symmetry recovers the motivating theory explained in \cref{subsec:motivation-real-arrangements-symmetry}.

Assume that $\mathcal{F}$ is the face monoid of a real, central hyperplane arrangement $\mathscr{A}$ in $V = \mathbb{R}^n$ and let $G$ be a finite subgroup of $GL(V)$ which preserves $\mathscr{A}.$ The following lemma  is (essentially) Lemma 14.28 in Aguiar--Mahajan \cite{Aguiar-Mahajan}. However, since we use slightly different language and we have different conventions, we include a proof for completeness, which in any event is more geometric in nature.
\begin{lemma}\label{lem:cohomology-face-lattice-iso-to-det-V}
    Assume the maximal chains of the graded semigroup poset $\mathcal{F}$ have length $k$ and let $\hat{1}$ be the maximal element in $\mathcal{F}$. Let $\mathcal{O}$ be the intersection of all the hyperplanes --- the maximal element in $\mathcal{L}(\mathcal{A})$. Then, as $G$-representations, 
    \[\widetilde{H}^{k - 1} \left(\mathcal{F}^{<\hat{1}} \right) \cong \det(V / \mathcal{O}).\]
\end{lemma}
\begin{proof}
Replacing $V$ by $V/\mathcal O$ and replacing each hyperplane $H\in \mathcal A$ with $H/\mathcal O$, we may assume without loss of generality that $\mathcal O=0$, i.e., the arrangement is essential.   Also, since $G$ is finite, we may assume that it is acting by orthogonal transformations. Choose unit length normal vectors $f_1,\ldots, f_d$ to the hyperplanes and choose signs so that the normal vector points toward the positive half-space of its corresponding hyperplane.  For $\sigma\in \mathcal F$, viewed as a sign vector, consider the Minkowski sum
\[\mathcal Z_{\sigma} = \sum_{\sigma_i=0}[-f_i,f_i]+\sum_{\sigma_i={-}}\{-f_i\}+\sum_{\sigma_i={+}}\{f_i\}.\]  Then $\mathcal Z:=\mathcal Z_{\hat 1}$ is the zonotope polar to $\mathcal A$, and $\sigma\mapsto \mathcal Z_{\sigma}$ is the isomorphism of $\mathcal F$ with the face poset of $\mathcal Z$ (see~\cite[Eqn.~(2.2)]{MSS}).  Since $G$ preserves $\mathcal A$ and acts orthogonally, $\mathcal Z$ is invariant under the $G$-action, and the above face poset isomorphism  is $G$-equivariant.  It follows that $\widetilde{H}^{k-1}(\mathcal{F}^{<\hat 1})\cong \widetilde{H}^{k-1}(\partial \mathcal Z)\cong \det(V)$, since $\mathcal Z$ is a $G$-invariant polytope containing the origin.
\end{proof}

\begin{lemma}\label{lem:inc-coefficient-products-and-top-cohomology}
    Let $P$ be the face poset of a regular CW decomposition of an $n$-sphere.  Then $\widetilde{H}^{n}(P; \mathbb{Z}) \cong \mathbb Z$ and is generated by the class $[\mathcal{C}]$ of any maximal chain $\mathcal{C}$ of $P$.
\end{lemma}
\begin{proof}
  We know $\widetilde{H}^n(P; \mathbb{Z})\cong \mathbb Z$ as $\|\Delta(P)\|\cong S^n$.
  By definition, $\widetilde{H}^{n}(P; \mathbb{Z})$ is the $\mathbb{Z}$-module spanned by the cohomology classes of maximal chains of $P$ modulo the relations from \cref{rem:top-cohomology}. Since all open rank two intervals in a CW poset have precisely two elements by \cref{lem:cw-poset-rank2-ints}, these relations are generated by those of the form $\mathcal{C} +\mathcal{C}' = 0$ for any maximal chains $\mathcal{C}, \mathcal{C'}$ which differ in precisely one spot. It is well known that any two maximal chains of $P$ are connected by a sequence of maximal chains which differ in only one spot.  It follows that  if $\mathcal C$, $\mathcal C'$ are any chains, then $[\mathcal C]=\pm [\mathcal C']$, and hence any maximal chain generates $\widetilde{H}^n(P;\mathbb{Z})$.
\end{proof}

We are now ready to state the close relationship between the degree and determinant characters.  Notice that $\det(X/\mathcal O)\otimes\det(\mathcal O)=\det(X)$ for any support $X$,  where $\mathcal O$ is the intersection of all the hyperplanes.  Since $\det$ is $\pm 1$-valued on the finite group $G$, we can rewrite this as $\det(\mathcal O)\otimes\det(X)=\det(X/\mathcal O)$.

\begin{prop} \label{prop:relating-det-and-deg} Assume that we are in the situation stated at the beginning of this subsection (real, central hyperplane arrangement face monoid).   For any support $X$, as $\GX$-representations,
    \[\det(X) \otimes \det(V) \cong  \deg(X).\]
\end{prop}
\begin{proof}
Let $\mathcal{O}$ be the intersection of all the hyperplanes --- the maximal element in $\mathcal{L}(\mathcal{A})$. We shall prove the equivalent statement that as $\GX$-representations, $\det(X / \mathcal{O}) \otimes \det(V / \mathcal{O}) \cong  \deg(X)$, or rephrased more helpfully given that all three characters are $\pm 1$-valued, that $\deg(X) \otimes \det(X / \mathcal{O}) \cong \det(V / \mathcal{O}).$

Fix a face $x$ with support $X$ and set $k:= \rk(\hat{1})$ and $j:= \rk(x).$ Applying \cref{lem:simpler-det-stabilizer} and \cref{lem:cohomology-face-lattice-iso-to-det-V} to $\mathcal{F}$ and to $\mathcal{F}_{\geq X}$, it suffices to prove that as $\GX$-representations,
   \begin{align*}
{\widetilde{H}^{j - 1}\left(\widetilde{\mathcal{F}^{<x}}\right)} \otimes \widetilde{H}^{k - j - 1}\left(\mathcal{F}_{\geq X}^{<\hat{1}}\right) \cong \widetilde{H}^{k - 1} \left(\mathcal{F}^{<\hat{1}}\right).
    \end{align*}

Recall the \textit{join} of LRBs from the discussion following \cref{thm:joins}. Define $f\colon \mathcal F^{<\hat 1}\to \mathcal F^{<x}\ast \mathcal F_{\geq X}^{<\hat{1}}$ by
\[f(b) = \begin{cases}b, & \text{if}\ \sigma(b)\geq X\\ xb, & \text{else.}\end{cases}\]
We claim that $f$ is a semigroup homomorphism.  Let us write $\odot$ for the multiplication in the join.  The interesting cases are $f(ab)$ and $f(ba)$ when $\sigma(a)\ngeq X$ and $\sigma(b)\geq X$.  Then $f(a)\odot f(b) = (xa)\odot b = xa$.  On the other hand, $f(ab) = x(ab) = (xa)b=xa$ since $\sigma(b)\geq \sigma(xa)$.  Similarly, $f(b)\odot f(a) = b\odot xa=xa$ and $f(ba) = x(ba) = (xb)a=xa$, as $\sigma(b)\geq \sigma(x)$.  

Note that $G_X$ acts on  $\mathcal F^{<x}\ast \mathcal F_{\geq X}^{<\hat{1}}$ by semigroup automorphisms via 
\[g\cdot b = \begin{cases}g(b), & \text{if}\ b\in \mathcal{F}^{<\hat 1}_{\geq X}\\
g\ast b, & \text{if}\ b\in \mathcal F^{<x}.\end{cases}\]
We claim that $f$ is $G_X$-equivariant.  If $\sigma(b)\geq X$, then $f(g(b)) = g(b) =g\cdot b=g\cdot f(b)$.  If $\sigma(b)\ngeq X$, then $f(g(b)) = xg(b) = xg(x)g(b)= xg(xb) = g\cdot xb = g\cdot f(b)$.
It follows that $f$ induces a surjective $G_X$-equivariant map of semigroup posets.

Now $\Sigma(\mathcal F^{<x}\ast \mathcal F^{<\hat{1}}_{\geq X})\cong S^{j-1}\ast S^{k-j-1}\cong S^{k}$.   Let $x_0 \lessdot \cdots \lessdot x_{j-1} \lessdot x \lessdot y_1 \lessdot \cdots \lessdot y_{k-j - 1}$ be a maximal chain of $\mathcal F^{<1}$ containing $x$.  Then its image under $f$ is itself, viewed as a chain in $\mathcal F^{<x}\ast \mathcal F_{\geq X}^{\hat 1}$, and it is maximal there.  It follows that the induced map $f^\ast$ is a nonzero map between $\widetilde{H}^{k-1}$ of these two $(k-1)$-spheres by \cref{lem:inc-coefficient-products-and-top-cohomology}, and hence is a $G_X$-equivariant isomorphism since these are $1$-dimensional vector spaces.

By \cref{thm:joins}, and its naturality in the two posets, we then have a $G_X$-equivariant isomorphism
\begin{align*}
{\widetilde{H}^{j - 1}\left(\widetilde{\mathcal{F}^{<x}}\right)} \otimes \widetilde{H}^{k - j - 1}\left(\mathcal{F}_{\geq X}^{<\hat{1}}\right) \cong \widetilde{H}^{k - 1} \left(\mathcal F^{<x}\ast \mathcal F^{<\hat{1}}_{\geq X}\right)\cong \widetilde{H}^{k - 1} \left(\mathcal{F}^{<\hat{1}}\right)
\end{align*}
as required.
\end{proof}
Combining \cref{prop:relating-det-and-deg}, \cref{cor:CWLRBs-as-poset-topology}, and  the fact that $\det(V) \otimes \det(V)$ is the trivial character, we recover a result subsuming the motivating \cref{prop:double-idem-spaces-homology}:
\begin{cor} \label{cor:recovering-det-twist-homology}Let $\dim_{\mathbb{R}}X - \dim_{\mathbb{R}}Y = k.$ As $\GY \cap \GX$-representations,
    \begin{align*}
        E_{Y} \cdot  \kk \mathcal{F} \cdot  E_X \cong \deg(X) \otimes  \widetilde{H}^{k - 2}(X, Y) \otimes \deg(Y) \cong \det(X) \otimes \widetilde{H}^{k - 2}(X, Y) \otimes \det(Y).
    \end{align*}
\end{cor}

\subsection{New application of \cref{intro:thmC}: $\catzero$-cube complexes}\label{ch:more-apps}\label{ch:cat0-inv-peirce}
In this section, we apply \cref{intro:thmC} (or equivalently, \cref{cor:CWLRBs-as-poset-topology}) to LRBs associated to $\catzero$-cube complexes. Recall the notion of a \textit{strongly simplicial action} on a $\catzero$-cube complex from \cref{def:lrb-action}. Our main goal is to show that if $G$ is a finite group acting strongly simplicially on a $\catzero$-cube complex $\mathcal{C}$, then the spaces $E_Y \cdot \kk \mathcal{F}(\mathcal{C}) \cdot E_X$ are trivial $\GY \cap \GX$-representations. This implies the invariant Peirce components $E_{[Y]} \cdot \kk \mathcal{F}(\mathcal{C}) \cdot E_{[X]}$ are permutation representations (\cref{cor:intro-cor-c}), which makes the Cartan invariants of $(\kk \mathcal{F}(\mathcal{C}))^G$ especially straightforward to compute (\cref{cor:Cartan-invariant}).

\subsubsection{Lemmas}
We first point out a couple of simple lemmas, both of which are well known and straightforward to prove, as they amount to groups of orthogonal transformations preserving a polytope containing the origin (choosing the correct geometric realization).

\begin{lemma}\label{lem:general-cube}
    Let $S_n$ act on the boundary of the cube $[0, 1]^n$ by permuting the $n$ coordinates. Then, the top reduced homology of the boundary of the cube, $\widetilde{H}_{n-1}^{\mathrm{sing}}(\partial([0, 1]^n))$, carries the sign representation of $S_n$.
\end{lemma}
\normalcolor
A similar result holds for the top cohomology of the proper part of the Boolean lattice, as its order complex is the boundary of a simplex.
\begin{lemma}\label{lem:general-boolean}
     Let $S_n$ act on the Boolean lattice $B_n$ in the expected manner. Then, the top reduced cohomology of the poset $B_n\setminus \{\hat{0}, \hat{1}\}$, $\widetilde{H}^{n-2}\left(B_n \setminus \{\hat{0}, \hat{1}\}\right)$, carries the sign representation of $S_n.$  
\end{lemma}

We now are ready to compute the degree character $\deg(Y, g)$ for $\mathcal{F}(\mathcal{C})$. Recall that there is a \textit{unique} set of hyperplanes which intersect to form $Y \in \mathcal{L}(\mathcal{C})\cong \supp(B)$. We write $h(Y)$ for this set.  If $G$ acts on a set $\Omega$, we use $\sgn^{\Omega}(g)$ to denote the sign of $g\in G$ as a permutation of $\Omega$.  If $\Gamma\subseteq \Omega$ is $G$-invariant, then $\sgn^{\Omega}(g)=\sgn^{\Omega\setminus \Gamma}(g)\sgn^{\Gamma}(g)$. 

\begin{lemma} \label{lem:degree-catzero} Assume that $G$ acts on a $\catzero$-cube complex $\mathcal{C}$ by a strongly simplicial action. Then, for any $Y \in \mathcal{L}(\mathcal{C}) \cong \Lambda(\mathcal{F}(\mathcal{C}))$ and any $g \in \GY,$
$\deg(Y, g) = \sgn^{h(Y)}(g).$
    \end{lemma}
    \begin{proof} 
   Fix $y$ with $\sigma(y) = Y$ and set $k:=\rk(y).$ Recall the notion of the $\ast$-action on $\widetilde{\mathcal{F}(\mathcal{C})^{<y}}$ from \cref{eqn:By-Tilde}. By \cref{lem:simpler-det-stabilizer}, 
        \[ \deg(Y) \cong \widetilde{H}_{k - 1}(\widetilde{{\mathcal{F}(\mathcal{C})^{<y}})}.\] 

Since $G$ acts by a strongly simplicial action, there exists a vertex $v \in \mathcal{C}$ which is fixed by all $g \in G$ by \cref{lem:fixed-cubes-by-support}. Without loss of generality, assume that we have oriented our hyperplanes so that $\sgn_H(v) = -$ for all hyperplanes $H.$ We claim that for $x < y$,
\begin{align}\label{eqn:permuting-signs}
    \sgn_H(g\ast x) = \sgn_{H}(yg(x)) = \begin{cases}
     \sgn_H(y) &\text{ if }H \notin h(Y)\\
        \sgn_{g^{-1}(H)}(x) & \text{ if }H \in h(Y).
    \end{cases}
\end{align}
The first equality is by definition of the $\ast$-action. For the second equality, the first case follows from the definition of multiplication in $\mathcal{F}(\mathcal{C})$, after observing that $\sgn_H(y) \neq 0$ for $H \notin h(Y).$ For the second case, since $H \in h(Y)$ implies $\sgn_H(y) = 0$, we have $\sgn_H(yg(x)) = \sgn_H(g(x)).$ If $\sgn_H(g(x)) = 0$, then $\sgn_{g^{-1}(H)}
(x) = 0$ too. So, assume $\sgn_H(g(x)) \neq 0$. Observe that $\sgn_H(g(x)) = \sgn_H(v)$ if and only if $\sgn_{g^{-1}(H)}(x) = \sgn_{g^{-1}(H)}(g^{-1}(v))=\sgn_{g^{-1}(H)}(v)$. Since $\sgn_H(v) = - = \sgn_{g^{-1}(H)}(v)$, we conclude $\sgn_{H}(g(x)) = \sgn_{g^{-1}(H)}(x).$

If we order the hyperplanes in $h(Y)$ as $H_1, H_2, \ldots, H_k$, then there is a permutation $\sigma \in S_k$ for which $g^{-1}(H_i) = H_{\sigma(i)}$ for all $1 \leq i \leq k.$ 
\cref{eqn:permuting-signs} implies that the $\ast$-action of $g$ on $\widetilde{\mathcal{F}(\mathcal{C})^{<y}}$ can be identified wih the induced action of $\sigma$ on the face poset of the boundary of the $k$-unit cube: $\mathcal{P}(\partial([0, 1]^k)) = \mathcal{P}([0, 1]^k) \setminus \hat{1}.$ Thus, by \cref{lem:general-cube}   
    \[ \deg(Y, g) = \sgn (\sigma) = \sgn^{h(Y)}(g)\]
as required.
    \end{proof}

    \begin{lemma} \label{lem:intervals}
   Assume that $G$ acts on a $\catzero$-cube complex $\mathcal{C}$ by a strongly simplicial action. Let $X \leq Y$ in $\mathcal{L}(\mathcal{C}) \cong \supp(\mathcal{F}(\mathcal{C}))$ with $k:=\rk(Y) - \rk(X)$. Then as a $G$-representation,  $\widetilde{H}^{k - 2}(X, Y)\cong \deg(X)\otimes \deg(Y)$
    \end{lemma}
    \begin{proof}
    Observe that interval $\left(X, Y\right)$ is isomorphic to the interior elements of the Boolean lattice of rank $k$ with atoms the $k$ hyperplanes in $h(Y) \setminus h(X)$. The atoms are permuted by $g$, so $g$ can be realized as a permutation in $S_k.$ By  \cref{lem:general-boolean}, the action of $g_\ast$ on the top cohomology is scaling by $\sgn^{h(Y)\setminus h(X)}(g)=\sgn^{h(Y)}(g)\cdot \sgn^{h(X)}(g)=\deg(X)(g)\cdot\deg(Y)(g)$ by \cref{lem:degree-catzero}.
    \end{proof}

\subsubsection{\cref{cor:intro-cor-c} and its implications}
By applying \cref{lem:intervals} 
along with \cref{prop:invariant-Peirce-decomp} and \cref{cor:CWLRBs-as-poset-topology}, we obtain the following result, written as \cref{cor:intro-cor-c} in the introduction.
    
\begin{cor}\label{cor:Peirce-decomp-rep}
       Assume that $G$ acts on a $\catzero$-cube complex $\mathcal{C}$ by a strongly simplicial action. Let $X \leq Y$ in $\supp(\mathcal{F}(\mathcal{C})).$ Then, as $\GX \cap \GY$-representations,
        \[E_Y \cdot \kk\mathcal{F}(\mathcal{C}) \cdot  E_X \cong \mathbb{1}.\]
        Therefore, as $G$-representations, \begin{align*}
            E_{[Y]}  \cdot \kk \mathcal{F}(\mathcal{C})  \cdot E_{[X]} 
             \cong \bigoplus_{\substack{[X' \leq Y']:\\ X' \in [X], Y' \in [Y]}} \kk \left[ G / (G_{X'} \cap G_{Y'})\right]\cong \kk\{ X\leq Y: X' \in [X], Y' \in [Y]\}.
        \end{align*}
    \end{cor}

Since the trivial representation appears in a permutation representation with multiplicity the number of orbits (or by applying \cref{lem:degree-catzero}, \cref{lem:intervals} and \cref{cor:cartan.cw}),
this supplies an especially nice interpretation of the Cartan invariants of $(\kk \mathcal{F}(\mathcal{C}))^G.$

\begin{cor}\label{cor:Cartan-invariant}
  Assume that $G$ acts on a $\catzero$-cube complex $\mathcal{C}$ by a strongly simplicial action. Then, the composition multiplicity of the $\left(\kk \mathcal{F}(\mathcal{C})\right)^{G}$-simple $M_{[Y]}$ in its projective indecomposable $P_{[X]}$ is \[\left[P_{[X]}: M_{[Y]}\right] = \#\{G\text{\rm-orbits} \ [X' \leq Y'] \ : X' \in [X], Y' \in [Y]\ \}.\]
\end{cor}

Note that $G$-orbits $[X'\leq Y']$ with $X'\in [X]$ and $Y'\in [Y]$ are in bijection with double cosets $G_XgG_Y$ such that $X\leq gY$ via $G_XgG_Y\mapsto [X\leq gY]$ with inverse $[hX\leq gY]\mapsto G_Xh^{-1}gG_Y$.

\begin{example}[Cartan invariants for $(\kk \mathscr{C}_n)^{D_{2n}}$ and $(\kk \mathscr{C}_n)^{C_n}$]\rm
Recall from \cref{ex:semisimple-commutative-invariant-rings} that there are three $D_{2n}$-orbits of $\supp(\mathscr{C}_n)$: the set $[\emptyset]$ containing the single empty intersection, the set $[h]$ containing each individual hyperplane, and the set $[i]$ containing all the nonempty intersections between two hyperplanes. The same is true of the $C_n$-orbits of $\supp(\mathscr{C}_n)$, where  $C_n$ is the cyclic subgroup of rotations. The Cartan matrices for $(\kk \mathscr{C}_n)^{D_{2n}}$ and $(\kk \mathscr{C}_n)^{C_n}$ can easily be computed using \cref{cor:Cartan-invariant} and are given in Figure~\ref{fig:cartan}.
\begin{figure}[htbp]
\begin{minipage}{.4\textwidth}
\centering
\begin{tabular}{c|c|c|c}
     &     $P_{[\emptyset]}$  & $P_{[h]}$ & $P_{[i]}$ \\ \hline
    $M_{[\emptyset]}$ & $1$& $0$ & $0$\\ \hline
    $M_{[h]}$ &$1$ & $1$ & $0$ \\ \hline
    $M_{[i]}$ & $1$ & $1$ & $1$
\end{tabular}
\end{minipage}
\begin{minipage}{.4\textwidth}
\centering
\begin{tabular}{c|c|c|c}
     & $P_{[\emptyset]}$  & $P_{[h]}$ & $P_{[i]}$ \\ \hline
    $M_{[\emptyset]}$ & $1$& $0$ & $0$\\ \hline
    $M_{[h]}$ &$1$ & $1$ & $0$ \\ \hline
    $M_{[i]}$ & $1$ & $2$ & $1$
\end{tabular}
\end{minipage}
\caption{Cartan matrices for $(\kk\mathscr {C}_n)^{D_{2n}}$ (left) and $(\kk\mathscr {C}_n)^{C_{n}}$ (right).\label{fig:cartan}}
\end{figure}
\end{example}

\section{Invariant Peirce components for hereditary LRBs}\label{sec:HereditaryLRBs}

Recall that the class of hereditary LRBs includes meet-semilattices as well as the LRBs arising from geometric lattices. Though the invariant Peirce components for semilattices are straightforward (\cref{prop:meet-semilattice-invariant-peirce}), they are more interesting for geometric lattice LRBs.

The goal of this section is to produce a general formula (\cref{intro:thmD} or \cref{thm:hereditary-invariant-peirce}) for the $G$-representation structure of the invariant Peirce components of hereditary LRBs in terms of poset topology. Then, in \cref{subsec:thmE}, we apply \cref{intro:thmD} to compute the invariant Peirce components for geometric lattice LRBs. As a byproduct, we obtain a geometric lattice generalization of the derangement representation. We analyze its connection to a  generalization of ``random-to-top'' shuffling in \cref{sub:randomtotop}.

We continue to assume that $\Char(\kk)\nmid |G|$.

\subsection{Lemmas towards \cref{intro:thmD}}
To prove \cref{intro:thmD} (or \cref{thm:hereditary-invariant-peirce}), we shall need a straightforward lemma (\cref{lem:peirce-to-rad-series}) as well as two quite technical lemmas: \cref{lem:hereditary-paths-quiver} and  \cref{lem:Paths-with-homology-of-semigroup-poset}. The reader may want to skip over the proofs of these lemmas at first read, and then return later.

We point out that for all $i \geq 0,$ \[\rad^i(\kk B) = \underbrace{\rad(\rad(\ldots (\rad(}_{i \text{ times}}\kk B))\ldots)) = \underbrace{\rad(\kk B) \cdot  \rad(\kk B) \cdots\rad(\kk B)}_{i \text{ times}} \] is a $G$-representation as the radical is preserved by automorphisms. 
\begin{lemma}\label{lem:peirce-to-rad-series}
    Let $B$ be a connected LRB and let $X, Y \in \supp(B)$ with $X < Y.$ Assume $n + 1$ is the nilpotency index of the radical of $\kk B$, so that $\rad^{n}(\kk B) \neq 0$ but $\rad^{n + 1}(\kk B) = 0.$ Then, as $G_X \cap G_Y$-representations,
    \begin{align*}
        E_Y \cdot \kk B \cdot E_X \cong \bigoplus_{m = 1}^{n + 1} E_Y \cdot {\left(\rad^{m - 1}(\kk B) / \rad^{m}(\kk B)\right)} \cdot E_X.
    \end{align*}
\end{lemma}

\begin{proof}
Observe that there is a $G_X \cap G_Y$-stable filtration
\begin{align*} 0 = E_Y\cdot\rad^{n + 1}(\kk B)\cdot E_X \subset E_Y \cdot\rad^{n}(\kk B) \cdot E_X \subset \cdots  \subset E_Y \cdot \rad^1(\kk B)\cdot E_X \subset   E_Y \cdot \kk B \cdot  E_X.\end{align*}
The lemma follows. 
\end{proof}

We now prove our first technical lemma, which uses some of the theory of quivers discussed in \cref{sec:Elementary-algebras-quivers-hereditary}.
\begin{lemma}\label{lem:hereditary-paths-quiver}
Let $B$ be a $\kk$-hereditary LRB and let $X < Y \in \supp(B)$. For $n \geq 0$, as $G_Y \cap G_X$-representations,
    \begin{align*}
        E_Y \cdot {\left(\rad^n(\kk B) / \rad^{n+1}(\kk B)\right)} \cdot E_X \cong \bigoplus_{\substack{G_X \cap G_Y\text{\rm -orbits}\\ [X = X_0 < \cdots < X_n = Y] }} \ \left(\bigboxtimes_{i = 1}^n  E_{X_i} \cdot {\left(\rad(\kk B) / \rad^2(\kk B)\right)} \cdot E_{X_{i - 1}} \right) \Bigg \uparrow_{\bigcap_i G_{X_i}}^{G_X \cap G_Y}.
    \end{align*}
\end{lemma}
\begin{proof}
If $Q$ is a finite acyclic quiver and $J$ is the arrow ideal of $\kk Q$, then $J^n/J^{n+1}$ has basis the cosets of the paths of length $n$, and so we can canonically identify the arrows of $Q$ with a basis of $J/J^2$.  Let us write $v\prec w$ if there is a nonempty path from $v$ to $w$.  Then there is a well-known vector space isomorphism
\[\Psi\colon w(J^n/J^{n+1})v\to \bigoplus_{v=v_0\prec v_1\prec \cdots \prec v_n=w} \bigotimes_{i=0}^{n-1} v_{n-i}(J/J^2)v_{n-i-1}\] sending a path $e_n\cdots e_1$ to $\overline {e_n}\otimes\cdots\otimes \overline {e_1}$.  

Let $\mathcal{Q}(B)$ be the quiver of $\kk B.$ Recall that the vertices $v_X$ are indexed by supports $X \in \supp(B)$ and for any $X, Y$ the arrows $v_Y \xleftarrow[]{b_{X, Y, j}}v_X$ can be thought of as being labelled by a set of elements \[\left \{b_{X, Y, j}: 1 \leq j \leq \dim E_Y\cdot {\left(\rad(\kk B / \rad^2(\kk B)\right)}\cdot E_X \right \} \subseteq E_Y\cdot \rad(\kk B) \cdot E_X\] which descend to a basis of $E_Y\cdot {\left(\rad(\kk B / \rad^2(\kk B)\right)}\cdot E_X.$  Since $\kk B$ is \textit{hereditary},  \cref{thm:quiver-elementary} and the discussion thereafter imply $\mathcal Q(B)$ is acyclic and there is an \textit{algebra} \textit{isomorphism} $\varphi\colon \kk B \longrightarrow \kk \mathcal{Q}$ which sends $E_X$ to $v_X$, sends $b_{X, Y, j}$ to the arrow $v_Y \xleftarrow[]{b_{X, Y, j}}v_X$, and has the property that $\varphi(\rad(\kk B)) = J,$ the arrow ideal.  
Using $\varphi$ to identify $\kk B$ with $\kk \mathcal{Q}$, we can transport $\Psi^{-1}$ to an isomorphism
\begin{align*}
    \Phi\colon  \bigoplus_{{X = X_0 < \cdots < X_n = Y}} \bigboxtimes_{i = 0}^{n-1} E_{X_{n-i}}\cdot  {\left(\rad(\kk B) / \rad^2(\kk B)\right)} \cdot E_{X_{i - 1}}  \longrightarrow E_Y \cdot {\left(\rad^n(\kk B) / \rad^{n+1}(\kk B)\right)} \cdot E_X
\end{align*}
which is the $\kk$-linear extension of $\overline{u_n} \otimes \overline{u_{n-1}} \otimes \cdots \otimes \overline{u_1}\longmapsto\overline{u_n u_{n-1}\cdots u_1}$.
Finally, we show that  $\Phi$ is $\GX \cap \GY$-equivariant. Let $g \in \GX \cap \GY.$ Then,
        \begin{align*}
             \Phi \left(g\left( \overline{u_n} \otimes \cdots \otimes \overline{u_1}\right)\right) 
             = \Phi \left(\overline{g(u_n)} \otimes \cdots \otimes \overline{g(u_1)}\right)
             = \overline{g(u_n) \cdots g(u_1)}
             = g \left(\overline{u_n\cdots u_1}\right)
             = g\left(\Phi( \overline{u_n} \otimes \cdots \otimes \overline{u_1})\right)
        \end{align*}
as required.
\end{proof}

\begin{definition}\rm
    We write $\pi_0\left(B_{\geq X}^{<y}\right)$ for the set of connected components of the Hasse diagram of the poset $B_{\geq X}^{<y}$. Analogously, for $a \in B_{\geq X}^{<y}$, we write $\pi_0(a)$ to represent the connected component of $a$ within $\pi_0(B_{\geq X}^{<y}).$ Later, we will even use $\pi_0(a) - \pi_0(a')$ to refer to the associated (reduced) cycle in $\widetilde{H}_0(B_{\geq X}^{<y}).$
\end{definition}

To prove our second technical lemma, we need the following result, which is essentially 
\cite[Lemma~3.20]{MSS}\footnote{\cite[Lemma~3.20]{MSS} states a less precise result but the proof gives our \cref{lem:connected-components-semigroup-poset}.}. 

\begin{lemma}\label{lem:connected-components-semigroup-poset}
    Let $X \leq \sigma(y).$ Each connected component in $\pi_0\left(B_{\geq X}^{<y}\right)$ contains an element with support $X.$ Additionally, two elements $x, x'$ with support $X$ belong to the same connected component of $B_{\geq X}^{<y}$ if and only if there exists a sequence \[x = \ell_0 < u_1 > \ell_1 < u_2 < \cdots >\ell_{k - 1} < u_{k} > \ell_k = x'\] in $B$ with each $\ell_i\in \sigma^{-1}(X)\cap B^{<y}$ and each $u_i\in B^{<y}$. 
\end{lemma}

We are now ready to prove our second technical lemma. 

\begin{lemma}\label{lem:Paths-with-homology-of-semigroup-poset}
    Let $B$ be any connected LRB, let $X < Y$ be two supports in $\supp(B)$, and let $y \in B$ with $\sigma(y) =  Y.$ Then, as $G_X \cap G_Y$-representations,
    \[\widetilde{H}_0 \left(\widetilde{B_{\geq X}^{<y}}\right) \cong E_Y \cdot {\left(\rad(\kk B) / \rad^2(\kk B)\right)} \cdot E_X.\]
\end{lemma}
\begin{proof}
Fix $x \in B_{\geq X}^{<y}$ with $\sigma(x) = X.$ By \cref{lem:connected-components-semigroup-poset}, the \textit{set} \[\left\{\pi_0(x') - \pi_0(x) \ : \ \sigma(x') = X, \ x' < y,  \text{ and }\pi_0(x') \neq \pi_0(x)\right\}\] forms a $\kk$-linear basis of $\widetilde{H}_0 \left(\widetilde{B_{\geq X}^{<y}}\right).$ We claim that the map
\begin{align*}
\Psi\colon \widetilde{H}_0 \left(\widetilde{B_{\geq X}^{<y}}\right) &\to E_Y \cdot {\left(\rad(\kk B) / \rad^2(\kk B)\right)} \cdot E_X\end{align*} 
which linearly extends 
\begin{align*}
\pi_0(x') - \pi_0(x) \longmapsto \overline{E_Y  {\left(x' -  x  \right)}  E_X}
\end{align*}
is a well-defined $(G_X \cap G_Y)$-module isomorphism. Since $\rad(\kk B)=\ker \sigma$, 
the map $\Psi$ does indeed have an appropriate target space.  
Also, if $\sigma(x_1)=X=\sigma(x_2)$, then $\Psi(\pi_0(x_1)-\pi_0(x_2)) = \overline{E_Y(x_1-x_2)E_X}$, assuming that $\Psi$ is well-defined.

To show that $\Psi$ is well-defined, we just need to show that if $x_1,x_2$ have support $X$, then $\pi_0(x_1)=\pi_0(x_2)$ implies that $E_Y(x_1-x_2)E_X\in \rad^2(\kk B)$.  By \cref{lem:connected-components-semigroup-poset}, it suffices to handle that case that $x_1, x_2<u$ with $u<y$ (via a telescoping sum argument).  Note that $\sigma(x_1)=\sigma(x_2)$ implies $(x_1-x_2)E_X\in \rad(\kk B\cdot E_X)$, as $\rad (\kk B)=\ker \sigma$. We verify that  $E_Y(x_1-x_2)E_X\in \rad^2(\kk B\cdot E_X)$.  Recall that $M_Z$ is the simple $\kk B$-module associated to $Z \in \supp(B)$. Let $\rho\colon \kk B\cdot E_X\to M_Z$ be a homomorphism.  It suffices to show that $E_Y(x_1-x_2)E_X\in \ker \rho$. But \[\rho(E_Y(x_1-x_2)E_X) = \rho(E_Yu(x_1-x_2)E_X) = E_Yu\rho((x_1-x_2)E_X) =0\]
since $E_Y$ annihilates $M_Z$ unless $Z=Y$, in which case $u$ annihilates it.

    We now show $\Psi$ is {$G_X \cap G_Y$-equivariant.} Let $g \in G_X \cap G_Y.$ Then,
    \begin{align*}
        \Psi\left(g\ast \left(\pi_0(x') - \pi_0(x)\right)\right) &= \Psi \left(\pi_0(y \cdot g(x')) - \pi_0(y \cdot  g(x))\right)\\
        &= \overline{E_Y   y\cdot g(x' - x)  \ E_X}\\
        &=  \overline{E_Y  g(x' - x)   E_X} \tag{$E_Yy = E_Y$}\\
        &= \overline{g \left(E_Y  (x' - x) E_X\right)}\\
        &= g \left[\Psi\left(\pi_0(x') - \pi_0(x)\right)\right].
    \end{align*}

Finally, we prove $\Psi$ is a vector space isomorphism. Since \[\rad(\kk B) \cdot E_X = \rad(\kk B\cdot E_X) = \mathrm{span}_\kk \{(x' - x)E_X: \sigma(x') = X, x' \neq x\},\] it follows that \begin{align*}
        E_Y \cdot {\rad(\kk B)}\cdot E_X = \mathrm{span}_{\kk}\{E_Y   (x' - x)  E_X : \sigma(x') = X, x' \neq x\}.
    \end{align*}
Since $\pi_0(x') = \pi_0(x)$ implies  $E_Y  (x'-x)  E_X \in E_Y \cdot \rad^2(\kk B)\cdot E_X,$ we have that  \begin{align*}E_Y \cdot {\left(\rad(\kk B) / \rad^2(\kk B)\right)} \cdot E_X
= \mathrm{span}_{\kk}\{\overline{E_Y   (x' - x) E_X }: \sigma(x') = X, \pi_0(x') \neq \pi_0(x)\}.\end{align*}
    
Since $\Psi(\pi_0(x') - \pi_0(x)) = \overline{E_Y   (x' - x)  E_X },$ the map $\Psi$ is surjective. Thus, it suffices to prove that \[\dim_\kk \widetilde{H}_0 \left(\widetilde{B_{\geq X}^{<y}}\right) = \dim_\kk E_Y  (\rad(\kk B) / \rad^2(\kk B)) E_X.\] However, note that  \[\dim_\kk E_Y \cdot (\rad(\kk B) / \rad^2(\kk B))\cdot   E_X = \dim_\kk \mathrm{Ext}_{\kk B}^{1}(M_X, M_Y)\] (see, for instance, \cite[p60]{MSS}). Our proof is then completed by \cite[Theorem 5.18]{MSS}, which establishes that $\dim_\kk \mathrm{Ext}_{\kk B}^{1}(M_X, M_Y) = \dim_\kk \widetilde{H}^0(B_{\geq X}^{< y})$.
\end{proof}

\subsection{\cref{intro:thmD}}
We are now ready to state and prove our formula for the invariant Peirce components for hereditary LRBs, which is called \cref{intro:thmD} in the Introduction. We begin with understanding the spaces $E_Y \cdot \kk B \cdot E_X.$
\begin{thm}\label{thm:hereditary-invariant-peirce}
 Let $B$ be a $\kk$-hereditary LRB and let $X, Y \in \supp(B)$ with $X < Y$. For each support $Z$ with $X < Z \leq Y,$ fix an element $b_Z$ with $\sigma(b_Z) = Z.$ As $G_X \cap G_Y$-representations,
    \[E_Y  \cdot \kk B \cdot  E_X \cong \bigoplus_{m \geq 1}\bigoplus_{\substack{G_X \cap G_Y\text{\rm -orbits}\\ [X= X_0 < X_1 < \cdots < X_m = Y]}}  \left(\bigboxtimes_{i = 1}^m \widetilde{H}_0 \left(\widetilde{B_{\geq X_{i-1}}^{<b_{X_i}}} \right)\right) \Bigg \uparrow_{\bigcap_i G_{X_i}}^{G_X \cap G_Y}. \]\end{thm}
\begin{proof}
Let $n + 1$ be the nilpotency index of $\rad(\kk B).$ By \cref{lem:peirce-to-rad-series}, \cref{lem:hereditary-paths-quiver}, and \cref{lem:Paths-with-homology-of-semigroup-poset}, as $(\GY \cap \GX)$-representations
    \begin{align*}
        E_Y \cdot \kk B \cdot  E_X &\cong \bigoplus_{m = 1}^{n + 1} E_Y \cdot {\left(\rad^{m - 1}(\kk B) / \rad^{m}(\kk B)\right)} \cdot E_X\\
            &\cong \bigoplus_{m = 1}^{\ell} \bigoplus_{\substack{G_X \cap G_Y\text{\rm -orbits}\\ [X = X_0 < \cdots < X_{m} = Y]}} \ \left({\bigboxtimes_{i = 1}^{m}}E_{X_{m+1-i}} \cdot {\left(\rad(\kk B) / \rad^2(\kk B)\right)}\cdot E_{X_{m-i}}\right) \Bigg \uparrow_{\bigcap_i G_{X_i}}^{G_X \cap G_Y}\\
        &\cong \bigoplus_{m = 1}^\ell \bigoplus_{\substack{G_X \cap G_Y\text{\rm -orbits}\\ [X = X_0 < \cdots < X_{m} = Y]}} \ \left({\bigboxtimes_{i = 1}^{m}} \widetilde{H}_0 \left(\widetilde{B_{\geq X_{i - 1}}^{<b_{X_i}}}\right) \right)\Bigg \uparrow_{\bigcap_i G_{X_i}}^{G_X \cap G_Y}.
    \end{align*}
\end{proof}
 \cref{intro:thmD} (restated below) follows from \cref{thm:hereditary-invariant-peirce} and \cref{cor:Peirce-decomp-rep}.

 \begin{thm}\label{thm:thmD-restated}
Let $B$ be a $\kk$-hereditary LRB and let $[X], [Y] \in \supp(B) / G$ with $[X] < [Y]$. For each support $Z$ with $[X] < [Z] \leq [Y],$ fix an element $b_Z$ with $\sigma(b_Z) = Z.$ Then, as $G$-representations,  \[ E_{[Y]} \cdot \kk B \cdot E_{[X]} \cong \bigoplus_{\substack{m \geq 1, \\ [X_0 < X_1 < \cdots < X_m]:\\ X_0 \in [X], X_m \in [Y]}}  \left({\bigboxtimes_{i = 1}^m} \widetilde{H}_0 \left(\widetilde{B_{\geq X_{i-1}}^{<b_{X_i}}} \right)\right) \Bigg \uparrow_{\bigcap_i G_{X_i}}^{G}.\]
\end{thm}

\subsection{\cref{intro:thmE} and generalized derangement representations}\label{subsec:thmE} 
\cref{intro:thmD} simplifies nicely for geometric lattice LRBs, in which case it has connections to various generalizations of derangements. In this section, we give more background on the generalized derangement representations from \cref{subsec:intro-derangement-reps} and prove \cref{intro:thmE}.

Throughout this subsection, let $\mathcal{L}$ be a geometric lattice and let $G$ be a group that acts on $\mathcal{L}$ by poset automorphisms. Recall that the support semilattice of the associated LRB $S(\mathcal{L})$ can be identified with $\mathcal{L}^{\mathrm{opp}}.$ {We shall hence use the elements of $\mathcal{L}$ as well as the \textit{ordering} of $\mathbf{\mathcal{L}}$} when indexing our invariant Peirce components. In particular, when we write $<_\mathcal{L}$, $\hat{1}_{\mathcal{L}}$, and $\hat{0}_{\mathcal{L}}$, we refer to the ordering, maximum, and minimum in $\mathcal{L}$, and not in $\supp(S(\mathcal{L})) \cong \mathcal{L}^{\mathrm{opp}}$.
For $X <_\mathcal{L} Y,$ we write $[X, Y]$ to denote the interval subposet of $\mathcal{L}$ consisting of all $Z$ with $X \leq_{\mathcal{L}} Z \leq_{\mathcal{L}} Y$; note that such closed intervals are geometric lattices in their own right. We write $(X, Y)$ to denote the corresponding open interval. 

\subsubsection{Brown's generalized derangement numbers}\label{subsubsec:browns-derangement-nums}
In \cite[Appendix C]{BrownonLRBs}, Brown defined a \textbf{generalized derangement number} $d_{\mathcal{L}}$,  which can be defined in a variety of ways:
\begin{enumerate}
    \item Recursively (Brown's original definition \cite[Equation (39)]{BrownonLRBs}): Letting $\#\mathcal{F}(\mathcal{L})$ be the number of complete flags (or maximal chains) of $\mathcal{L}$, define $d_{\mathcal{L}}$ so that it satisfies
\begin{align}\label{eqn:Browns-def}
    \sum_{X \in \mathcal{L}}d_{[X, \hat{1}_{\mathcal{L}}]} = \#\mathcal{F}(\mathcal{L}).
\end{align}
\item Applying M\"{o}bius inversion to \cref{eqn:Browns-def}, we get an explicit, but \textit{signed}, formula for $d_{\mathcal{L}}:$
\begin{align}\label{eqn:derangement-mobius-inversion}
   d_{\mathcal{L}} = \sum_{X \in \mathcal{L}} \mu(\hat{0}_{\mathcal{L}}, X) \cdot \# \mathcal{F}([X, \hat{1}_{\mathcal{L}}]).
\end{align}
\item Explicitly, in a way that manifestly demonstrates the nonnegativity of $d_{\mathcal{L}}$ (follows from a straightforward induction argument on \cite[Proposition 10]{BrownonLRBs}): $d_\mathcal{L} = 1$ if $\mathcal{L}$ has one element, and otherwise \begin{align}\label{eqn:manifestly-positive}
    d_\mathcal{L} = \sum_{\substack{\ell \geq 0, \\ \hat{0}_{\mathcal{L}} = X_0 <_{\mathcal{L}} \cdots <_{\mathcal{L}} X_{\ell} <_{\mathcal{L}} \hat{1}_{\mathcal{L}}}}\prod_{i = 0}^\ell (a(X_i,X_{i+1}) - 1),
\end{align}
where $a(X_i,X_{i+1})$ is the positive number given by the number of \textit{atoms} of the interval $[X_i,X_{i+1}]$. 
\end{enumerate}

Brown's motivation for the generalized derangement numbers came from the spectra of generalizations of the \textit{Tsetlin library} Markov chain to geometric lattice LRBs. We say just a bit about this motivation here in order to reinterpret generalized derangement numbers in terms of Peirce components. We will briefly return to random walks with a special case of the Tsetlin library in \cref{sub:randomtotop}.

Consider a random walk on the space of \textit{complete flags} of a geometric lattice $\mathcal{L}$ that is given by multiplying by an element $x \in \mathbb RS(\mathcal{L})$ whose coefficients with respect to the standard basis give a probability distribution on  $S(\mathcal{L})$. Brown's general theory for random walks on LRBs (\cite[Theorem 1]{BrownonLRBs}) implies that (i) this random walk is diagonalizable and (ii) the eigenvalues $\lambda_X$ of such a walk are indexed by the elements of the lattice. Moreover, the eigenvalue $\lambda_X$ is  $\varepsilon_X(x)$ and there is a formula for its \textit{multiplicity} $m_X$. Applying Brown's result to the special case $X = \hat{0}_{\mathcal{L}}$ and using \cref{eqn:derangement-mobius-inversion}, we obtain
\begin{align}\label{eqn:mult-0-eigenvalue}
    m_{\hat{0}_{\mathcal{L}}}  = \sum_{Y \in \mathcal{L}} \mu(\hat{0}_{\mathcal{L}}, Y) \cdot \#\mathcal{F}([X, \hat{1}_{\mathcal{L}}]) = d_\mathcal{L}.
\end{align}

Importantly, $d_{\mathcal{L}}$ can also be viewed as the dimension of the ``maximal'' invariant Peirce component. By \cref{lem:rep-structure-of-cols} (applied to $E_{\hat{1}_{\mathcal{L}}} = E_{[\hat{1}_{\mathcal{L}}]}$) the space of complete flags of $\mathcal{L}$ is precisely $\kk S(\mathcal{L})\cdot E_{\hat{1}_{\mathcal{L}}}$. However, $m_{\hat{0}_{\mathcal{L}}}$ can be interpreted as the composition multiplicity of the $\kk S(\mathcal{L})$-simple $M_{\hat{0}_{\mathcal{L}}}$ in the space of complete flags:
\begin{align}\label{eqn:derangement-in-terms-of-invariant-Peirce}
d_{\mathcal{L}} = \dim_{\kk} E_{\hat{0}_\mathcal{L}}\cdot \kk S(\mathcal{L}) \cdot E_{\hat{1}_{\mathcal{L}}}.
\end{align}

\subsubsection{The classical derangement representation}\label{subsubsec:der-rep} In \cite{FrenchDesarmenienWachs}, Desarm\'{e}nien--Wachs introduced (through the language of symmetric functions) a symmetric group representation which is now known as the \textit{derangement representation}. This representation turns out to have interesting connections with the complex of injective words \cite{ReinerWebb}, random-to-random shuffling \cite{Uyemura-Reyes, DIEKER2018427, lafreniere, reiner2014spectra}, and configuration spaces \cite{HershReiner}. 

See \cite[Proposition 3.1]{BraunerComminsReiner} for a collection of various definitions and formulas for the derangment representation. Proposition 3.1(B, C) of \cite{BraunerComminsReiner} categorify the enumerative formulas \cref{eqn:recursive} and \cref{eqn:Mobius} for derangements, respectively. Our work (see \cref{prop:derangement-rep-positive}(3)) will imply a categorification of the manifestly positive formula \cref{eqn:derangement-positive}.

 In \cite{BraunerComminsReiner},  it is essentially shown that that the ``maximal'' invariant Peirce component $E_{\emptyset} \cdot \kk \free_n \cdot E_{\{1, 2, \ldots, n\}}$ for the free LRB (or equivalently for the LRB $S(B_n)$ associated to the \textit{Boolean lattice}) carries the derangement representation.
Although the language in \cite{BraunerComminsReiner} is not in terms of invariant Peirce components, this is implicitly what is being studied -- a direct explanation of the translation is given in \cite[Corollary 7.2.4]{ComminsThesis}. We omit the proof here since our more general theory will recover this automatically by comparing the recursive formulation of the derangement representation in \cite[Proposition 3.1(B)] {BraunerComminsReiner} with  \cref{prop:derangement-rep-positive}(2)). 
\subsubsection{Generalized derangement representations} 
Given \cref{eqn:derangement-in-terms-of-invariant-Peirce} and the discussion in \cref{subsubsec:der-rep}, it is natural to define a \textit{generalized derangement representation} $\mathrm{Der}(\mathcal{L})$ of $G$ as the maximal invariant Peirce component
  \[
    \mathrm{Der}(\mathcal{L}) := E_{\hat{0}_\mathcal{L}} \cdot \kk S(\mathcal{L}) \cdot E_{\hat{1}_{\mathcal{L}}}. 
    \] If we take the derangement representation of an {interval} {within} $\mathcal{L}$, i.e. the interval $[\hat{0}_\mathcal{L}, X]$, we assume that the group changes appropriately, i.e. $\mathrm{Der}([\hat{0}_\mathcal{L}, X])$ is a $G_X$-representation.
    
We shall see in \cref{sub:randomtotop} that the generalized derangement representations are the kernels of a ``generalized random-to-top'' operator on the complete flags of $\mathcal{L}$.  

The goal for the rest of this subsection is to prove \cref{intro:thmE} from \cref{subsec:intro-derangement-reps}, which categorifies \eqref{eqn:Browns-def}-\eqref{eqn:manifestly-positive}. We first prove \cref{intro:thmE}(3) as an easy consequence of \cref{intro:thmD} and \cref{intro:thmE}(1) using a proposition.  \cref{intro:thmE}(2) will use the equivariant M\"obius inversion theory of~\cite{AssafSpeyer}. 

\subsubsection*{Proving \cref{intro:thmE}(3) and \cref{intro:thmE}(1).} Let $X >_{\mathcal{L}} Y$ be a pair from $\mathcal L$. Recall from \cref{intro:thmE}(3) the $(G_X \cap G_Y)$-representation 
    \begin{align*}
        V_{X,Y}:=\kk \{Z \in \mathcal{L}: Y \lessdot_{\mathcal{L}} Z \leq_{\mathcal{L}} X\} - \mathbb{1}.
    \end{align*}

\begin{lemma}\label{lem:lil-0-homologies-for}
    Let $B = S(\mathcal{L})$, let $X >_{\mathcal{L}} Y$ be a pair of flats in $\mathcal{L}$, and let $b_Y \in B$ be an element with support $Y$. For any subgroup $H$ of $G_{X}\cap G_{Y}$, as $H$-representations,
\begin{align*}
    \widetilde{H}_0 \left(\widetilde{B^{<b_{Y}}_{\geq X}}\right) \cong V_{X,Y}.
\end{align*}
\end{lemma}

\begin{proof}
As a representation the reduced homology is the permutation representation on the path-connected components minus the trivial representation.
    Let the rank of $Y$ in $\mathcal{L}$ be $j.$ By definition of $S(\mathcal{L})$, the element $b_{Y}$ is some flag in $\mathcal{L}$ of the form $b_{Y} = (Z_0, Z_1, \ldots, Z_{j - 1}, Y)$.
    Since the Hasse diagram of $B_{\geq X}$ is a tree (\cref{prop:matroids-are-rooted-trees}), the path components of $B_{\geq X}^{<b_{Y}}$ are those of the subtrees of $B_{\geq X}^{<b_{Y}}$ whose roots are the flags of the form  $R_Z:= (Z_0, Z_1, \ldots, Z_{j - 1}, Y, Z)$ for $Y\lessdot_{\mathcal L} Z\leq_{\mathcal L}X$.
    Finally, observe that the $\ast$-action of any $g \in G_X \cap G_Y$ on $B_{\geq X}^{<b_Y}$ sends the subtree with root $R_{Z}$ to the subtree with root $R_{g(Z)}.$
\end{proof}
\begin{prop}\label{p:der.up}
Let $X\in \mathcal L$ and let $F_X=\sum_{Y\geq_{\mathcal L} X}E_Y$. Then $F_X\cdot \kk S(\mathcal L)=F_X\cdot \kk S(\mathcal L)\cdot F_X\cong \kk S([X,\hat 1_{\mathcal L}])$ via a $G_X$-equivariant isomorphism which restricts to a $G_X$-isomorphism $E_X\cdot \kk S(\mathcal L)\cdot E_{\hat 1_{\mathcal L}}\cong \mathrm{Der}([X,\hat 1_{\mathcal L}])$.
\end{prop}
\begin{proof}
Let $\rk X=k$ and fix $x=(X_1,\ldots, X_k)$ a complete flag with $X_k=X$. Then one easily checks that $S(\mathcal L)^{\leq x}\cong S([X,\hat 1_{\mathcal L}])$ via $(X_1,\ldots, X, X_{k+1}\cdots X_m)\mapsto (X_{k+1},\ldots,X_m)$. Moreover, this isomorphism is $G_X$-equivariant for the $\ast$-action. Indeed, if $g\in G_X$, then 
\begin{align*}
g\ast (X_1,\ldots, X,X_{k+1},\ldots, X_m) & = (X_1,\ldots, X)(g(X_1),\ldots, X,g(X_{k+1}),\ldots, g(X_m)) \\ & = (X_1,\ldots, X,g(X_{k+1}),\ldots, g(X_m))
\end{align*}

  Note that $G_X$ fixes $F_X$. 
Moreover, $F_X\cdot \kk S(\mathcal L)\cdot F_X=F_X \cdot \kk S(\mathcal L)\cong \kk S(\mathcal L)^{\leq x}\cong \kk S([X,\hat 1_{\mathcal L}])$ by \cref{p:deletions}.  The first isomorphism sends $F_Xb$ to $xb$.  This is $G_X$-equivariant as $g(F_Xb) = F_Xg(b)\mapsto xg(b) = xg(x)g(b) = xg(xb) = g\ast xb$ since $x,g(x)$ both have support $X$.

Notice that all these isomorphisms are all over the projection to $\kk [X,\hat 1_{\mathcal L}]=\sigma(F_X)\kk\supp (B)$ and $E_X,E_{\hat 1_{\mathcal L}}\in F_X\cdot \kk S(\mathcal L)$.  It follows that $E_X\cdot \kk S(\mathcal L)\cdot E_{\hat 1_{\mathcal L}}$ is sent to an isomorphic copy of $\mathrm{Der}([X,\hat 1_{\mathcal L}])$.
\end{proof}

We now prove parts (1) and (3) of \cref{intro:thmE}.  
Recall that $\kk \mathcal{F}(\mathcal{L})$ is $\kk$-vector space spanned by the complete flags of $\mathcal{L}.$ 

\begin{cor} \label{cor:derangement-rep-genera-matroids}
Let $\mathcal{L}$ be a geometric lattice. Then, as $G$-representations
\begin{enumerate}
    \item Part (3) of \cref{intro:thmE}: \begin{align*}
      \mathrm{Der}(\mathcal{L}) \cong   \bigoplus_{\substack{m \geq 1, \\G\text{\rm -orbits}\\ [\hat{1}_{\mathcal{L}}= X_0 >_{\mathcal{L}}  \cdots >_{\mathcal{L}} X_m = \hat{0}_{\mathcal{L}}]}}  \left({\bigboxtimes_{i = 0}^{m-1}} V_{X_i, X_{i + 1}}\right) \Bigg \uparrow_{\bigcap_i G_{X_i}}^{G}.
    \end{align*}  
    \item Part (1) of \cref{intro:thmE}:\begin{align*}
        \kk \mathcal{F}(\mathcal{L}) \cong \bigoplus_{[X] \in \mathcal{L} / G} \mathrm{Der}([X, \hat{1}_{\mathcal{L}}] )\Big \uparrow_{G_X}^G.
    \end{align*}
\end{enumerate}
\end{cor}
\begin{proof}
(1) follows from combining \cref{intro:thmD} (\cref{thm:thmD-restated}) with \cref{lem:lil-0-homologies-for}. 

For (2), note that  \cref{lem:rep-structure-of-cols} implies that as $G$-representations,
    $\kk \mathcal{F}(\mathcal{L}) \cong \kk S(\mathcal{L}) \  \cdot E_{[\hat{1}_{\mathcal{L}}]}.$
Hence,
\begin{align*}
\kk S(\mathcal{L}) \cdot  E_{[\hat{1}_{\mathcal{L}}]} \cong \bigoplus_{[X] \in \mathcal{L} / G}E_{[X]} \cdot \kk S(\mathcal{L}) \cdot E_{\hat{1}_{\mathcal{L}}}
\end{align*}
and it is sufficient to show that as $G$-representations, $ E_{[X]} \cdot \kk S(\mathcal{L}) \cdot E_{\hat{1}_{\mathcal{L}}} \cong \mathrm{Der}([X, \hat{1}_{\mathcal{L}}]) \Big \uparrow_{G_X}^G$.   Indeed, we have
\[E_{[X]}\cdot \kk S(\mathcal L)\cdot E_{\hat 1_{\mathcal L}}\cong (E_X\cdot \kk S(\mathcal L)\cdot E_{\hat 1_{\mathcal L}})\big\uparrow_{G_X}^G\cong \mathrm{Der}([X, \hat{1}_{\mathcal{L}}])\big\uparrow_{G_X}^G\] by \cref{p:der.up}
\end{proof}

\begin{remark}\label{rmk.useful.one}\rm
Note that as $G_Y$-representations, 
\begin{align}\label{eq:der.eq}
\kk\mathcal F([Y,\hat 1_{\mathcal L}])\cong \bigoplus_{[X]\geq_{\mathcal L} [Y]}\mathrm{Der}([X,\hat 1_{\mathcal L / G}])\Big\uparrow_{G_X\cap G_Y}^{G_X}
\end{align}
by \cref{p:der.up} and \cref{cor:derangement-rep-genera-matroids} applied to $[X,\hat 1_{\mathcal L}]$.
\end{remark}

\subsubsection{Proving \cref{intro:thmE}(2)}
It is left to prove \cref{intro:thmE}(2).

We shall need  the following equivariant M\"obius inversion result of Assaf and Speyer~\cite[Theorem~5.1]{AssafSpeyer}. We have included a simpler proof for completeness, and because the original paper has an easily fixable sign error in the statement\footnote{It is corrected by replacing $(-1)^{j + 1}$ with $(-1)^j$ in their definition of $\mathfrak{m}_{\mathrm{p}}(p).$ The error in their proof is localized to the very last step; the last sentence should have $-\mathrm{Tr}_g(U_{\hat{0}})$ instead of $\mathrm{Tr}_g(U_{\hat{0}})$}.  If $P$ is a $G$-poset with minimum $\hat{0}$, define the virtual $G_p$-representation
\[\mathfrak m_{\mathrm{eq}}(p) := \sum_{j\geq -2}(-1)^j\widetilde H_{j}(\hat 0,p).\]

\begin{thm}[Assaf-Speyer~{\cite{AssafSpeyer}}]\label{t:ASm}
Let $\kk$ be a field of characteristic $0$, let $G$ act on a poset $P$ with minimum $\hat 0$, and let $V=\bigoplus_{p\in P}U_p$ where the $U_p$ are vector spaces such that $g(U_p)=U_{g(p)}$ for $p\in P$.  Define $V_p = \bigoplus_{q\geq p}U_q$.    Then as virtual representations of $G$, 
\begin{equation}\label{eqn:mobius.inv}
U_{\hat 0}= \sum_{[p]\in P/G}\left(\mathfrak m_{\mathrm{eq}}(p)\otimes V_p\right)\bigg\uparrow_{G_p}^G=\sum_{j\geq -2}(-1)^j\left(\bigoplus_{[p]\in P/G}\widetilde H_j(\hat 0,p)\otimes V_p\right)\Bigg\uparrow_{G_p}^G.
\end{equation}
\end{thm}
\begin{proof}
Expanding the induction tells us that in the representation ring the right hand side of \cref{eqn:mobius.inv} is 
\begin{equation}
    \begin{split}\label{eq:assafspey}
        \sum_{j\geq -2}(-1)^j \left(\bigoplus_{p\in P} \widetilde H_j(\hat 0,p)\otimes V_p\right)&= \sum_{j\geq -2}(-1)^j\left(\bigoplus_{p\in P}\bigoplus_{q\geq p}\widetilde H_j(\hat 0,p)\otimes U_q\right)\\ &=
        \sum_{j\geq -2}(-1)^j\left(\bigoplus_{q\in P}\bigoplus_{p\leq q}\widetilde H_j(\hat 0,p)\otimes U_q\right).
    \end{split}
\end{equation}

We now compute the value the character of the right hand side of \cref{eq:assafspey} on $g\in G$.  Let $P^g$ denote the subposet of $P$ consisting of those elements fixed under the action of $g.$ 
First note that $g$ preserves the summand $\widetilde H_j(\hat 0,p)\otimes U_q$ if and only if $p,q\in P^g$.   Writing $\mathrm{Tr}_g(U_q)$ for the trace of $g$ on $U_q$, noting that $\Delta(P)^g=\Delta(P^g)$, and  following our usual conventions from poset homology in \cref{sec:poset-topology-conventions} (e.g., \cref{rem:empty-interval}), an application of the Hopf trace theorem and 
P.~Hall's theorem yields that the character 
is
\begin{align*}
\sum_{q\in P^g}\mathrm{Tr}_g(U_q)\sum_{p\leq q, p\in P^g}\widetilde{\chi}(\Delta((\hat 0,p)^g)) = \sum_{q\in P^g}\mathrm{Tr}_g(U_q)\sum_{p\leq q, p\in P^g}\mu_{P^g}(\hat 0,p) = \mathrm{Tr}_g(U_{\hat 0})
\end{align*}
 using that the sums $\sum_{p \leq q}\mu_{P^G}(p, q)$ vanish for $q \neq \hat{0}$ by the definition of M\"{o}bius function.
\end{proof}

In order to prove Part (2) of \cref{intro:thmE} in positive characteristic, we need the following lemma from modular representation theory.  Recall that an $RG$-lattice is an $RG$-module that is projective as an $R$-module. For a prime $p$, we write $\mathbb{Z}_{(p)}$ for the localization of $\mathbb{Z}$ at $p.$

\begin{lemma}\label{lem:brauer.char}
Let $G$ be a finite group and $p$ a prime such that $p\nmid |G|$.  Suppose that $M,N$ are $\mathbb Z_{(p)}G$-lattices such that $\mathbb Q\otimes_{\mathbb Z_{(p)}}M\cong \mathbb Q\otimes_{\mathbb Z_{(p)}} N$.  Then $\kk\otimes_{\mathbb Z_{(p)}}M\cong \mathbb \kk \otimes_{\mathbb Z_{(p)}}N$ for any field $\kk$ of characteristic $p$.
\end{lemma}
\begin{proof}
It's enough to handle that case that $\kk = \mathbb F_p$ by extension of scalars. By the Brauer-Nesbitt theorem~\cite[Theorem~9.4.8]{webbrepntheory} $F_p\otimes_{\mathbb Z_{(p)}}M\cong M/pM$ and $\mathbb F_p\otimes_{\mathbb Z_{(p)}}N\cong N/pN$ have the same composition factors, and hence, since $p\nmid |G|$, they are isomorphic by Maschke's theorem.
\end{proof}

\begin{prop}[Part (2) of \cref{intro:thmE}] \label{prop:mobius-rep}
As virtual $G$-representations, 
    \begin{align*}
        \mathrm{Der}(\mathcal{L}) \cong \bigoplus_{[Y] \in \mathcal{L} / G} (-1)^{\rk(Y) - 2} \left( \widetilde{H}_{\rk(Y) - 2}(\hat{0}_{\mathcal{L}}, Y) \otimes \kk\mathcal{F}([Y, \hat{1}_{\mathcal{L}}]) \right)\Big \uparrow_{G_Y}^G.
    \end{align*}
\end{prop}
\begin{proof}
We first prove the claim in the case that $\kk$ has characteristic $0$ using Theorem~\ref{t:ASm}. 
Fix $Y >_\mathcal{L} \hat{0}_\mathcal{L}.$ Since $\mathcal{L}$ is a geometric lattice, so is the closed interval $[\hat{0}_\mathcal{L}, Y]$ for all $Y \in \mathcal{L}.$  Folkman's theorem \cite{Folkman} then gives that the order complex of $(\hat{0}_{\mathcal{L}}, Y)$ is a wedge of spheres of dimension $\rk(Y) - 2$, and thus the homology of $(\hat{0}_{\mathcal{L}}, Y)$ is nonzero only in degree $\rk(Y) - 2.$ Therefore, by \cref{t:ASm} 
\[\mathfrak{m}_{\mathrm{eq}}(Y) = \bigoplus_{j \geq -2} (-1)^j\widetilde{H}_j (\hat{0}_\mathcal{L}, Y) = (-1)^{\rk(Y) - 2} \widetilde{H}_{\rk(Y) - 2}(\hat{0}_\mathcal{L}, Y)\]
(and also for $Y=\hat 0$).

For each $X \in \mathcal{L}$, set $U_X:=\mathrm{Der}\left([X, \hat{1}_\mathcal{L}]\right)$ 
Then, by \cref{rmk.useful.one}, \[V_Y:= \bigoplus_{X \geq_\mathcal{L} Y} U_X= \bigoplus_{X \geq_\mathcal{L} Y} \mathrm{Der}\left([X, \hat{1}_\mathcal{L}]\right) \cong \kk \mathcal{F}\left([Y, \hat{1}_\mathcal{L}]\right)\] as $G_Y$-representations. Hence, by Theorem~\ref{t:ASm}, as virtual $G$-representations, \[\mathrm{Der}(\mathcal{L}) \cong \bigoplus_{[Y] \in \mathcal{L} / G} (-1)^{\rk(Y) - 2} \left(\widetilde{H}_{\rk(Y) - 2}(\hat{0}, Y) \otimes \kk \mathcal{F}([Y, \hat{1}_\mathcal{L}])\right) \Bigg \uparrow_{G_Y}^G,\] as required. 

    To prove the general case (for positive characteristic) we use \cref{lem:brauer.char}.  If we construct our cfpoi for $\mathbb QS(\mathcal L)$ as per \cref{prop:permuted-idempotents} then  $\{E_X:X\in \mathcal L\}$ is defined over $\mathbb Z_{(p)}$, where $p$ is the characteristic of $\kk$, and projects to a $G$-invariant cpfoi for $\mathbb F_pS(\mathcal L)$. Then $E_{\hat 0_{\mathcal L}}\cdot \mathbb Z_{(p)}S(\mathcal L)\cdot E_{\hat 1}$ is a $\mathbb Z_{(p)}G$-lattice, being a $\mathbb Z_{(p)}$-direct summand in $\mathbb Z_{(p)}S(\mathcal L)$.  Tensoring with $\kk$ yields $E_Y\cdot \mathbb \kk S(\mathcal L)\cdot E_{\hat 1_{\mathcal L}}$ and with $\mathbb Q$ yields $E_{\hat 0_{\mathcal L}}\cdot \mathbb QS(\mathcal L)\cdot E_{\hat 1_{\mathcal L}}$.
    
    Clearly, $\mathbb Z_{(p)}\mathcal F([Y,\hat 1_{\mathcal L}])\Big\uparrow_{G_Y}^G$ tensors over $\kk$ to  $\kk \mathcal F([Y,\hat 1_{\mathcal L}])\Big\uparrow_{G_Y}^G$ and over $\mathbb Q$ to $\mathbb Q \mathcal F([Y,\hat 1_{\mathcal L}])\Big\uparrow_{G_Y}^G$. On the other hand, by Folkman's theorem~\cite{Folkman}, $\Delta([\hat 0_{\mathcal L},Y])$ is a wedge of $(\rk(Y)-2)$-spheres, and hence has free abelian homology. So by the universal coefficient theorem, we have 
    \begin{align*}
   \widetilde{H}_{\rk(Y)-2}((\hat 0_{\mathcal L},Y);\mathbb Z_{(p)})\big\uparrow_{G_Y}^G&\cong \mathbb Z_{(p)}\otimes_{\mathbb Z} \widetilde{H}_{\rk(Y)-2}((\hat 0_{\mathcal L},Y);\mathbb Z)\big\uparrow_{G_X}^G\\
   \widetilde{H}_{\rk(Y)-2}((\hat 0_{\mathcal L},Y);\mathbb Q)\big\uparrow_{G_Y}^G&\cong \mathbb Q\otimes_{\mathbb Z}\widetilde{H}_{\rk(Y)-2}((\hat 0_{\mathcal L},Y);\mathbb Z)\big\uparrow_{G_Y}^{G} \\
  \widetilde{H}_{\rk(Y)-2}((\hat 0_{\mathcal L},Y);\mathbb \kk)\big\uparrow_{G_X}^G & \cong  \kk \otimes_{\mathbb Z} \widetilde{H}_{\rk(Y)-2}((\hat 0_{\mathcal L},Y);\mathbb Z)\big\uparrow_{G_X}^G 
   \end{align*}
The desired conclusion now follows from the characteristic zero case and \cref{lem:brauer.char}.    
\end{proof}

\subsubsection{Example of \cref{intro:thmE}: formulas for the classical derangement representation}\label{subsec:Examples-thmE}

As an example, we apply \cref{intro:thmE} to the free LRB $\free_n$, which is the LRB $S(B_n)$ arising from the Boolean lattice $B_n$. It is straightforward to carry out similar computations for $\pg(n-1, q)$ and $\ag(n, q)$, but we omit them for space considerations. 
We first set some notation and review the relevant representation theory of the symmetric group $S_n$. For a composition $\alpha = (\alpha_1, \ldots, \alpha_k)\vDash n,$ we view \textbf{$S_{\alpha} := S_{\alpha_1} \times S_{\alpha_2} \times \ldots \times S_{\alpha_{\ell(\alpha)}}$} as the subgroup of $S_n$ permuting the first $\alpha_1$ positive integers amongst themselves, the next $\alpha_2$ amongst themselves, and so on. 
For each integer $0 \leq m \leq \ell(\alpha),$ define the partial sum
\begin{align}
    Q_m(\alpha)&:= \sum_{k = 1}^m \alpha_k.
\end{align}

 Fix a complete flag $X_0 \lessdot_{B_n} X_1 \lessdot_{B_n} \cdots \lessdot_{B_n} X_n$ in $B_n$ to be $\emptyset \lessdot_{B_n} \{1\} \lessdot_{B_n} \{1, 2\} \lessdot_{B_n} \{1, 2, 3\} \lessdot_{B_n} \cdots \lessdot_{B_n} \{1, 2, \ldots, n\}.$
For $\alpha \vDash n$, observe that  \[\bigcap_{m = 0}^{\ell(\alpha)} \stab_G(X_{Q_m(\alpha)}) = S_\alpha.\]

Recall that the irreducible representations of the symmetric group are indexed by partitions $\lambda$ of $n$, written $S^\lambda.$ The sign representation
corresponds to $\lambda = 1^n$. Setting $\lambda = {(n-1, 1)}$ 
gives the reflection representation; namely,
\[S^{(n-1, 1)} \cong \mathbb{1} \big \uparrow_{S_{(n - 1, 1)}}^{S_n} - \mathbb{1} \cong \kk\{1, 2, \cdots, n\} - \mathbb{1}.\]

 By specializing \cref{intro:thmE} to this setting, we obtain the following. The first two categorifications below are already well-known (see, for example, \cite[Prop. 9]{FrenchDesarmenienWachs}, \cite[Prop. 2.2]{ReinerWebb}, or \cite[Prop 3.1 (B, C)]{BraunerComminsReiner}).

\begin{cor} \label{prop:derangement-rep-positive}
Assume  the characteristic of $\kk$ does not divide $n!.$ As $S_n$-representations,
    \begin{align*}
       \kk S_n&\cong \bigoplus_{k = 0}^n \left(\mathrm{Der}(B_k) \otimes \mathbb{1}\right)\Big \uparrow_{S_k \times S_{n - k}}^{S_n},\\
    \mathrm{Der}(B_n) &\cong \bigoplus_{k = 0}^n (-1)^k \left(S^{(1^k)} \otimes \kk S_{n-k}  \right)\Big \uparrow_{S_k \times S_{n-k}}^{S_n},\text{ and }\\
        \mathrm{Der}(B_ n) &\cong \bigoplus_{\alpha \vDash n} \left(\bigotimes_{i = 1}^{\ell(\alpha)}S^{(\alpha_i - 1, 1)}\right)  \Big \uparrow_{S_\alpha}^{S_n}.
    \end{align*}
\end{cor}

\begin{proof}
The first isomorphism follows from \cref{intro:thmE}(1) by observing that $[\emptyset, X_k] = B_k$ and that the flag space $\kk \mathcal{F}(B_n)$ carries the regular representation $\kk S_n.$

By the same facts, and since $\widetilde{H}^{k - 2}(B_k)$ carries the sign character of $S_k$ by \cref{lem:general-boolean},  the second isomorphism follows from \cref{intro:thmE}(2).

For the third isomorphism, by \cref{intro:thmE}(3), it suffices to show that as $S_{\alpha}$-representations, \begin{align*}
  \bigotimes_{i = 1}^{\ell(\alpha)}V_{X_{Q_i(\alpha)}, X_{Q_{i - 1}(\alpha)}}&\cong \bigotimes_{i = 1}^{\ell(\alpha)}S^{(\alpha_i - 1, 1)}.
\end{align*}

Fix $1 \leq m \leq \ell(\alpha).$ Observe that as $S_{\alpha_m}$-representations, \[V_{X_{Q_m(\alpha)}, X_{Q_{m - 1}(\alpha)}} \cong \kk \{i: Q_{m-1}(\alpha) < i \leq Q_m(\alpha)\} - \mathbb{1} \cong S^{(\alpha_{m} - 1, 1)}.\] Thus, as $S_\alpha $-representations,
\begin{align*}
  V_{X_{Q_m(\alpha)}, X_{Q_{m - 1}(\alpha)}}
   &\cong \underbrace{\mathbb{1}\otimes \mathbb{1} \otimes \cdots \otimes \mathbb{1}}_{m - 1 \text{ times}} \otimes S^{(\alpha_{m} - 1, 1)}\normalcolor \otimes \underbrace{\mathbb{1}\otimes \cdots \otimes \mathbb{1}}_{\ell(\alpha) - m}\text{ times}.
\end{align*}
The desired isomorphism then follows by taking the inner tensor product as $m$ varies.
\end{proof}

\subsection{Connections to eigenspaces of generalized random-to-top operators}\label{sub:randomtotop}
We keep our assumptions that $\mathcal{L}$ is a finite geometric lattice and that $G$ is a finite group acting on $\mathcal{L}$ by automorphisms. However, we now restrict our field $\kk$ and assume that $\Char(\kk) = 0.$

Consider the element $x \in (\kk S(\mathcal{L}))^G$ given by summing the length-one elements of $S(\mathcal{L})$, \[x:= \sum_{Y \gtrdot_\mathcal{L} \hat{0}_\mathcal{L}}(Y).\]
This element acts on the space of complete flags $\kk \mathcal{F}(\mathcal{L})$ by left multiplication and its action can be thought of as a generalized \textbf{random-to-top} operator. The classical {random-to-top} operator\footnote{The classical random-to-top operator is sometimes normalized by $\frac{1}{n}$ so that it becomes a Markov chain.} is the linear map on $\kk S_n$ sending a permutation $w = w_1 w_2 \cdots w_n$ (in one-line notation) to the sum \[w \mapsto \sum_{i = 1}^n \color{red}{w_i}\normalcolor  w_1 w_2 \cdots w_{i - 1}w_{i + 1}\cdots w_n.\] It is straightforward to check that when $\mathcal{L} = B_n$, the space of complete flags $\kk \mathcal{F}(B_n)$ can be identified with $\kk S_n$ and the action of $x$ on $\kk S_n$ under this identification coincides with the classical random-to-top operator.

Brown's theory \cite{BrownonLRBs} explains the eigenvalues of the action of $x$ on $\kk \mathcal{F}(\mathcal{L})$ for any geometric lattice $\mathcal{L}$. In particular, since $G$ fixes $x$, the eigenvalues of $x$ are indexed by $G$-orbits $[X] \in \supp(B) / G$ with \[\lambda_{[X]} = \# \{Y: \hat{0}_\mathcal{L} \lessdot_\mathcal{L} Y \leq_\mathcal{L} X \}.\]
Further, his theory explains that the operator $x$ is diagonalizable, and it describes the dimensions of the eigenspaces of $x$ in terms of the generalized derangement numbers. Namely, the dimension of the $\lambda_{[X]}$-eigenspace\footnote{It is possible that $\lambda_{[X]} = \lambda_{[Y]}$ for $[X] \neq [Y]$; one can further lump these eigenspaces if this is the case.} is $\#\{X' \in [X]\}\cdot d_{[X, \hat{1}_{\mathcal{L}}]} = [G: G_X] \cdot d_{[X, \hat{1}_\mathcal{L}]}$.

Since $x$ is invariant under the action of $G$, its eigenspaces are $G$-representations. The $S_n$-representation structures of the classical random-to-top operator eigenspaces are well-studied \cite{Uyemura-Reyes, reiner2014spectra, BraunerComminsReiner} and are known to have connections to the classical derangement representation. A similar analysis for the $GL_n(\mathbb{F}_q)$-structure of the eigenspaces of $x$ when $\mathcal{L} = \mathrm{PG}(n-1, q)$ was carried out in \cite{BraunerComminsReiner}.

The following proposition refines the work of Brown in \cite{BrownonLRBs} and generalizes the work in \cite{reiner2014spectra,Uyemura-Reyes,BraunerComminsReiner} by describing the $G$-representation structure on the eigenspaces of $x$ for any geometric lattice $\mathcal{L}$. Our description is in terms of the generalized derangement representations.

\begin{prop}
   Fix an orbit $[X] \in \mathcal{L} / G.$ The $\lambda_{[X]}$-eigenspace for the action of $x$ on the space of complete flags  $\kk \mathcal{F}(\mathcal{L})$ carries the $G$-representation $\mathrm{Der}([X, \hat{1}_\mathcal{L}]) \big \uparrow_{G_X}^G$ of dimension $[G:G_X]\cdot d_{[X, \hat{1}_\mathcal{L}]}$. Hence $\ker x=\mathrm{Der}(\mathcal L)$.
\end{prop}
\begin{proof}
    We use Saliola's work on the eigenvectors of random walks on LRBs in \cite{SaliolaEigenvectors}. For $Z \in \mathcal{L}$, define 
    \[P_Z := 
    \sum_{\hat{0}_\mathcal{L} \lessdot_\mathcal{L} Y \leq_\mathcal{L} Z} (Y).\]  
    Let $L_Z$ be the subsemigroup of partial flags ending in $Z$, or equivalently, the elements in $S(\mathcal{L})$ of support $Z$. The element $P_Z$ has a non-negative integer matrix $T_Z$ for its action on $\kk L_Z$ by left multiplication with respect to the basis given by the elements of $L_Z$. As the elements with positive weight in $P_Z$ generate $S([\hat{0}_\mathcal{L}, Z])$, the matrix $T_Z$ is irreducible and hence has a unique Perron-Frobenius eigenvector (over $\mathbb Q$) which is a probability distribution, cf. \cite[Corollary 6]{SaliolaEigenvectors}. Moreover, the Perron-Frobenius eigenvalue is $\lambda_{[Z]}$ since $T_Z$ has column sums $\lambda_{[Z]}$; see~\cite{SaliolaEigenvectors}. Set $\pi_Z$ to be the sum of the elements in $L_Z$ with coefficients given by this Perron-Frobenius eigenvector.

    We point out that $\pi_Z$ meets the requirements of the $\ell_Z$ in \cref{prop:construct-idems} in that it is a sum of elements of support $Z$ whose coefficients sum to $1$. Additionally, we claim that $g(\pi_Z) = \pi_{g(Z)}.$ This follows from the uniqueness of the Perron-Frobenius probability eigenvector for $P_{g(Z)}$  given that 
    \[P_{g(Z)}g(\pi_Z) = g(P_Z)g(\pi_Z) = g(P_Z \pi_Z) = \#\{\hat{0}_\mathcal{L} \lessdot_\mathcal{L} Y \leq_{\mathcal{L}} Z\} \cdot g(\pi_Z) = \#\{\hat{0}_\mathcal{L} \lessdot_\mathcal{L} Y \leq_{\mathcal{L}} g(Z)\}\cdot  g(\pi_Z).\]
    Hence, the cfpoi $\{F_X: X \in \supp(B)\}$ for $\kk B$ formed from setting $\ell_Z = \pi_Z$ is a cfpoi which is invariant with respect to $G.$
    
    Now, Corollary 5 in Saliola's work \cite{SaliolaEigenvectors} implies that the $\lambda_{[X]}$-eigenspace of $x$ is \[ F_{X} \cdot \kk \mathcal{F}(\mathcal{L}) \Big \uparrow_{G_X}^G.\] Finally, using \cref{lem:rep-structure-of-cols} and \cref{p:der.up}, we can rewrite the eigenspace as\[ F_{X} \cdot \kk S(\mathcal{L}) \cdot F_{\hat{1}_\mathcal{L}} \big \uparrow_{G_X}^G \cong \mathrm{Der}([X, \hat{1}_\mathcal{L}]) \big \uparrow_{G_X}^G, \] as desired. The final statement follows from \cref{cor:derangement-rep-genera-matroids}(2), and since $\lambda_{[X]}=0$ if and only if $X=\hat 0_{\mathcal L}$.
\end{proof}

\bibliographystyle{acm} 
\bibliography{bibliography}

\end{document}